\documentclass[preprint, 12pt]{elsarticle}

\usepackage{amsmath,amsfonts}
\usepackage{algorithmic}
\usepackage{algorithm}
\usepackage{array}
\usepackage[caption=false,font=normalsize,labelfont=sf,textfont=sf]{subfig}
\usepackage{textcomp}
\usepackage{stfloats}
\usepackage{url}
\usepackage{verbatim}
\usepackage{graphicx}
\usepackage{ntheorem}
\theoremseparator{.}
\usepackage{multirow}
\usepackage{float}

\newtheorem{assumption}{Assumption}
\newtheorem{proposition}{Proposition}
\newtheorem{theorem}{Theorem}
\newtheorem{lemma}{Lemma}
\newtheorem{definition}{Definition}
\newtheorem{remark}{Remark}
\newtheorem{proof}{Proof}
\newtheorem{example}{Example}

\begin{document}

\begin{frontmatter}

\title{A Perturbed Value-Function-Based Interior-Point Method for Perturbed Pessimistic Bilevel Problems}

\author[1]{Haimei Huo}
\ead{ab1234@mail.dlut.edu.cn}

\author[2]{Risheng Liu \corref{cor1}}
\ead{rsliu@dlut.edu.cn}

\author[3]{Zhixun Su \corref{cor1}}
\ead{zxsu@dlut.edu.cn}

\cortext[cor1]{Corresponding author}

\affiliation[1]{organization={School of Mathematical Sciences},
            addressline={Dalian University of Technology}, 
            city={Dalian},                    
            country={China}}
\affiliation[2]{organization={DUT-RU International School of Information Science and Engineering},
            addressline={Dalian University of Technology}, 
            city={Dalian},                    
            country={China}}

\affiliation[3]{organization={School of Mathematical Sciences},
            addressline={Dalian University of Technology}, 
            city={Dalian},                    
            country={China}}

\begin{abstract}
Bilevel optimizaiton serves as a powerful tool for many machine learning applications. Perturbed pessimistic bilevel problem PBP$\epsilon$, with $\epsilon$ being an arbitrary positive number, is a variant of the bilevel problem to deal with the case where there are multiple solutions in the lower level problem. However, the provably convergent algorithms for PBP$\epsilon$ with a nonlinear lower level problem are lacking. To fill the gap, we consider in the paper the problem PBP$\epsilon$ with a nonlinear lower level problem. By introducing a log-barrier function to replace the inequality constraint associated with the value function of the lower level problem, and approximating this value function, an algorithm named Perturbed Value-Function-based Interior-point Method(PVFIM) is proposed. We present a stationary condition for PBP$\epsilon$, which has not been given before, and we show that PVFIM can converge to a stationary point of PBP$\epsilon$. Finally, experiments are presented to verify the theoretical results and to show the application of the algorithm to GAN.
\end{abstract}


\end{frontmatter}

\section{Introduction}\label{section:1}

Bilevel optimization has received significant attention in many machine learning applications including hyperparameter optimization\cite{1, 2, 3, 4, 5}, meta learning\cite{6}, neural architecture search\cite{7}, and Generative Adversarial Networks(GAN)\cite{8,9}. Mathematically, bilevel optimization can be described as follows:
\begin{align}   \label{eq1}
&\min\limits_{\boldsymbol{x}\in \mathcal{X}} F(\boldsymbol{x}, \boldsymbol{y})  \\
&\text{s.t.} ~\boldsymbol{y} \in \mathcal{S}(\boldsymbol{x}) := { \underset {\boldsymbol{y}\in \mathcal{Y}(\boldsymbol{x})} { \operatorname {arg\,min} } \, f(\boldsymbol{x}, \boldsymbol{y})} \nonumber
\end{align}
where $\mathcal{X}\subset \mathbb{R}^n$, $\mathcal{Y}(\boldsymbol{x}) \subset \mathbb{R}^m$ is a set dependent on $\boldsymbol{x}$, $f(\boldsymbol{x}, \boldsymbol{y})$, $F(\boldsymbol{x}, \boldsymbol{y}):\mathcal{X} \times \mathbb{R}^m \rightarrow \mathbb{R}$ are the objective functions of the lower level(LL) problem and upper level(UL) problem, respectively, and for each $\boldsymbol{x}$, the feasible region of $F(\boldsymbol{x}, \boldsymbol{y})$ is $\mathcal{S}(\boldsymbol{x})$, which consists of optimal solutions to the LL problem. Moreover, as the bilevel problem can be can be thought of as a two-player game, the LL problem and the UL problem also can be called the follower and the leader, respectively; see \cite{10,11}.

However, the goal of the bilevel problem in (\ref{eq1}) is ambiguous if set $\mathcal{S}(\boldsymbol{x})$ is not singleon for some $\boldsymbol{x}\in \mathcal{X}$, since the minimization in the UL problem is only with respect to(w.r.t.) $\boldsymbol{x}$. To resolve it, (\ref{eq1}) is usually reformulated into the optimistic bilevel problem or the pessimistic bilevel problem depending on practical applications\cite{11,12,13,14}. To be concrete, if the LL and UL problems are cooperative\cite{13}, problem (\ref{eq1}) is generally reformulated into the optimistic bilevel problem:
\begin{align*}  
\min\limits_{\boldsymbol{x}\in \mathcal{X}} \varphi(\boldsymbol{x}), \quad \varphi(\boldsymbol{x}):= \min\limits_{\boldsymbol{y}} \{F(\boldsymbol{x}, \boldsymbol{y}): \boldsymbol{y} \in \mathcal{S}(\boldsymbol{x})\}. \nonumber
\end{align*}
If, on the other hand, the LL and UL problems are uncooperative\cite{13}, problem (\ref{eq1}) is generally reformulated into the pessimistic bilevel problem(PBP):
\begin{align}  \label{eq2}
\min\limits_{\boldsymbol{x}\in \mathcal{X}} \varphi(\boldsymbol{x}), \quad \varphi(\boldsymbol{x}):= \max\limits_{\boldsymbol{y}} \{F(\boldsymbol{x}, \boldsymbol{y}): \boldsymbol{y} \in \mathcal{S}(\boldsymbol{x})\}.
\end{align}

Since the optimistic bilevel problem is easier to address, optimality conditions and numerical algorithms have been extensively studied; e.g., \cite{10,12,13,15,16,17}. In contrast, there are few algorithms available for the pessimistic bilevel problem.

Recently, for pessimistic bilevel problems with optimal solutions, some algorithms have been developed. Specifically, a penalty method, a reducibility method, a Kth-Best algorithm, and a descent algorithm are proposed in \cite{18}, \cite{19}, \cite{20}, and \cite{21}, respectively, to compute global or local optimal solutions for linear pessimistic bilevel problems with optimal solutions. In addition, for the general pessimistic bilevel problems with optimal solutions, a maximum entropy approach, a Relaxation-and-Correction scheme, and an algorithm named BVFSM are proposed in \cite{22}, \cite{23}, and \cite{24}, respectively, to compute global optimal solutions.

However, as is known, the optimal solutions to pessimistic bilevel problems can not be guaranteed to exist generally due to the very strong assumptions for ensuring the existence  of optimal solutions\cite{14}. In this sense, it seems that problem (\ref{eq2}) is not well defined. To address the problem, researchers have proposed to replace the PBP in (\ref{eq2}) with another solvable approximation problem\cite{14,25}. As done in \cite{14,25}, replacing the optimal point set $\mathcal{S}(\boldsymbol{x})$ in (\ref{eq2}) with the set of $\epsilon$-optimal points $\mathcal{S}_{\epsilon}(\boldsymbol{x}):=\{\boldsymbol{y} \in \mathcal{Y}(\boldsymbol{x}): f(\boldsymbol{x}, \boldsymbol{y})\le f^*(\boldsymbol{x})+\epsilon \}$, where $\epsilon$ is an arbitrary positive number and $f^*(\boldsymbol{x}) := \min_{\boldsymbol{y}\in \mathcal{Y}(\boldsymbol{x})} f(\boldsymbol{x}, \boldsymbol{y})$, a perturbed pessimistic bilevel problem PBP$\epsilon$:
\begin{align}  \label{eq3}
\min\limits_{\boldsymbol{x}\in \mathcal{X}} \varphi_{\epsilon}(\boldsymbol{x}), \quad \varphi_{\epsilon}(\boldsymbol{x}):= \max\limits_{\boldsymbol{y}} \{F(\boldsymbol{x}, \boldsymbol{y}): \boldsymbol{y} \in \mathcal{S}_{\epsilon}(\boldsymbol{x})\}. 
\end{align}
can be obtained. It is proved in \cite{14} that under certain mild assumptions, PBP$\epsilon$ is solvable for any $\epsilon>0$, and that the minimum value of PBP$\epsilon$ can approximate the infimum of PBP in (\ref{eq2}) arbitrarily well as long as we choose sufficiently small $\epsilon$. Actually, problem PBP$\epsilon$ is itself a meaningful problem. For example, the leader may only want to execute a decision that performs best among all $\epsilon$-optimal solutions. After all, it is generally difficult to obtain a globally optimal solution to the LL problem.

Therefore, in this paper, we reformulate problem (\ref{eq1}) where the leader and the follower are uncooperative as the problem PBP$\epsilon$ in (\ref{eq3}) following \cite{14,25}, and consider problem PBP$\epsilon$. As far as we know, the convergence of the existing algorithms for PBP$\epsilon$ can be guaranteed only when the lower level problem is linear. Recently, for problem PBP$\epsilon$ with the linear lower level problem, \cite{14} showed the equivalence between PBP$\epsilon$ and a single-level mathematical program MPCC$\epsilon$ regarding global optimal solutions, and a mixed integer approach is proposed to solve MPCC$\epsilon$ to obtain the global optimal solution of PBP$\epsilon$. \cite{23} mentioned that the Relaxation-and-Correction scheme for solving PBP can be used to solve problem PBP$\epsilon$ to obtain global optimal solutions. However, the Relaxation-and-Correction scheme is only used to solve the linear pessimistic bilevel problem, and the algorithms for solving general pessimistic bilevel problems are not provided. 

In the paper, our aim is to develop a provably convergent algorithm for PBP$\epsilon$ in (\ref{eq3}) with nonlinear lower level problem. Firstly, we obtain the stationary condition of PBP$\epsilon$ by using the standard version of PBP$\epsilon$ proposed in \cite{14}, the value function approach, and the Fritz-John type necessary optimality condition, as usally done for optimistic bilevel problems. Then, motivated by the recent work in \cite{24} for solving pessimistic bilevel problems, we propose an algorithm named Perturbed Value-Function-based Interior-point Method(PVFIM). To be specific, we first reformulate PBP$\epsilon$ as a problem containing an inequality constraint. Later, we introduce a log-barrier term into the UL function $F(\boldsymbol{x}, \boldsymbol{y})$ to remove the inequality constraint, and use an approximate function to replace the resultant objective function. As a result, by solving a sequence of approximate minimax problems, we can obtain the solution of PBP$\epsilon$. Theoretically, we show that PVFIM can produce a sequence from which we can obtain a stationary point of PBP$\epsilon$. Our main contributions are listed below:

\begin{enumerate}
\item[(1)] We present a stationary condition of problem PBP$\epsilon$, which has not been given before.

\item[(2)] We develop an algorithm named PVFIM for problem PBP$\epsilon$ with nonlinear lower level problem, and we prove that PVFIM can converge to a stationary point of PBP$\epsilon$; we are the first to provide the provably convergent algorithms for PBP$\epsilon$ with nonlinear lower level problem.

\item[(3)] We perform experiments to validate our theoretical results and show the potential of our proposed algorithm for GAN applications.
\end{enumerate}

For the rest of this paper, we first make some settings for problem PBP$\epsilon$ in (\ref{eq3}), and review some materials on nonsmooth analysis in Section \ref{section:2}. Next, we propose a new algorithm named PVFIM for problem PBP$\epsilon$ in Section \ref{section:3}. Then, in Section \ref{section:4}, we show the convergence of our proposed algorithm. Finally, we present the experimental results in Section \ref{section:5}, and conclude and summarize the paper in Section \ref{section:6}.

\textbf{Notations.} We denote by $\|\cdot \|$ the $l_2$ norm for vectors and spectral norm for matrices. We use $\mathbb{B}_{\delta}(\boldsymbol{x})$ to denote the open ball centered at $\boldsymbol{x}$ with radius $\delta$, and for a set $\mathcal{A}$, co$\mathcal{A}$ denotes the convex hull of $\mathcal{A}$. For $f(\boldsymbol{x}, \boldsymbol{y}) : \mathbb{R}^p \times \mathbb{R}^q\rightarrow \mathbb{R}$, $\nabla f(\boldsymbol{x}, \boldsymbol{y})$ denotes the gradient of $f(\boldsymbol{x}, \boldsymbol{y})$ taken w.r.t. $(\boldsymbol{x},\boldsymbol{y})$, $\nabla_{\boldsymbol{x}} f(\boldsymbol{x}, \boldsymbol{y})$(resp. $\nabla_{\boldsymbol{y}} f(\boldsymbol{x}, \boldsymbol{y})$) denotes the gradient of $f(\cdot, \boldsymbol{y})$(resp. $f(\boldsymbol{x}, \cdot)$) w.r.t. $\boldsymbol{x}$(resp. $\boldsymbol{y}$), and $\nabla_{\boldsymbol{y}\boldsymbol{x}} f(\boldsymbol{x}, \boldsymbol{y}):= \nabla_{\boldsymbol{x}}(\nabla_{\boldsymbol{y}}f(\boldsymbol{x},\boldsymbol{y}))$ and $\nabla_{\boldsymbol{y}\boldsymbol{y}} f(\boldsymbol{x}, \boldsymbol{y}):= \nabla_{\boldsymbol{y}}(\nabla_{\boldsymbol{y}}f(\boldsymbol{x},\boldsymbol{y}))$.

\section{Preliminaries} \label{section:2}
In this section, we first make some settings for problem PBP$\epsilon$ in (\ref{eq3}). Then we review some background materials on nonsmooth analysis.

\subsection{Problem Setting} 
In this paper, for PBP$\epsilon$ in (\ref{eq3}), we set $\mathcal{Y}(\boldsymbol{x})$ to be a fixed set $\mathcal{Y}\subset \mathbb{R}^m$ for $\boldsymbol{x}\in\mathcal{X}$, that is, for any given $\epsilon > 0$, we consider problem PBP$\epsilon$ below:
\begin{equation}  \label{eq4}
\min\limits_{\boldsymbol{x}\in \mathcal{X}} \varphi_{\epsilon}(\boldsymbol{x}), \quad \varphi_{\epsilon}(\boldsymbol{x}):= \max\limits_{\boldsymbol{y}} \{F(\boldsymbol{x}, \boldsymbol{y}): \boldsymbol{y} \in \mathcal{S}_{\epsilon}(\boldsymbol{x})\}
\end{equation}
where $\mathcal{S}_{\epsilon}(\boldsymbol{x}) := \{\boldsymbol{y}\in \mathcal{Y}: f(\boldsymbol{x}, \boldsymbol{y}) \le f^*(\boldsymbol{x}) + \epsilon \}$, and 
\begin{equation}  \label{eq5}
f^*(\boldsymbol{x}) := \min_{\boldsymbol{y}} \{f(\boldsymbol{x}, \boldsymbol{y}): \boldsymbol{y} \in \mathcal{Y}\}.
\end{equation}

In the following, we first make some assumptions on the constraint sets.

\begin{assumption}\label{assum:1}
We suppose that sets $\mathcal{X} \subset \mathbb{R}^n$ and $\mathcal{Y} \subset \mathbb{R}^m$ are nonempty, convex, and closed. Furthermore, there exist positive numbers $H$ and $M>1$ such that 
\begin{equation*}
\|\boldsymbol{x}\| \le H, \qquad \|\boldsymbol{y}\| \le M 
\end{equation*} 
for all $\boldsymbol{x}\in \mathcal{X}$ and $\boldsymbol{y} \in \mathcal{Y}$.
\end{assumption}

Next, we make some assumptions on the objective functions.

\begin{assumption}\label{assum:2}
We suppose that $\mathcal{C}\subset \mathbb{R}^n$ is a set containing $\mathcal{X}$, open, and convex; functions $F(\boldsymbol{x}, \boldsymbol{y})$, $f(\boldsymbol{x}, \boldsymbol{y}): \mathcal{C}\times \mathcal{Y}\rightarrow \mathbb{R}$ are twice continuously differentiable. Furthermore, for each $\boldsymbol{x}\in \mathcal{X}$, $F(\boldsymbol{x}, \cdot)$ is $\mu$-strongly concave on $\mathcal{Y}$; $f(\cdot, \cdot)$ is convex on $\mathcal{C}\times \mathcal{Y}$.
\end{assumption}

Based on Assumption \ref{assum:1} and Assumption \ref{assum:2}, the differentiability of $f^*(\boldsymbol{x})$ in (\ref{eq5}) can be ensured, as shown below.

\begin{proposition} \label{pro:1}
Suppose Assumptions \ref{assum:1} and \ref{assum:2} hold. Then $f^*(\boldsymbol{x})$ in (\ref{eq5}) is differentiable on $\mathcal{X}$. Furthermore, for each $\boldsymbol{x}\in \mathcal{X}$, $\{\nabla_{\boldsymbol{x}}f(\boldsymbol{x}, \boldsymbol{y}): \boldsymbol{y}\in \mathcal{S}(\boldsymbol{x})\}$ is a single point set, and 
\begin{equation*}
\nabla f^*(\boldsymbol{x}) = \nabla_{\boldsymbol{x}}f(\boldsymbol{x}, \boldsymbol{y}_1^*)
\end{equation*}
where $\boldsymbol{y}_1^* \in \mathcal{S}(\boldsymbol{x})$, and $\mathcal{S}(\boldsymbol{x}) := { { \operatorname {arg\,min} }_{\boldsymbol{y}} \, \{f(\boldsymbol{x}, \boldsymbol{y}): \boldsymbol{y} \in \mathcal{Y}\}}$.
\end{proposition}

For the proof of Proposition \ref{pro:1}, you can refer to Proposition 1 in \cite{16}. In the following, some Lipschitz conditions are imposed on the objective functions.

\begin{assumption}\label{assum:3}
Let $\boldsymbol{z}_1 := (\boldsymbol{x}_1, \boldsymbol{y}_1)$, $\boldsymbol{z}_2 := (\boldsymbol{x}_2, \boldsymbol{y}_2)$. 
\begin{enumerate}
\item[(1)] $F(\boldsymbol{x}, \boldsymbol{y})$ and $f(\boldsymbol{x}, \boldsymbol{y})$ are $h_0$- and $L_0$-Lipschitz, respectively, that is, for any $\boldsymbol{x}_1$, $\boldsymbol{x}_2\in \mathcal{X}$, $\boldsymbol{y}_1$, $\boldsymbol{y}_2 \in \mathcal{Y}$, 
\begin{equation*}
|F(\boldsymbol{z}_1)-F(\boldsymbol{z}_2)|\le h_0 \|\boldsymbol{z}_1 - \boldsymbol{z}_2 \|, \qquad |f(\boldsymbol{z}_1) - f(\boldsymbol{z}_2)| \le L_0 \| \boldsymbol{z}_1 - \boldsymbol{z}_2 \|.
\end{equation*}
\item[(2)] $\nabla F(\boldsymbol{x}, \boldsymbol{y})$ and $\nabla f(\boldsymbol{x}, \boldsymbol{y})$ are $h_1$- and $L_1$-Lipschitz, respectively, that is, for any $\boldsymbol{x}_1$, $\boldsymbol{x}_2\in \mathcal{X}$, $\boldsymbol{y}_1$, $\boldsymbol{y}_2 \in \mathcal{Y}$, 
\begin{equation*}
\|\nabla F(\boldsymbol{z}_1)-\nabla F(\boldsymbol{z}_2)\| \le h_1 \|\boldsymbol{z}_1 - \boldsymbol{z}_2\|, \qquad \|\nabla f(\boldsymbol{z}_1) - \nabla f(\boldsymbol{z}_2)\| \le L_1 \|\boldsymbol{z}_1 - \boldsymbol{z}_2 \|.
\end{equation*}
\item[(3)] $\nabla_{\boldsymbol{y}\boldsymbol{x}}f(\boldsymbol{x}, \boldsymbol{y})$ and $\nabla_{\boldsymbol{y}\boldsymbol{y}} f(\boldsymbol{x}, \boldsymbol{y})$ are $L_2$- and $L_3$-Lipschitz, respectively, that is, for any $\boldsymbol{x}_1$, $\boldsymbol{x}_2 \in \mathcal{X}$, $\boldsymbol{y}_1$, $\boldsymbol{y}_2 \in \mathcal{Y}$, 
\begin{equation*}
\|\nabla_{\boldsymbol{y}\boldsymbol{x}} f(\boldsymbol{z}_1) - \nabla_{\boldsymbol{y}\boldsymbol{x}} f(\boldsymbol{z}_2)\| \le L_2 \|\boldsymbol{z}_1 - \boldsymbol{z}_2\|,~ \|\nabla_{\boldsymbol{y}\boldsymbol{y}} f(\boldsymbol{z}_1) - \nabla_{\boldsymbol{y}\boldsymbol{y}}f(\boldsymbol{z}_2)\| \le L_3 \|\boldsymbol{z}_1 - \boldsymbol{z}_2 \|.
\end{equation*}
\end{enumerate}
\end{assumption}

To facilitate our further discussion, we provide the definitions of local optimal solutions and interior points below.

\begin{definition}\label{def:1}
A point $\bar{\boldsymbol{x}}\in \mathcal{X}$ is called a local optimal solution for problem PBP$\epsilon$ in (\ref{eq4}), if there exists a neighbourhood $\mathcal{A}$ of $\bar{\boldsymbol{x}}$ such that 
\begin{equation*}
\varphi_{\epsilon}(\bar{\boldsymbol{x}}) \le \varphi_{\epsilon}(\boldsymbol{x}), ~\qquad \forall \boldsymbol{x}\in \mathcal{X}\cap \mathcal{A}.
\end{equation*}
\end{definition}

\begin{definition} \label{def:2}
Given a set $\mathcal{A}$, a point $\bar{\boldsymbol{x}}$ is called an interior point of $\mathcal{A}$, if there exists a $\delta(>0)$ such that $\mathbb{B}_{\delta}(\bar{\boldsymbol{x}}) \subset \mathcal{A}$.
\end{definition}

Furthermore, to obtain a stationary condition of PBP$\epsilon$, the following assumption is made.

\begin{assumption} \label{assum:4}
All the local optimal solutions for problem PBP$\epsilon$ in (\ref{eq4}) are in the interior of $\mathcal{X}$.
\end{assumption}

We remark that Assumption \ref{assum:4} is mild. In the following, an example which satisfies Assumption \ref{assum:4} is shown.
\begin{example}  \label{example:1}
For problem PBP$\epsilon$ in (\ref{eq4}), let $\varphi_{\epsilon}(x)= \sin x$ and $\mathcal{X}=[\pi/2, 4\pi]$. It is obvious that the local  optimal solutions of $\varphi_{\epsilon}(x)$ on $\mathcal{X}$ are $x_1 = 3\pi/2$, and $x_2 = 7\pi/2$, and all of them are the interior points of $\mathcal{X}$.
\end{example}

\subsection{Basic Tools}
For the nonsmooth function $h(\boldsymbol{x}): \mathbb{R}^n \rightarrow \mathbb{R}$, $\hat{\partial}h(\bar{\boldsymbol{x}})$ and $\partial h(\bar{\boldsymbol{x}})$ denote the limiting subgradient and the Clarke generalized gradient of $h(\boldsymbol{x})$ at $\bar{\boldsymbol{x}}$, respectively.

In the following proposition, we borrow the conclusions on the Clarke normal cone and the Clarke generalized gradient which are given in \cite{27} and (2.5), (2.6) of \cite{26}.

\begin{proposition} \label{pro:2}
\begin{enumerate}
\item[(1)] If $\mathcal{A}\subset \mathbb{R}^n$ is a nonempty, closed, and convex set, then the Clarke normal cone to $\mathcal{A}$ at $\bar{\boldsymbol{x}}\in \mathcal{A}$, i.e., $\mathcal{N}_{\mathcal{A}}(\bar{\boldsymbol{x}})$, can be expressed as follows: 
\begin{equation*}
\mathcal{N}_{\mathcal{A}}(\bar{\boldsymbol{x}}) := \{\boldsymbol{v}: \langle \boldsymbol{v}, \boldsymbol{x}-\bar{\boldsymbol{x}}\rangle \le 0, ~\forall \boldsymbol{x}\in \mathcal{A}\}.
\end{equation*}
\item[(2)] If function $h(\boldsymbol{x}): \mathbb{R}^n \rightarrow \mathbb{R}$ is Lipschitz continuous near $\bar{\boldsymbol{x}}$, then $\partial h(\bar{\boldsymbol{x}}) = \text{co}\hat{\partial}h(\bar{x})$. Furthermore, $\text{co}\hat{\partial}(-h)(\bar{\boldsymbol{x}}) = -\text{co}\hat{\partial}h(\bar{\boldsymbol{x}})$.
\end{enumerate}
\end{proposition}

Next, following Corollary 2.4.2 in \cite{28}, we show the following calculation rule for the Clarke generalized gradient which is useful in the paper.

\begin{proposition} \label{pro:3}
Suppose $h_1(\boldsymbol{x}): \mathbb{R}^n \rightarrow \mathbb{R}$ is Lipschitz continuous near $\bar{\boldsymbol{x}}\in \mathbb{R}^n$, and $h_2(\boldsymbol{x}): \mathbb{R}^n \rightarrow \mathbb{R}$ is continuously differentiable near $\bar{\boldsymbol{x}}$. Then
\begin{equation*}
\partial(h_1 + h_2)(\bar{\boldsymbol{x}}) = \partial h_1(\bar{\boldsymbol{x}}) + \{\nabla h_2(\bar{\boldsymbol{x}})\}.
\end{equation*}
\end{proposition}

Then, we consider the following optimization problem:
\begin{align}     \label{eq6}\nonumber
&\min\limits_{\boldsymbol{x}, \boldsymbol{y}} F_1(\boldsymbol{x}) + F_2(\boldsymbol{x}, \boldsymbol{y})  \\ \nonumber
& ~\text{s.t.}~~\boldsymbol{x}\in \mathcal{X}, ~\boldsymbol{y}\in \mathcal{Y}, \\  
& ~~~~ ~~f(\boldsymbol{x}, \boldsymbol{y}) - v(\boldsymbol{x}) \le 0,  \\    \nonumber
&~~~~ ~~ g(\boldsymbol{x}, \boldsymbol{y})\le 0     \nonumber
\end{align}
where 
\begin{equation} \label{eq7}
v(\boldsymbol{x}) := \min\limits_{\boldsymbol{y}}\{f(\boldsymbol{x}, \boldsymbol{y}): \boldsymbol{y}\in \mathcal{Y},~g(\boldsymbol{x}, \boldsymbol{y}) \le 0\}.
\end{equation}

Notice that $v(\boldsymbol{x})$ in (\ref{eq7}) is generally nonsmooth even if functions $f(\boldsymbol{x}, \boldsymbol{y})$ and $g(\boldsymbol{x},\boldsymbol{y})$ are smooth. Therefore, problem (\ref{eq6}) is generally a nonsmooth optimization problem. Recently, \cite{15} showed that if all the functions involved for a nonsmooth optimization problem are local Lipschitz continuous, the generalized lagrange multiplier rule of Clarke can be used to obtain the necessary optimality condition of the nonsmooth optimization problem. 

To go further, the following assumption for problem (\ref{eq6}) is made.

\begin{assumption} \label{assum:5}
Let $\mathcal{X}\subset \mathbb{R}^n$ and $\mathcal{Y}\subset \mathbb{R}^m$ be nonempty, convex, and compact. For any $\bar{\boldsymbol{x}}\in \mathcal{X}$, suppose $F_1(\boldsymbol{x}): \mathcal{X} \rightarrow \mathbb{R}$ is Lipschitz continuous near $\bar{\boldsymbol{x}}$, and $F_2(\boldsymbol{x}, \boldsymbol{y})$, $f(\boldsymbol{x}, \boldsymbol{y})$, $g(\boldsymbol{x}, \boldsymbol{y}): \mathcal{X} \times \mathcal{Y} \rightarrow \mathbb{R}$ are continuously differentiable functions.
\end{assumption}

It is obvious that Assumption \ref{assum:5} can imply that $F_2(\boldsymbol{x}, \boldsymbol{y})$, $f(\boldsymbol{x}, \boldsymbol{y})$, and $g(\boldsymbol{x}, \boldsymbol{y})$ are local Lipschitz continuous on $\mathcal{X}\times \mathcal{Y}$. In the following, a sufficient condition for the local Lipschitz continuity of $v(\boldsymbol{x})$ in (\ref{eq7}) is provided.

As preparation, for each $\boldsymbol{x}\in \mathcal{X}$, we define
\begin{equation}   \label{eq8}
\mathcal{A}_{\boldsymbol{x}} := { { \operatorname {arg\,min} }_{\boldsymbol{y}} \, \{f(\boldsymbol{x}, \boldsymbol{y}): \boldsymbol{y} \in \mathcal{Y}, ~ g(\boldsymbol{x}, \boldsymbol{y})\le 0 \}}.
\end{equation}
Furthermore, given $\boldsymbol{x}\in \mathcal{X}$, for $\boldsymbol{y}$ satisfying $\boldsymbol{y}\in \mathcal{Y}$ and $g(\boldsymbol{x}, \boldsymbol{y})\le 0$, we define
\begin{align}  \label{eq9} \nonumber
\mathcal{M}_{\boldsymbol{x}}^{\lambda}(\boldsymbol{y}) :=& \left \{ \sigma: 0 \in \lambda \nabla_{\boldsymbol{y}}f(\boldsymbol{x}, \boldsymbol{y}) + \sigma \nabla_{\boldsymbol{y}} g(\boldsymbol{x}, \boldsymbol{y}) + \mathcal{N}_{\mathcal{Y}}(\boldsymbol{y}), \right. \\  
& \left. \qquad \sigma\ge 0, \qquad \sigma g(\boldsymbol{x}, \boldsymbol{y})=0 \right \}
\end{align}
where $\lambda \in \{0, 1\}$, and the definition of $\mathcal{N}_{\mathcal{Y}}(\boldsymbol{y})$ is given in Proposition \ref{pro:2}.

\begin{proposition} \label{pro:4}
Suppose Assumption \ref{assum:5} holds. Let $\boldsymbol{x}_0$ be an interior point of $\mathcal{X}$. Define 
\begin{equation*}
\mathcal{M}_{\boldsymbol{x}_0}^0(\mathcal{A}_{\boldsymbol{x}_0}) := \mathop{\cup}\limits_{\boldsymbol{y}\in \mathcal{A}_{\boldsymbol{x}_0}} \mathcal{M}_{\boldsymbol{x}_0}^0(\boldsymbol{y})
\end{equation*}
where the definitions of $\mathcal{A}_{\boldsymbol{x}_0}$ and $\mathcal{M}_{\boldsymbol{x}_0}^0(\boldsymbol{y})$ are given in (\ref{eq8}) and (\ref{eq9}) with $\boldsymbol{x}=\boldsymbol{x}_0$ and $\lambda = 0$, respectively. If $\mathcal{M}_{\boldsymbol{x}_0}^0(\mathcal{A}_{\boldsymbol{x}_0})=\{0\}$, then $v(\boldsymbol{x})$ in (\ref{eq7}) is Lipschitz continuous near $\boldsymbol{x}_0$, and the limiting subgradient of $v(\boldsymbol{x})$ at $\boldsymbol{x}_0$ satisfies 
\begin{equation*}
\hat{\partial}v(\boldsymbol{x}_0)  \subset \{ \nabla_{\boldsymbol{x}}f(\boldsymbol{x}_0, \boldsymbol{y}) + \sigma \nabla_{\boldsymbol{x}}g(\boldsymbol{x}_0, \boldsymbol{y}): \boldsymbol{y} \in \mathcal{A}_{\boldsymbol{x}_0}, \sigma \in \mathcal{M}_{\boldsymbol{x}_0}^1(\boldsymbol{y})\}
\end{equation*} 
where the definitions of $\mathcal{A}_{\boldsymbol{x}_0}$ and $\mathcal{M}_{\boldsymbol{x}_0}^1(\boldsymbol{y})$ are given in (\ref{eq8}) and (\ref{eq9}) with $\boldsymbol{x}=\boldsymbol{x}_0$ and $\lambda = 1$, respectively.
\end{proposition}

For the proof of Proposition \ref{pro:4}, you can refer to Proposition 2.1 and Remark 2.1 in \cite{15}. In the following, a Fritz-John type necessary optimality condition for the nonsmooth problem in (\ref{eq6}) is shown. For more details, see Theorem 2.1 in \cite{15} and Proposition 1.3 in \cite{13}.

\begin{theorem} \label{tho:1}
Suppose Assumption \ref{assum:5} holds. Let $(\bar{\boldsymbol{x}}, \bar{\boldsymbol{y}})$ be a local optimal solution to problem (\ref{eq6}). If $v(\boldsymbol{x})$ is Lipschitz continuous near $\bar{\boldsymbol{x}}$, then there exist $\lambda_1\in \{0, 1\}$, $\lambda_2\ge 0$, $\lambda_3\ge 0$ not all zero such that 
\begin{align*}
&0 \in \lambda_1(\partial F_1(\bar{\boldsymbol{x}}) + \nabla_{\boldsymbol{x}}F_2(\bar{\boldsymbol{x}}, \bar{\boldsymbol{y}})) + \lambda_2 (\nabla_{\boldsymbol{x}}f(\bar{\boldsymbol{x}}, \bar{\boldsymbol{y}})-\partial v(\bar{\boldsymbol{x}})) 
+ \lambda_3 \nabla_{\boldsymbol{x}}g(\bar{\boldsymbol{x}}, \bar{\boldsymbol{y}}) + \mathcal{N}_{\mathcal{X}}(\bar{\boldsymbol{x}}), \\
&0 \in \lambda_1 \nabla_{\boldsymbol{y}} F_2(\bar{\boldsymbol{x}}, \bar{\boldsymbol{y}}) + \lambda_2 \nabla_{\boldsymbol{y}}f(\bar{\boldsymbol{x}}, \bar{\boldsymbol{y}}) + \lambda_3 \nabla_{\boldsymbol{y}}g(\bar{\boldsymbol{x}}, \bar{\boldsymbol{y}}) + \mathcal{N}_{\mathcal{Y}}(\bar{\boldsymbol{y}}), \\
& \lambda_3 g(\bar{\boldsymbol{x}}, \bar{\boldsymbol{y}}) = 0
\end{align*}
where $\partial F_1(\bar{\boldsymbol{x}})$ and $\partial v(\bar{\boldsymbol{x}})$ are the Clarke generalized gradient of $F_1(\boldsymbol{x})$ and $v(\boldsymbol{x})$ at $\bar{\boldsymbol{x}}$, respectively, and the definitions of $\mathcal{N}_{\mathcal{X}}(\bar{\boldsymbol{x}})$ and $\mathcal{N}_{\mathcal{Y}}(\bar{\boldsymbol{y}})$ are given in Proposition \ref{pro:2}.
\end{theorem}

\section{Algorithm} \label{section:3}
The structure of problem PBP$\epsilon$ in (\ref{eq4}) is intricate. It involves three intricated optimization problems. To obtain a problem easier to tackle, the existing studies \cite{14,23} for problem PBP$\epsilon$ focus on reformulating PBP$\epsilon$ into a single level mathematical program MPCC$\epsilon$ based on the KKT condition. However, problem PBP$\epsilon$ and MPCC$\epsilon$ are only equivalent in the sense of global optimal solutions. Thus, they need to solve MPCC$\epsilon$ to global optimum to obtain a solution of PBP$\epsilon$. Unfortunately, MPCC$\epsilon$ is still hard to solve, and the solution procedure for MPCC$\epsilon$ to obtain the global optimal solutions is not practical for machine learning applications since there are many constraints for the MPCC$\epsilon$. 

In the paper, we propose a new approach for PBP$\epsilon$, i.e., reformulating PBP$\epsilon$ into a two level approximate minimax problem, and solving a sequence of approximate minimax problem to obtain the solution of PBP$\epsilon$.

\subsection{Approximate Minimax Problem}

Given $\boldsymbol{x}\in \mathcal{X}$, recall that $\varphi_{\epsilon}(\boldsymbol{x})$ in (\ref{eq4}) is the value function of the following optimization problem:
\begin{equation}   \label{eq10}
\max\limits_{\boldsymbol{y}} F(\boldsymbol{x}, \boldsymbol{y})  \qquad \text{s.t.}~ ~\boldsymbol{y} \in \mathcal{S}_{\epsilon}(\boldsymbol{x})
\end{equation}
where $\mathcal{S}_{\epsilon}(\boldsymbol{x}):=\{\boldsymbol{y} \in \mathcal{Y} : f(\boldsymbol{x}, \boldsymbol{y})\le f^*(\boldsymbol{x})+\epsilon \}$, and $f^*(\boldsymbol{x})$
is defined in (\ref{eq5}).

By using the definition of $\mathcal{S}_{\epsilon}(\boldsymbol{x})$, problem (\ref{eq10}) can be written as 
\begin{equation*}   
\max\limits_{\boldsymbol{y}} F(\boldsymbol{x}, \boldsymbol{y})  \qquad \text{s.t.}~ ~\boldsymbol{y} \in \mathcal{Y}, ~f(\boldsymbol{x}, \boldsymbol{y})\le f^*(\boldsymbol{x})+\epsilon .
\end{equation*}

Then, motivated by the recent work in \cite{24} for solving problem PBP, a log-barrier term is introduced into the objective function $F(\boldsymbol{x}, \boldsymbol{y})$ to remove the inequality constraint, and the following problem
\begin{equation}   \label{eq11}
\max\limits_{\boldsymbol{y} \in \mathcal{Y}} F(\boldsymbol{x}, \boldsymbol{y}) + \tau \ln( f^*(\boldsymbol{x}) + \epsilon - f(\boldsymbol{x}, \boldsymbol{y})) 
\end{equation}
where $\tau \in (0, 1)$, is obtained.

It is obvious that $f^*(\boldsymbol{x})$ in (\ref{eq11}) is an implicitly defined function. To obtain the value of $f^*(\boldsymbol{x})$, we need to solve problem 
\begin{equation}   \label{eq12}
\min\limits_{\boldsymbol{y} \in \mathcal{Y}} f(\boldsymbol{x}, \boldsymbol{y}) 
\end{equation}
to global optimality, which is impractical. In order to deal with the problem, we replace $f^*(\boldsymbol{x})$ in (\ref{eq11}) with 
\begin{equation}   \label{eq13}
f_J(\boldsymbol{x}):= f(\boldsymbol{x}, \boldsymbol{y}_J(\boldsymbol{x})).
\end{equation}
where $\boldsymbol{y}_J(\boldsymbol{x})$, an approximate solution to problem (\ref{eq12}), is obtained by running $J$ steps of projected gradient descent starting from initial value $\boldsymbol{y}_0(\boldsymbol{x}) \in \mathcal{Y}$ in the form of 
\begin{equation}   \label{eq14}
\boldsymbol{y}_{j+1}(\boldsymbol{x}) = \text{proj}_{\mathcal{Y}}\big(\boldsymbol{y}_j(\boldsymbol{x}) -  \alpha \nabla_{\boldsymbol{y}}f(\boldsymbol{x}, \boldsymbol{y}_j(\boldsymbol{x}))\big)
\end{equation}
with $j=0,\ldots, J-1$, where $\alpha$ is the stepsize.

As a result, we approximate PBP$\epsilon$ in (\ref{eq4}) by the minimax problem
\begin{equation}  \label{eq15}
\min\limits_{\boldsymbol{x}\in \mathcal{X}} \varphi_{\epsilon, \tau, J}(\boldsymbol{x}), \quad \varphi_{\epsilon, \tau, J}(\boldsymbol{x}):= \max\limits_{\boldsymbol{y}\in \mathcal{Y}} G_{\epsilon, \tau, J}(\boldsymbol{x}, \boldsymbol{y})
\end{equation}
where $0<\tau<1$, $J$ is a positive integer,
\begin{equation}  \label{eq16}
G_{\epsilon, \tau, J}(\boldsymbol{x}, \boldsymbol{y}) = F(\boldsymbol{x}, \boldsymbol{y}) + \tau \ln(f_J(\boldsymbol{x})+ \epsilon - f(\boldsymbol{x}, \boldsymbol{y}))
\end{equation}
and the definition of $f_J(\boldsymbol{x})$ is given in (\ref{eq13}).

To go further, the following assumption is made on the projected gradient descent steps for solving problem (\ref{eq12}).

\begin{assumption} \label{assum:6}
For the projected gradient descent steps in (\ref{eq14}), let $\alpha = 1/L_1$, $\boldsymbol{y}_0(\boldsymbol{x}) = \boldsymbol{y}_0 \in \mathcal{Y}$ for any $\boldsymbol{x}\in \mathcal{X}$, and for any $\boldsymbol{x}\in \mathcal{X}$, $\boldsymbol{y}_0 \in \mathcal{Y}$, positive integer $J$, the projected gradient descent steps to solve problem (\ref{eq12}) in (\ref{eq14}) satisfy
\begin{align}   \label{eq17}
\boldsymbol{y}_{j+1}(\boldsymbol{x}) & = \text{proj}_{\mathcal{Y}}\big(\boldsymbol{y}_j(\boldsymbol{x}) -  \alpha \nabla_{\boldsymbol{y}}f(\boldsymbol{x}, \boldsymbol{y}_j(\boldsymbol{x}))\big) \\  \nonumber
& = \boldsymbol{y}_j(\boldsymbol{x}) -  \alpha \nabla_{\boldsymbol{y}}f(\boldsymbol{x}, \boldsymbol{y}_j(\boldsymbol{x})), ~ j=0, \ldots, J-1
\end{align}
where $L_1$ is given in Assumption \ref{assum:3}.
\end{assumption}

We remark that Assumption \ref{assum:6} is mild. In the following, an example which satisfies Assumption \ref{assum:6} is shown.

\begin{example}  \label{example:2}
For problem (\ref{eq12}), we let $\mathcal{X}=[\pi/2, 4\pi]$, $\mathcal{Y}=[-40, 40] \times [-20, 20]$, $f(x, \boldsymbol{y})= ([\boldsymbol{y}]_1 - 2 [\boldsymbol{y}]_2)^2 + x$, where $\boldsymbol{y} := ([\boldsymbol{y}]_1, [\boldsymbol{y}]_2)$. By simple calculations, it is easy to know that for any $x \in \mathcal{X}$, $\boldsymbol{y}_j(x)\in \mathcal{Y}$, $0<\alpha < 1/8$, we have 
\begin{equation*}
\boldsymbol{y}_j(x) - \alpha \nabla_{\boldsymbol{y}}f(x, \boldsymbol{y}_j(x)) \in \mathcal{Y}.
\end{equation*}
Therefore, for any $x\in \mathcal{X}$, $\boldsymbol{y}_0 \in \mathcal{Y}$, positive integer $J$, the formula in (\ref{eq17}) always can be satisfied.
\end{example}

Under Assumptions \ref{assum:2} and \ref{assum:6}, $f_J(\boldsymbol{\boldsymbol{x}})$ in (\ref{eq13}) can be ensured to be differentiable(see Lemma \ref{lemma:4} in the appendices). Then, for the minimax problem in (\ref{eq15}), we define the first order Nash equilibrium point as follows.

\begin{definition}  \label{def:3}
We say that a point $(\bar{\boldsymbol{x}}, \bar{\boldsymbol{y}})$, which satisfies $\bar{\boldsymbol{x}}\in \mathcal{X}$, $\bar{\boldsymbol{y}}\in \mathcal{Y}$, and $f_J(\bar{\boldsymbol{x}}) + \epsilon - f(\bar{\boldsymbol{x}}, \bar{\boldsymbol{y}})> 0$, is a $\sigma$-first order Nash equilibrium($\sigma$-FNE) point of problem (\ref{eq15}) if $(\bar{\boldsymbol{x}}, \bar{\boldsymbol{y}})$ satisfies
\begin{align*}
& \langle  \nabla_{\boldsymbol{x}}G_{\epsilon, \tau, J}(\bar{\boldsymbol{x}}, \bar{\boldsymbol{y}}), \boldsymbol{x} - \bar{\boldsymbol{x}} \rangle \ge - \sigma, \qquad  \forall \boldsymbol{x}\in \mathcal{X}, \\
& \|\nabla_{\boldsymbol{y}} G_{\epsilon, \tau, J}(\bar{\boldsymbol{x}}, \bar{\boldsymbol{y}}) \| \le \sigma.
\end{align*}

\end{definition}

Furthermore, recall that for each $\boldsymbol{x}\in \mathcal{X}$, $\varphi_{\epsilon, \tau, J}(\boldsymbol{x})$ in (\ref{eq15}) is the value function of maximization problem:
\begin{equation}  \label{eq18}
\max\limits_{\boldsymbol{y}\in \mathcal{Y}} G_{\epsilon, \tau, J}(\boldsymbol{x}, \boldsymbol{y})
\end{equation}
where $G_{\epsilon, \tau, J}(\boldsymbol{x}, \boldsymbol{y})$ is given in (\ref{eq16}). In the following, an assumption on the maximization problem in (\ref{eq18}) is made.

\begin{assumption} \label{assum:7}
There exists a positive number $c(c\le \min\{\epsilon/2, L_0 H, 1\})$ such that for any $0<\tau<1$, positive integer $J$, $\boldsymbol{x}\in \mathcal{X}$, at least a global optimal solution to problem (\ref{eq18}) is in the set $\mathcal{Y}_J$ defined as follows
\begin{equation}\label{eq19}
\mathcal{Y}_J := \mathop{\cap}\limits_{\boldsymbol{z}\in \mathcal{X}}\{ \boldsymbol{y} \in \mathcal{Y}: f_J(\boldsymbol{z}) +\epsilon - f(\boldsymbol{z}, \boldsymbol{y}) \ge c \}
\end{equation}
where $H$ is given in Assumption \ref{assum:1}, $L_0$ is given in Assumption \ref{assum:3}, and for each $\boldsymbol{z}\in \mathcal{X}$, the definition of $f_{J}(\boldsymbol{z})$ is given in (\ref{eq13}).
\end{assumption}

\begin{remark}\label{remark:1}
Under Assumptions \ref{assum:1}, \ref{assum:2}, and \ref{assum:7}, for any positive integer $J$, the set $\mathcal{Y}_J$ defined in (\ref{eq19}) is nonempty, compact, and convex. To be specific, let $J_0$ be any given positive integer. From Assumption \ref{assum:7}, we know that $\mathcal{Y}_{J_0}$ is nonempty since for any $\tau_0 \in (0, 1)$, $\boldsymbol{x}_0\in \mathcal{X}$, at least a global optimal solution to problem (\ref{eq18}) with $\tau = \tau_0$, $J=J_0$, $\boldsymbol{x}=\boldsymbol{x}_0$ is in the set $\mathcal{Y}_{J_0}$; from Assumption \ref{assum:1}, we know that $\mathcal{Y}_{J_0}$ is bounded; from Assumption \ref{assum:2}, we know that $f(\boldsymbol{x},\boldsymbol{y})$ is continuous and convex w.r.t. $\boldsymbol{y}$, and therefore, $\mathcal{Y}_{J_0}$ is closed and convex. Therefore, $\mathcal{Y}_{J_0}$ is nonempty, compact, and convex.
\end{remark}

In the following, an example which satisfies Assumption \ref{assum:7} is shown.

\begin{example}  \label{example:3}
For problem PBP$\epsilon$ in (\ref{eq4}), we let $\mathcal{X}=[\pi/2, 4\pi]$, $\mathcal{Y}=[-40, 40] \times [-20, 20]$, $F(x, \boldsymbol{y}) = -([\boldsymbol{y}]_1 - x)^2 - ([\boldsymbol{y}]_2 - x/2)^2 + \sin x$, $f(x, \boldsymbol{y})= ([\boldsymbol{y}]_1 - 2 [\boldsymbol{y}]_2)^2 + x$, where $\boldsymbol{y} := ([\boldsymbol{y}]_1, [\boldsymbol{y}]_2)$. Following the above discussion, PBP$\epsilon$ can be approximated by the minimax problem in (\ref{eq15}), and the corresponding maximization problem is in (\ref{eq18}). Setting $c=\epsilon/2$. It is easy to know that for any $x \in \mathcal{X}$, positive integer $J$, we have $\mathcal{S}(x)\subset \{ \boldsymbol{y}\in \mathcal{Y}: f_J(x) + \epsilon/2 - f(x, \boldsymbol{y}) \ge 0\}$, which can be obtained by using the definition of $\mathcal{S}(x)$ in Proposition \ref{pro:1} and the definition of $f_J(x)$ in (\ref{eq13}). As a result, for any positive integer $J$, we have 
\begin{equation*}
\mathop{\cap}_{x\in \mathcal{X}} \mathcal{S}(x)\subset \mathcal{Y}_J
\end{equation*}
where $\mathcal{Y}_J$ is defined in (\ref{eq19}) with $c = \epsilon/2$. Furthermore, by simple calculations, it is easy to know that $\mathop{\cap}_{x \in \mathcal{X}} {\mathcal{S}(x)}= \{ \boldsymbol{y}=([\boldsymbol{y}]_1, [\boldsymbol{y}_2)] \in \mathcal{Y}: [\boldsymbol{y}]_1 = 2[\boldsymbol{y}]_2 \}$, and for any $\tau>0$, positive integer $J$, $x \in \mathcal{X}$, the global optimal solution to problem (\ref{eq18}) is $(x, x/2)$, which belongs to the set $\mathop{\cap}_{x \in \mathcal{X}} \mathcal{S}(x)$, and thus belongs to the set $\mathcal{Y}_J$. Therefore, Assumption \ref{assum:7} is satisfied by setting $c=\epsilon/2$.
\end{example}

Next, for any $0<\tau<1$, positive integer $J$, $\boldsymbol{x}\in \mathcal{X}$, we consider the following optimization problem 
\begin{equation}  \label{eq20}
\max\limits_{\boldsymbol{y}\in \mathcal{Y}_J(\boldsymbol{x})} G_{\epsilon, \tau, J}(\boldsymbol{x}, \boldsymbol{y})
\end{equation}
where 
\begin{equation} \label{eq21}
\mathcal{Y}_J(\boldsymbol{x}) := \bigg\{ \boldsymbol{y} \in \mathcal{Y}: f_J(\boldsymbol{x}) + \epsilon - f(\boldsymbol{x}, \boldsymbol{y})\ge \frac{c_0}{2}\bigg\}
\end{equation}
, and $c_0$ is a sufficient small number. Notice that if Assumption \ref{assum:7} holds, and we set $c_0=c$($c$ in Assumption \ref{assum:7}), the optimal solution to problem (\ref{eq20}) must be an optimal solution to problem (\ref{eq18}) since for the $\mathcal{Y}_J$ defined in (\ref{eq19}), we have $\mathcal{Y}_J\subset \mathcal{Y}_J(\boldsymbol{x})$. Thus, we can obtain an approximate solution to problem (\ref{eq18}) by solving problem (\ref{eq20}). To facilicate our theoretical analysis, in the following, instead of solving problem (\ref{eq18}) directly, we solve problem (\ref{eq20}) to obtain an approximate solution to problem (\ref{eq18}).

\subsection{A New Algorithm Named PVFIM for Problem PBP$\epsilon$ in (\ref{eq4})}

Based on the above discussion, a new algorithm named PVFIM is proposed for problem PBP$\epsilon$ in (\ref{eq4}), as shown in Algorithm \ref{alg:1}. For the minimax problem in (\ref{eq15}), by setting $\tau=\tau_l$, $J=J_l$, a $\sigma_l$-FNE point $(\boldsymbol{x}_l, \boldsymbol{y}_l)$ can be obtained(see line 2 in Algorithm \ref{alg:1}). By varying the values of $\tau_l$, $J_l$, $\sigma_l$, it produces a sequence of approximate first order Nash equilibrium points $\{(\boldsymbol{x}_l, \boldsymbol{y}_l)\}$, from which we can obtain a stationary point of problem PBP$\epsilon$. 

\begin{algorithm}[htb]
\caption{A Perturbed Value-Function-Based Interior-Point Method(PVFIM) for PBP$\epsilon$ in (\ref{eq4})}  
\label{alg:1}
\begin{algorithmic}[1]
\REQUIRE  $\{\tau_l\}_{l=1}^{\infty}$, $\{J_l\}_{l=1}^{\infty}$, $\{\sigma_l\}_{l=1}^{\infty}$, $c_0$. Set $l=1$\\
\REPEAT
\STATE Find a $\sigma_l$-FNE point $(\boldsymbol{x}_l, \boldsymbol{y}_l)$ for problem (\ref{eq15}) with $\tau = \tau_l$, $J=J_l$ via Algirithm \ref{alg:2} \\
\STATE $l = l + 1$
\UNTIL{convergence of $(\boldsymbol{x}_l, \boldsymbol{y}_l)$ to a point $(\bar{\boldsymbol{x}}, \bar{\boldsymbol{y}})$ }
\end{algorithmic}
\end{algorithm}

\begin{algorithm}[htb]
\caption{Find a $\sigma_l$-FNE Point $(\boldsymbol{x}_l, \boldsymbol{y}_l)$ for Problem (\ref{eq15}) with $\tau = \tau_l$, $J=J_l$ }  
\label{alg:2}
\begin{algorithmic}[1]
\REQUIRE  $\boldsymbol{x}_0 \in \mathcal{X}$, stepsizes $\alpha$, $\beta$, and $\eta$, iteration steps $T_l$ and $K_l$. Set $T=T_l$, $K=K_l$\\
\FOR{$t = 0, \ldots, T-1$}
\STATE Initialize $\boldsymbol{y}_0(\boldsymbol{x}_t) \in \mathcal{Y}$ \\
\FOR {$j = 0, \ldots, J-1$}
\STATE $\boldsymbol{y}_{j+1}(\boldsymbol{x}_t) = \text{proj}_{\mathcal{Y}}\big(\boldsymbol{y}_j(\boldsymbol{x}_t) -  \alpha \nabla_{\boldsymbol{y}}f(\boldsymbol{x}_t, \boldsymbol{y}_j(\boldsymbol{x}_t))\big)$\\
\ENDFOR
\STATE Compute $f_J(\boldsymbol{x}_t)$ via (\ref{eq13}) \\
\STATE Set $\mathcal{Y}_J(\boldsymbol{x}_t)$ to be the set in (\ref{eq21}) with $\boldsymbol{x}=\boldsymbol{x}_t$ and $c_0$ given in Algorithm \ref{alg:1}   \\
\STATE Initialize $\boldsymbol{y}_0(\boldsymbol{x}_t) \in \mathcal{Y}_J(\boldsymbol{x}_t)$
\FOR {$k = 0, \ldots, K-1$}
\STATE $\boldsymbol{y}_{k+1}(\boldsymbol{x}_t)= \text{proj}_{\mathcal{Y}_J(\boldsymbol{x}_t)}(\boldsymbol{y}_k(\boldsymbol{x}_t) +  \beta \nabla_{\boldsymbol{y}}G_{\epsilon, \tau, J}(\boldsymbol{x}_t, \boldsymbol{y}_k(\boldsymbol{x}_t)))$\\
\ENDFOR 
\STATE $\boldsymbol{x}_{t+1} = \text{proj}_{\mathcal{X}}(\boldsymbol{x}_t - \eta \boldsymbol{a}_t )$, where $\boldsymbol{a}_t$ is in (\ref{eq22}) \\
\ENDFOR
\STATE Theory: choose $\bar{t}\in \{0, \ldots, T-1\}$, set $\boldsymbol{x}_l = \boldsymbol{x}_{\bar{t}}$, $\boldsymbol{y}_l = \boldsymbol{y}_K(\boldsymbol{x}_{l}) $

Practice: set $\boldsymbol{x}_l = \boldsymbol{x}_T$, $\boldsymbol{y}_l = \boldsymbol{y}_{K}(\boldsymbol{x}_l)$
\end{algorithmic}
\end{algorithm}

We next illustrate how to obtain a $\sigma_l$-FNE point for problem (\ref{eq15}) with $\tau = \tau_l$, $J=J_l$. As shown in Algorithm \ref{alg:2}, for each $\boldsymbol{x}_t$, it first runs $J$ steps of projected gradient descent starting from initial value $\boldsymbol{y}_0(\boldsymbol{x_t})$ in the form of (\ref{eq14}) to obtain an approximate solution $\boldsymbol{y}_J(\boldsymbol{x}_t)$ to problem (\ref{eq12}) with $\boldsymbol{x}=\boldsymbol{x}_t$(see line 4 in Algorithm \ref{alg:2}), and thus obtain the value of $f_J(\boldsymbol{x}_t)$(see line 6 in Algorithm \ref{alg:2}). Then by setting the projected region to be $\mathcal{Y}_J(\boldsymbol{x}_t)$(see line 7 in Algorithm \ref{alg:2}), it runs $K(=K_l)$ steps of projected gradient ascent to obtain an approximate solution $\boldsymbol{y}_K(\boldsymbol{x}_t)$ to problem (\ref{eq20}), and thus obtain an approximate solution to the maximization problem of problem (\ref{eq15}) with $\boldsymbol{x}=\boldsymbol{x}_t$(see line 10 in Algorithm \ref{alg:2}). Based on the output $\boldsymbol{y}_J(\boldsymbol{x}_t)$ in line 4 of Algorithm \ref{alg:2} and the output $\boldsymbol{y}_K(\boldsymbol{x}_t)$ in line 10 of Algorithm \ref{alg:2}, it constructs
\begin{align} \label{eq22}
\boldsymbol{a}_t  = \nabla_{\boldsymbol{x}}F(\boldsymbol{x}_t, \boldsymbol{y}_K(\boldsymbol{x}_t))  + \frac{\tau}{f_J(\boldsymbol{x}_t) + \epsilon - f(\boldsymbol{x}_t, \boldsymbol{y}_K(\boldsymbol{x}_t))}\boldsymbol{b}_t
\end{align}
with $\boldsymbol{b}_t = \nabla_{\boldsymbol{x}} f(\boldsymbol{x}_t, \boldsymbol{y}_J(\boldsymbol{x}_t)) - \nabla_{\boldsymbol{x}}f(\boldsymbol{x}_t, \boldsymbol{y}_K(\boldsymbol{x}_t))$ as an estimate of the gradient $\nabla \varphi_{\epsilon, \tau, J}(\boldsymbol{x}_t)$.

\section{Theoretical Results} \label{section:4}
In the section, we first present the stationary condition for problem PBP$\epsilon$ in (\ref{eq4}). Then, we prove that PVFIM can converge to the stationary point of problem PBP$\epsilon$. Detailed theoretical analysis can be found in the appendices.

\subsection{Stationary Condition for Problem PBP$\epsilon$ in (\ref{eq4})}

Recently, the necessary optimality conditions\cite{26} for PBP are derived. It is naturally to think of utilizing the necessary optimality conditions of PBP to obtain the stationary condition of PBP$\epsilon$ since PBP$\epsilon$ is a variant of PBP. However, the necessary optimality conditions for PBP are obtained under additional Lipschitz conditions, which are very strong and would limit the application scope of the algorithm. To deal with the problem, we use the standard version of PBP$\epsilon$, the value function approach, and the Fritz-John type necessary optimality condition to obtain the stationary condition of PBP$\epsilon$.

Following \cite{14}, the standard version of problem PBP$\epsilon$ in (\ref{eq4}), denoted as SPBP$\epsilon$, is as follows

\begin{equation}   \label{eq23}
\min\limits_{\boldsymbol{x}, \boldsymbol{y}} F(\boldsymbol{x}, \boldsymbol{y}), \qquad \text{s.t.}~ ~\boldsymbol{x}\in \mathcal{X},~\boldsymbol{y} \in \mathcal{R}_{\epsilon}(\boldsymbol{x})
\end{equation}
where 
\begin{equation}   \label{eq24}
\mathcal{R}_{\epsilon}(\boldsymbol{x}) := { \underset {\boldsymbol{y}\in \mathcal{S}_{\epsilon}(\boldsymbol{x})} { \operatorname {arg\,max} } \, F(\boldsymbol{x}, \boldsymbol{y})}.
\end{equation}

The relation between local optimal points of PBP$\epsilon$ and SPBP$\epsilon$ can be established as follows; see Proposition 4.2 in \cite{14}.

\begin{proposition}[\cite{14}] \label{pro:5}
Suppose Assumptions \ref{assum:1} and \ref{assum:2} hold. 
\begin{enumerate}
\item[(1)] If $\bar{\boldsymbol{x}}\in \mathcal{X}$ is a local optimal solution to problem PBP$\epsilon$ in (\ref{eq4}), then $(\bar{\boldsymbol{x}}, \bar{\boldsymbol{y}})$ is a local optimal solution to problem SPBP$\epsilon$ in (\ref{eq23}) for any $\bar{\boldsymbol{y}}\in \mathcal{R}_{\epsilon}(\bar{\boldsymbol{x}})$.
\item[(2)] Let $\bar{\boldsymbol{x}} \in \mathcal{X}$. If for any $\bar{\boldsymbol{y}}\in \mathcal{R}_{\epsilon}(\bar{\boldsymbol{x}})$, $(\bar{\boldsymbol{x}}, \bar{\boldsymbol{y}})$ is a local optimal solution to problem SPBP$\epsilon$ in (\ref{eq23}), then $\bar{\boldsymbol{x}}$ is a local optimal solution to problem PBP$\epsilon$ in (\ref{eq4}).
\end{enumerate}
\end{proposition}

Proposition \ref{pro:5} shows that problem PBP$\epsilon$ in (\ref{eq4}) and problem SPBP$\epsilon$ in (\ref{eq23}) are equivalent in terms of local optimal solutions if $\mathcal{R}_{\epsilon}(\boldsymbol{x})$ is a single point set for any $\boldsymbol{x}\in \mathcal{X}$. Notice that by using the value function approach \cite{12,13,15}, SPBP$\epsilon$ in (\ref{eq23}) can be written as the following single level mathematical programming problem, denoted as SPBP$\epsilon$ still, 
\begin{align}     \label{eq25}\nonumber
&\min\limits_{\boldsymbol{x}, \boldsymbol{y}} F(\boldsymbol{x}, \boldsymbol{y})  \\ \nonumber
& ~\text{s.t.}~~\boldsymbol{x}\in \mathcal{X}, ~\boldsymbol{y}\in \mathcal{Y}, \\  
& ~~~~ ~~F^*(\boldsymbol{x})\le F(\boldsymbol{x}, \boldsymbol{y}),  \\    \nonumber
&~~~~ ~~ f(\boldsymbol{x}, \boldsymbol{y})\le f^*(\boldsymbol{x}) + \epsilon     \nonumber
\end{align}
with the help of function 
\begin{equation} \label{eq26}
F^*(\boldsymbol{x}):= \max\limits_{\boldsymbol{y}\in \mathcal{S}_{\epsilon}(\boldsymbol{x})} F(\boldsymbol{x}, \boldsymbol{y})
\end{equation}
and the definition of $\mathcal{S}_{\epsilon}(\boldsymbol{x})$.

The equivalence between Problem SPBP$\epsilon$ in (\ref{eq23}) and problem SPBP$\epsilon$ in (\ref{eq25}) is shown below.

\begin{proposition} \label{pro:6}
Suppose Assumptions \ref{assum:1} and \ref{assum:2} hold. Problem SPBP$\epsilon$ in (\ref{eq23}) and problem SPBP$\epsilon$ in (\ref{eq25}) are fully equivalent in terms of local and global optimal solutions.
\end{proposition}

Proposition \ref{pro:6} can be easily obtained following Remark 3.1 in \cite{10}. By combining Proposition \ref{pro:5} and Proposition \ref{pro:6}, we know that, if $\bar{\boldsymbol{x}}$ is a local optimal solution to problem PBP$\epsilon$ in (\ref{eq4}), then $(\bar{\boldsymbol{x}}, \bar{\boldsymbol{y}})$ is a local optimal solution to problem SPBP$\epsilon$ in (\ref{eq25}) for any $\bar{\boldsymbol{y}}\in \mathcal{R}_{\epsilon}(\bar{\boldsymbol{x}})$. Thus, the stationary condition for problem PBP$\epsilon$ can be obtained by using the necessary optimality condition of problem SPBP$\epsilon$ in (\ref{eq25}).

In the following, a Fritz-John type necessary optimality(FJ) condition for problem SPBP$\epsilon$ in (\ref{eq25}) is shown, which is obtained based on the nonsmooth analytical material introduced in Section \ref{section:2}. 

\begin{theorem}[FJ condition] \label{tho:2}
Suppose Assumptions \ref{assum:1} and \ref{assum:2} hold. Let $(\bar{\boldsymbol{x}}, \bar{\boldsymbol{y}})$ be a local optimal solution to problem SPBP$\epsilon$ in (\ref{eq25}), with $\bar{\boldsymbol{x}}$ being an interior point of $\mathcal{X}$. Then there exist $\lambda_1 \in \{0, 1\}$, $\lambda_2 \ge 0$, and $\lambda_3 \ge 0$ not all zero, inteters $I$, $J$, $r_{ij}\ge 0$ satisfying $\sum_{i=1}^{I}\sum_{j=1}^J r_{ij}=1$, $\boldsymbol{y}_i \in \mathcal{R}_{\epsilon}(\bar{\boldsymbol{x}})$, $\sigma_{ij}\ge 0$, where $i\in \{1, \ldots, I\}$, $j\in \{1, \ldots, J\}$, and the definition of $\mathcal{R}_{\epsilon}(\bar{\boldsymbol{x}})$ is given in (\ref{eq24}), such that 
\begin{align*}    \nonumber
& 0 = \lambda_1 \nabla_{\boldsymbol{x}}F(\bar{\boldsymbol{x}}, \bar{\boldsymbol{y}}) + \lambda_2 \bigg(\sum_{i=1}^I \sum_{j=1}^J r_{ij}\boldsymbol{e}_{ij} - \nabla_{\boldsymbol{x}} F(\bar{\boldsymbol{x}}, \bar{\boldsymbol{y}})\bigg)  + \lambda_3(\nabla_{\boldsymbol{x}} f(\bar{\boldsymbol{x}}, \bar{\boldsymbol{y}}) - \nabla f^*(\bar{\boldsymbol{x}})),  \\ \nonumber
& 0 \in \lambda_1 \nabla_{\boldsymbol{y}}F(\bar{\boldsymbol{x}}, \bar{\boldsymbol{y}}) - \lambda_2 \nabla_{\boldsymbol{y}}F(\bar{\boldsymbol{x}}, \bar{\boldsymbol{y}}) + \lambda_3 \nabla_{\boldsymbol{y}}f(\bar{\boldsymbol{x}}, \bar{\boldsymbol{y}}) + \mathcal{N}_{\mathcal{Y}}(\bar{\boldsymbol{y}}), \\
& \lambda_3(f(\bar{\boldsymbol{x}}, \bar{\boldsymbol{y}}) - f^*(\bar{\boldsymbol{x}}) - \epsilon) = 0
\end{align*}
and for each $i\in \{1, \ldots, I\}$, $j\in \{1, \ldots, J\}$, 
\begin{align*}
& \sigma_{ij}(f(\bar{\boldsymbol{x}}, \boldsymbol{y}_i) - f^*(\bar{\boldsymbol{x}}) - \epsilon) = 0, \\
& 0 \in -\nabla_{\boldsymbol{y}}F(\bar{\boldsymbol{x}}, \boldsymbol{y}_i)+ \sigma_{ij}\nabla_{\boldsymbol{y}}f(\bar{\boldsymbol{x}}, \boldsymbol{y}_i) + \mathcal{N}_{\mathcal{Y}}(\boldsymbol{y}_i)
\end{align*}
where $\boldsymbol{e}_{ij} = \nabla_{\boldsymbol{x}} F(\bar{\boldsymbol{x}}, \boldsymbol{y}_i) - \sigma_{ij}(\nabla_{\boldsymbol{x}}f(\bar{\boldsymbol{x}}, \boldsymbol{y}_i) - \nabla f^*(\bar{\boldsymbol{x}}))$ for $i \in \{1, \ldots, I\}$, $j\in \{1, \ldots, J\}$, $\nabla f^*(\bar{\boldsymbol{x}})$ can be computed through Proposition \ref{pro:1}, and the definition of $\mathcal{N}_{\mathcal{Y}}(\bar{\boldsymbol{y}})$ is given in Proposition \ref{pro:2}.
\end{theorem}

In the following, we show that for any feasible point $(\bar{\boldsymbol{x}}, \bar{\boldsymbol{y}})$ of problem SPBP$\epsilon$ in (\ref{eq25}), with $\bar{\boldsymbol{x}}$ being an interior point of $\mathcal{X}$, the above FJ condition with $\lambda_1 = 0$ can be satisfied.

\begin{proposition}\label{pro:7}
Suppose Assumptions \ref{assum:1} and \ref{assum:2} hold. If $(\bar{\boldsymbol{x}}, \bar{\boldsymbol{y}})$ is a feasible point of problem SPBP$\epsilon$ in (\ref{eq25}) with $\bar{\boldsymbol{x}}$ being an interior point of $\mathcal{X}$, then $(\bar{\boldsymbol{x}}, \bar{\boldsymbol{y}})$ satisfies the FJ condition in Theorem \ref{tho:2} with $\lambda_1 = 0$.
\end{proposition}

Nevertheless, the FJ condition for problem SPBP$\epsilon$ is a weaker optimality condition. However, the assumption for the KKT condition(i.e., the FJ condition in Theorem \ref{tho:2} with $\lambda_1=1$) in \cite{15} is very strong, which would limit the application of the algorithms. Therefore, we still use the FJ condition in the paper. In fact, our proposed algorithm can converge to a point which satisfies the FJ condition in Theorem \ref{tho:2} with $\lambda_1=1$.

Notice that all the local optimal points for PBP$\epsilon$ in (\ref{eq4}) satisfy the necessary optimality condition in Theorem \ref{tho:2}, which can be obtained based on Assumption \ref{assum:4} and the relation between local optimal points of PBP$\epsilon$ and SPBP$\epsilon$ in (\ref{eq25}). Then, we define the following stationary points for problem PBP$\epsilon$.
\begin{definition}
Suppose Assumptions \ref{assum:1}, \ref{assum:2}, and \ref{assum:4} hold. We say that a point $\bar{\boldsymbol{x}}\in \mathcal{X}$ is a stationary point of problem PBP$\epsilon$ in (\ref{eq4}) if there exists $\bar{\boldsymbol{y}}\in \mathcal{R}_{\epsilon}(\bar{\boldsymbol{x}})$, $\lambda_1 \in \{0, 1\}$, $\lambda_2 \ge 0$, and $\lambda_3 \ge 0$ not all zero, inteters $I$, $J$, $r_{ij}\ge 0$ satisfying $\sum_{i=1}^{I}\sum_{j=1}^J r_{ij}=1$, $\boldsymbol{y}_i \in \mathcal{R}_{\epsilon}(\bar{\boldsymbol{x}})$, $\sigma_{ij}\ge 0$, where $i\in \{1, \ldots, I\}$, $j\in \{1, \ldots, J\}$, and the definition of $\mathcal{R}_{\epsilon}(\bar{\boldsymbol{x}})$ is given in (\ref{eq24}), such that the conclusion in Theorem \ref{tho:2} holds.
\end{definition}

\subsection{Convergence of PVFIM in Algorithm \ref{alg:1}}

In this subsection, we show the convergence of PVFIM in Algorithm \ref{alg:1}. Firstly, we show that for any given $\sigma \in (0,1)$, Algorithm \ref{alg:2} can find a $\sigma$-FNE point for problem (\ref{eq15}) in which $J$ is an arbitrary positive integer, and $\tau$ satisfies the inequality in (\ref{eq27}) if we choose $T$, $K$ to be the numbers that satisfy the inequalities in (\ref{eq28}).

\begin{theorem} \label{tho:3}
Suppose Assumptions \ref{assum:1}, \ref{assum:2}, \ref{assum:3}, \ref{assum:4}, \ref{assum:6}, and \ref{assum:7} hold. Let $c_0$ in Algorithm \ref{alg:1} be the number equal to $c$ in Assumption \ref{assum:7}. Assume that for any $\tau \in(0, 1)$, positive integer $J$, $\boldsymbol{x}\in \mathcal{X}$, 
\begin{equation*}
\boldsymbol{y}^*(\boldsymbol{x}):= { \underset {\boldsymbol{y}\in \mathcal{Y}_{J}(\boldsymbol{x})} { \operatorname {arg\,max} } \, G_{\epsilon, \tau, J}(\boldsymbol{x}, \boldsymbol{y})}
\end{equation*}
is an interior point of $\mathcal{Y}$, where $G_{\epsilon, \tau, J}(\boldsymbol{x}, \boldsymbol{y})$ is given in (\ref{eq16}), and $\mathcal{Y}_J(\boldsymbol{x})$ is given in (\ref{eq21}) with $c_0 = c$. Let $l_2 \in (0, 1)$. Given $\sigma \in(0, 1)$, for problem (\ref{eq15}), let $J$ be an arbitrary positive integer, and $\tau$ be a positive number satisfying
\begin{equation} \label{eq27}
\tau \le \frac{c\sigma^2}{36HJL_0\bar{l}(J)}
\end{equation}
where 
\begin{equation*} 
\bar{l}(J) = \max\bigg\{ \frac{(2H(L_{\varphi}(J)/2 + l_2) + h_0 + 4/cL_0)^2}{l_2}, l_1(J)\bigg\},
\end{equation*}
, $l_1(J) = \max\{1, L_{11}(J)H, L_G\}$, and $H$, $h_0$, $L_0$, and $c$ are given in Assumptions \ref{assum:1}, \ref{assum:3}, \ref{assum:3}, and \ref{assum:7}, respectively, and $L_{11}(J)$, $L_G$, $L_{\varphi}(J)$ are given in Lemmas \ref{lemma:5}, \ref{lemma:6}, and \ref{lemma:7} in the appendices, respectively. Choose stepsizes $\alpha$, $\beta$, $\eta$ in Algorithm \ref{alg:2} to be $\alpha = 1/L_1$, $\beta = 1/L_G$, and $\eta = 1/(L_{\varphi}(J)/2+l_2)$, where $L_1$ is given in Assumption \ref{assum:3}. Then if 
\begin{equation}   \label{eq28}
T \ge \frac{9M_3\bar{l}(J)}{\sigma^2}, \qquad  K \ge \frac{2\ln(\sigma^2/(36M\bar{l}(J)^2))}{\ln(1-\mu/L_G)}
\end{equation}
where $M_3$ is given in Lemma \ref{lemma:8} in the appendices, and $M$ and $\mu$ are given in Assumptions \ref{assum:1} and \ref{assum:2}, respectively, Algorithm \ref{alg:2} can find a $\sigma$-FNE point for problem (\ref{eq15}), that is, there exists $t \in \{0, \ldots, T-1\}$ such that $(\boldsymbol{x}_t, \boldsymbol{y}_K(\boldsymbol{x}_t))$ is a $\sigma$-FNE point of problem (\ref{eq15}).
\end{theorem}

Then, we show that Algorithm \ref{alg:1} can converge to a stationary point of problem PBP$\epsilon$ in (\ref{eq4}).

\begin{theorem} \label{tho:4}
Under the same conditions of Theorem \ref{tho:3}. Let $\{\sigma_l\}_{l=1}^{\infty}$, $\{J_l\}_{l=1}^{\infty}$, $\{\tau_l\}_{l=1}^{\infty}$ in Algorithm \ref{alg:1} be the sequences such that $\sigma_l \rightarrow 0$, $J_l \rightarrow \infty$, and for each positive integer $l$, $\sigma_l \in (0, 1)$, $\tau_l$ is the number satisfying (\ref{eq27}) with $\sigma = \sigma_l$, $J=J_l$, $T_l$, $K_l$ are the numbers satisfying the inequalities in (\ref{eq28}) with $\sigma = \sigma_l$ and $J = J_l$, stepsizes $\alpha$, $\beta$, $\eta$ are chosen according to Theorem \ref{tho:3} with $J=J_l$, and $(\boldsymbol{x}_l, \boldsymbol{y}_l)$ be the $\sigma_l$-FNE point generated by Algorithm \ref{alg:2} for solving problem (\ref{eq15}) with $\tau = \tau_l$, $J=J_l$. Then if $(\bar{\boldsymbol{x}}, \bar{\boldsymbol{y}})$ is an accumulation point of the sequence $\{(\boldsymbol{x}_l, \boldsymbol{y}_l)\}_{l=1}^{\infty}$ with $\bar{\boldsymbol{x}}$ being an interior point of $\mathcal{X}$, $\bar{\boldsymbol{x}}$ is a stationary point of problem PBP$\epsilon$ in (\ref{eq4}).
\end{theorem}

We remark that there must exist accumulation points for the sequence $\{(\boldsymbol{x}_l, \boldsymbol{y}_l)\}_{l=1}^{\infty}$ by the compactness of sets $\mathcal{X}$ and $\mathcal{Y}$, and the $\boldsymbol{x}$-component of the accumulation points can be guaranteed to be in the interior of $\mathcal{X}$ if $\mathcal{X}$ is a sufficient large region. Furthermore, let $l_0$ be a positive integer, for each positive integer $l$, let $\sigma_{l} = \bar{\sigma}_l(1 - \mu/L_G)^{{(l+ l_0)}/{2}}$, $J_l = l + l_0$, $\tau_l = (1 - \mu/L_G)^{l+l_0}$, $T_l = (1 - \mu/L_G)^{-(l+l_0)}$, and $K_l = 2(l+l_0)$, where 
\begin{equation*}   
\bar{\sigma}_l  =\max\bigg\{\frac{6(HJL_0 \bar{l}(J))^{1/2}}{c^{1/2}}, (9M_3\bar{l}(J))^{\frac{1}{2}}, 6\bar{l}(J)M^{1/2}\bigg\},   
\end{equation*}
with $J = l+ l_0$ and $\bar{\sigma}_l = \mathcal{O}((l + l_0)^8)$, and $H$, $L_0$, $c$, $L_G$, $M_3$, $M$, $\mu$ are the numbers same to that given in Theorem \ref{tho:3}. By simple calculations, it is easy to know the sequences $\{\sigma_l\}_{l=1}^{\infty}$, $\{J_l\}_{l=1}^{\infty}$, $\{\tau_l\}_{l=1}^{\infty}$, $\{T_l\}_{l=1}^{\infty}$, $\{K_l\}_{l=1}^{\infty}$ satisfy the condition in Theorem \ref{tho:4} by choosing appropriate $l_0$.

\section{Experimental Results}\label{section:5}

We first validate the theoretical results on a synthetic perturbed pessimistic bilevel problem. Then we evaluate the performance of our proposed algorithm PVFIM on the applications to GAN.

\subsection{Synthetic Perturbed Pessimistic Bilevel Problems}
In this subsection, we perform experiments on the perturbed pessimistic bilevel problem PBP$\epsilon$ in (\ref{eq4}) whose constraint sets and objective functions are given in Example \ref{example:3}. The experimental details are in the appendices.

By simple calculations, it is easy to know that this perturbed pessimistic bilevel problem PBP$\epsilon$ satisfies all the conditions in Theorem \ref{tho:4}.

Following the discussion in Section \ref{section:4}, PBP$\epsilon$ can be reformulated as the problem SPBP$\epsilon$ in (\ref{eq25}). By simple calculations, for each $x$, $\mathcal{R}_{\epsilon}(x) = \{\boldsymbol{y} = ([\boldsymbol{y}]_1, [\boldsymbol{y}]_2): [\boldsymbol{y}]_1 = x, [\boldsymbol{y}]_2 = x/2\}$($\mathcal{R}_{\epsilon}(x)$ is given in (\ref{eq24})), which is a single point set, and thus problem PBP$\epsilon$ and problem SPBP$\epsilon$ in (\ref{eq25}) are equivalent in terms of local optimal solutions(see Propositions \ref{pro:5}, \ref{pro:6}). Furthermore, the feasible points to problem SPBP$\epsilon$ in (\ref{eq25}) are $(x, \boldsymbol{y})$ with $x \in \mathcal{X}$, $[\boldsymbol{y}]_1 = x$, $[\boldsymbol{y}]_2 = x/2$ and  the local optimal points to problem SPBP$\epsilon$ in (\ref{eq25}) are $(x_1^*, \boldsymbol{y}_1^*)$ with $x_1^* = 3\pi/2$, $[\boldsymbol{y}_1^*]_1 = 3\pi/2$, $[\boldsymbol{y}_1^*]_2 = 3\pi/4$, and $(x_2^*, \boldsymbol{y}_2^*)$ with $x_2^* = 7\pi/2$, $[\boldsymbol{y}_2^*]_1 = 7\pi/2$, $[\boldsymbol{y}_2^*]_2 = 7\pi/4$, which are also global optimal points.

\begin{figure*}[!t]
\centering
\subfloat[]{\includegraphics[scale = 0.18]{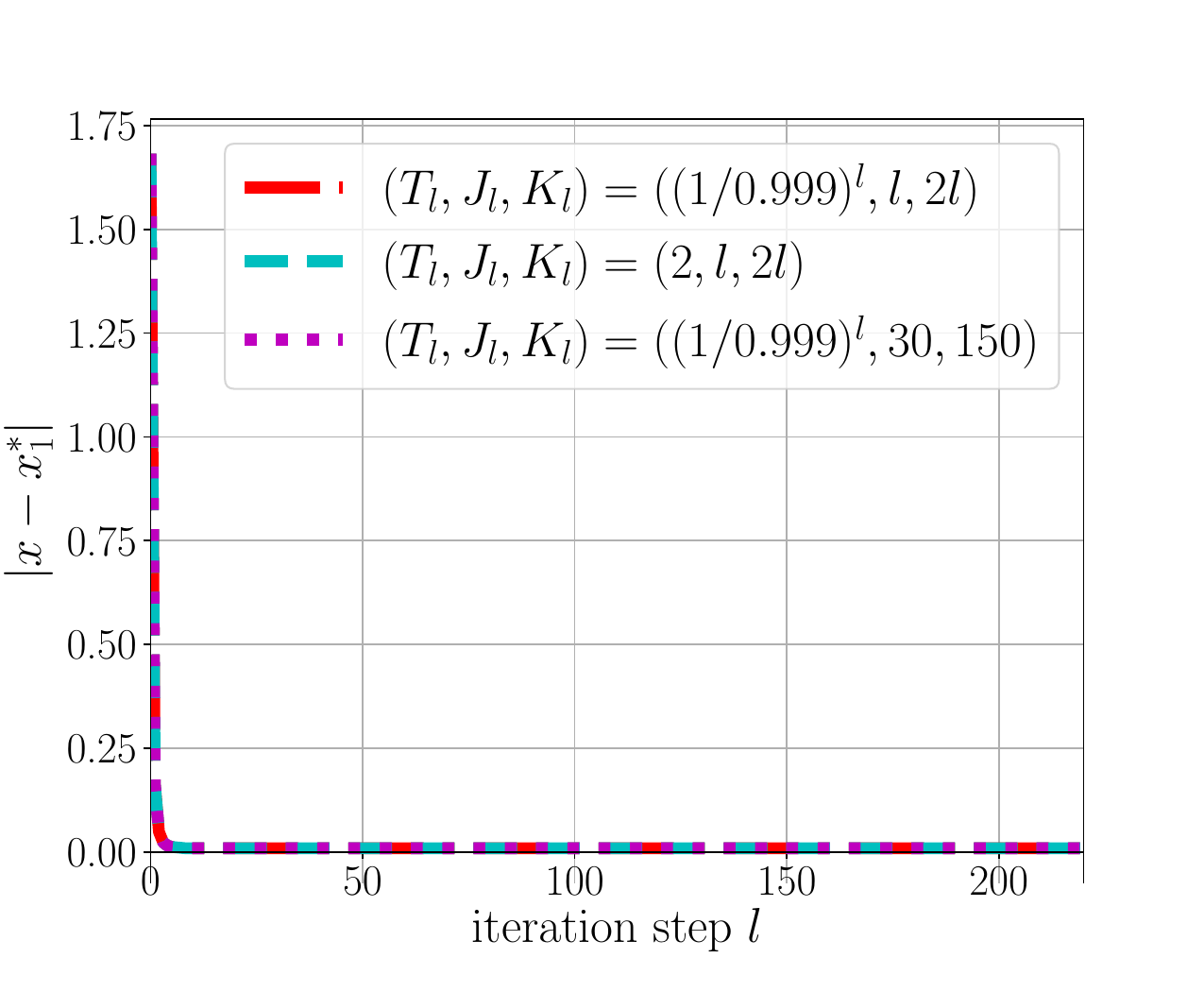}%
\label{fig1_1}}
\hfil
\subfloat[]{\includegraphics[scale = 0.18]{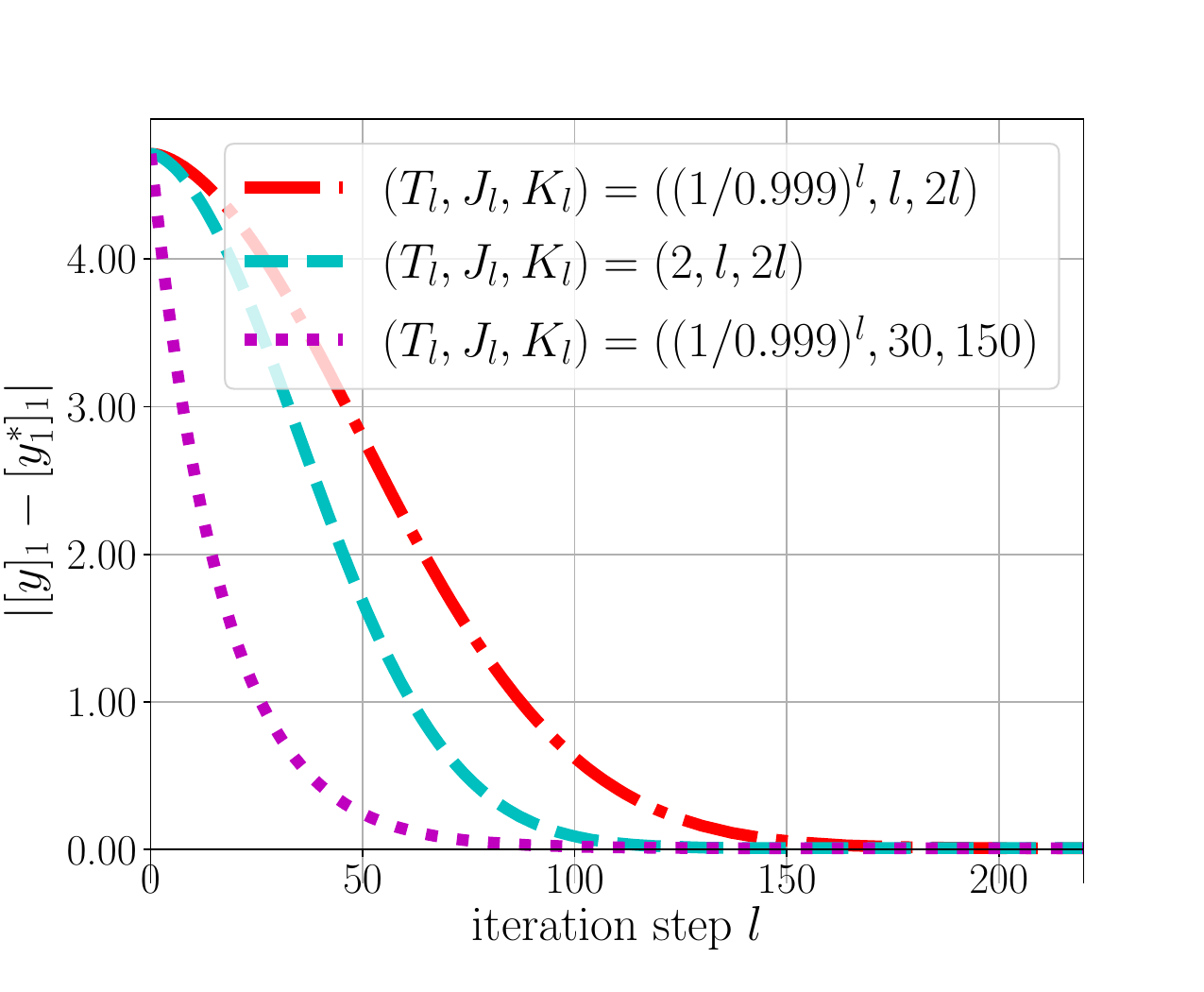}%
\label{fig1_2}}
\hfil
\subfloat[]{\includegraphics[scale = 0.18]{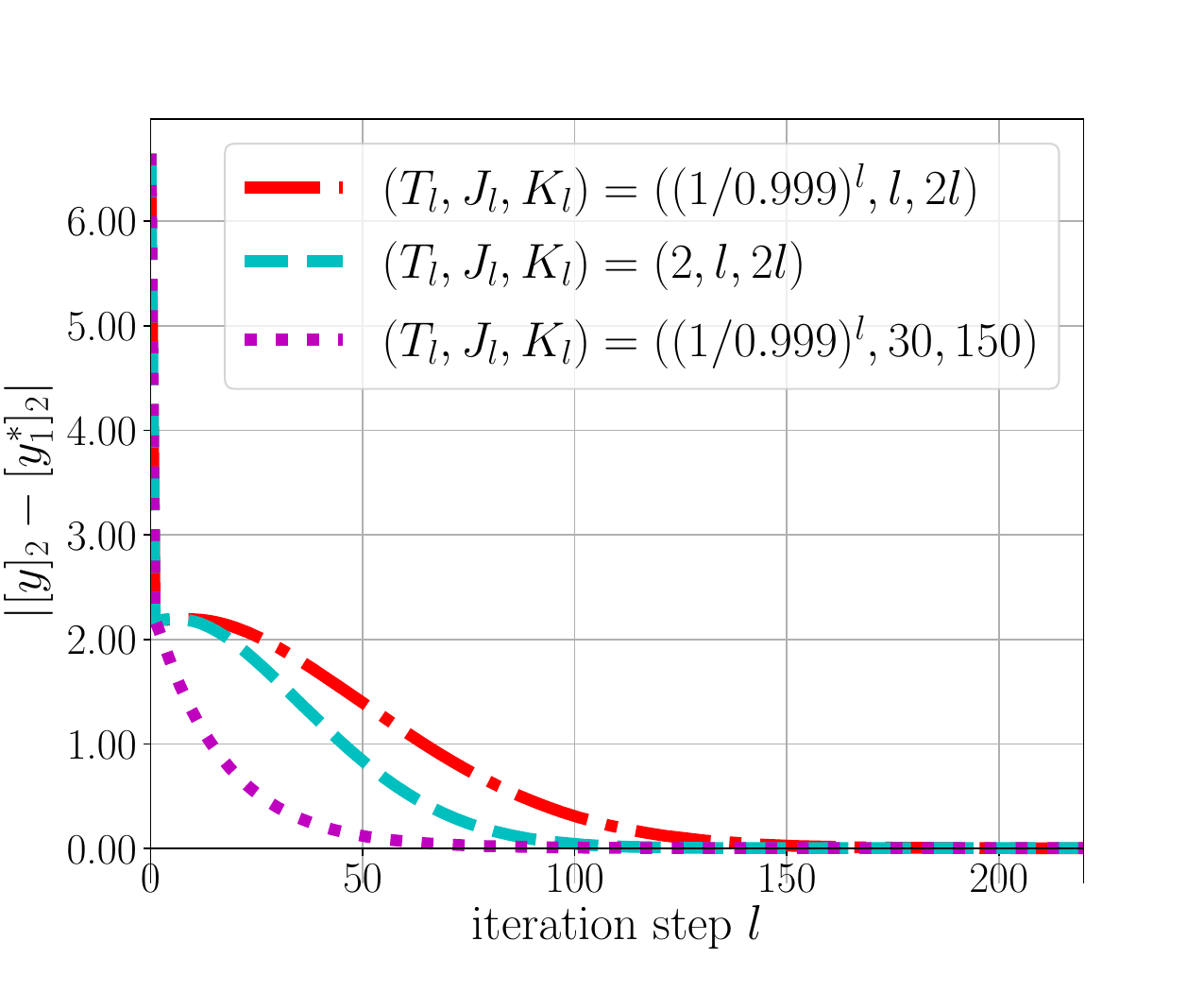}%
\label{fig1_3}}
\caption{Illustrating the convergence of PVFIM to the global optimal point of problem PBP$\epsilon$ in (\ref{eq4}) with different choices of $(T_l, J_l, K_l)$ for each positive integer $l$.}
\label{fig:1}
\end{figure*}

\begin{figure*}[!t]
\centering
\subfloat[]{\includegraphics[scale = 0.18]{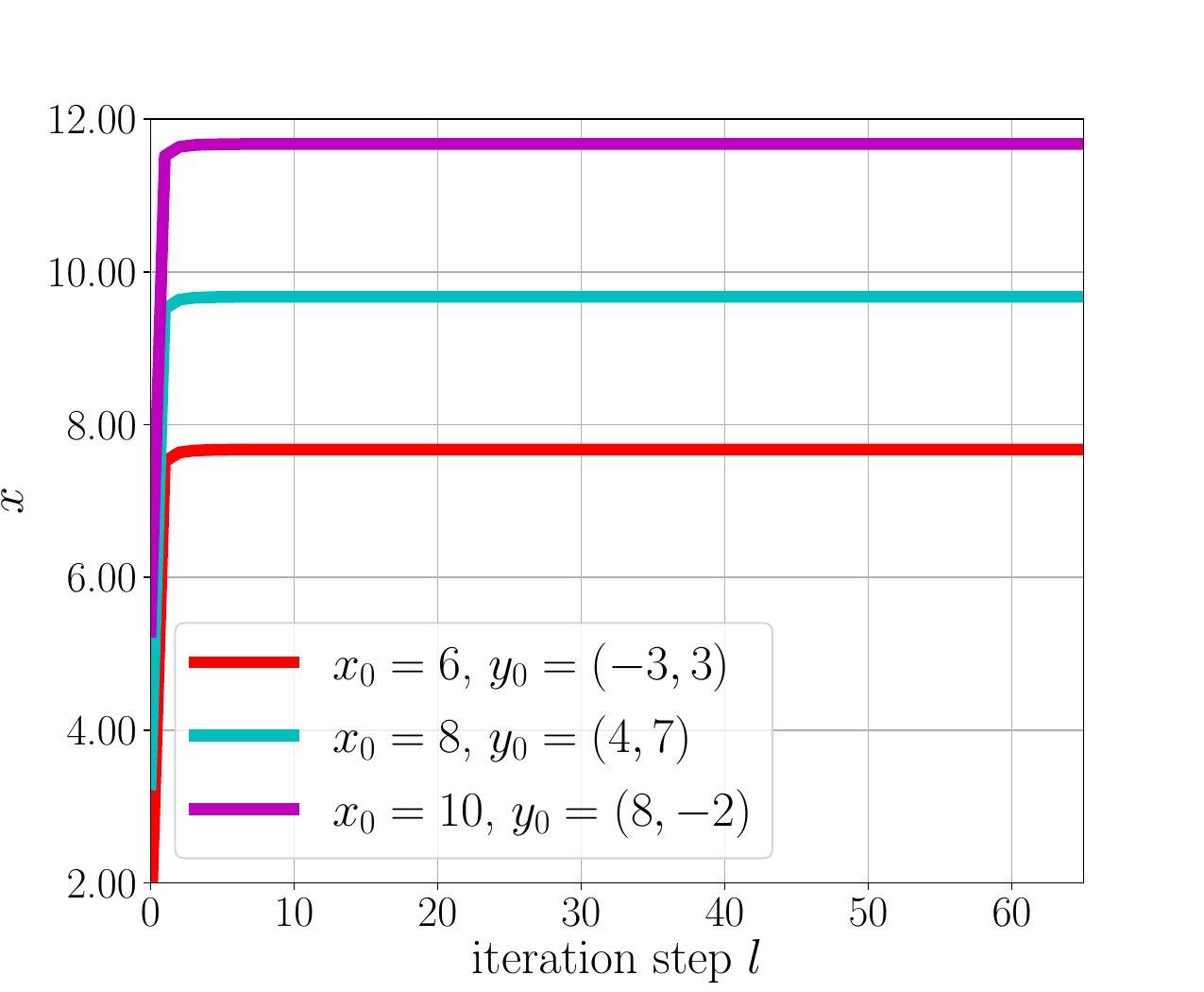}%
\label{fig2_1}}
\hfil
\subfloat[]{\includegraphics[scale = 0.18]{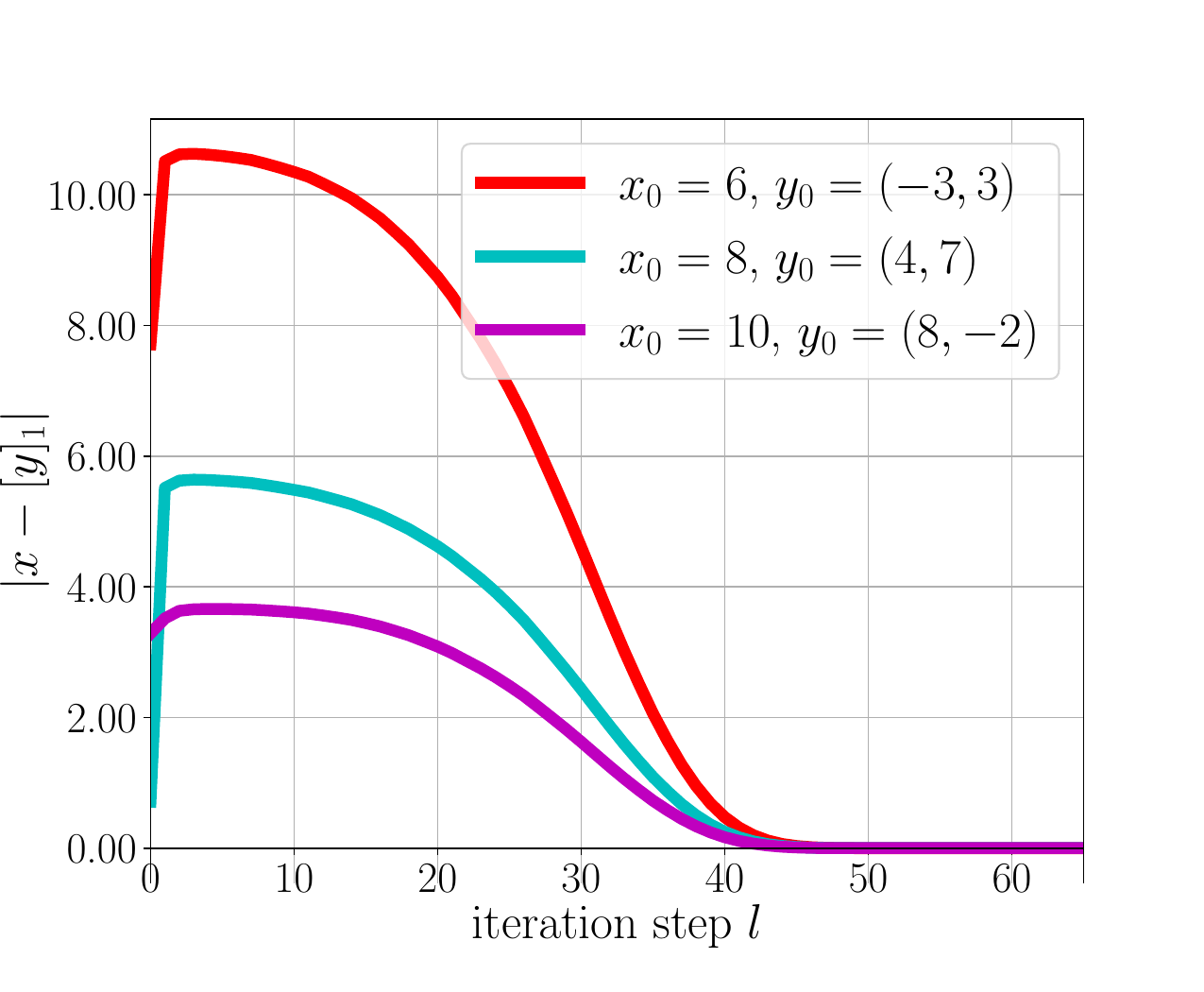}%
\label{fig2_2}}
\hfil
\subfloat[]{\includegraphics[scale = 0.18]{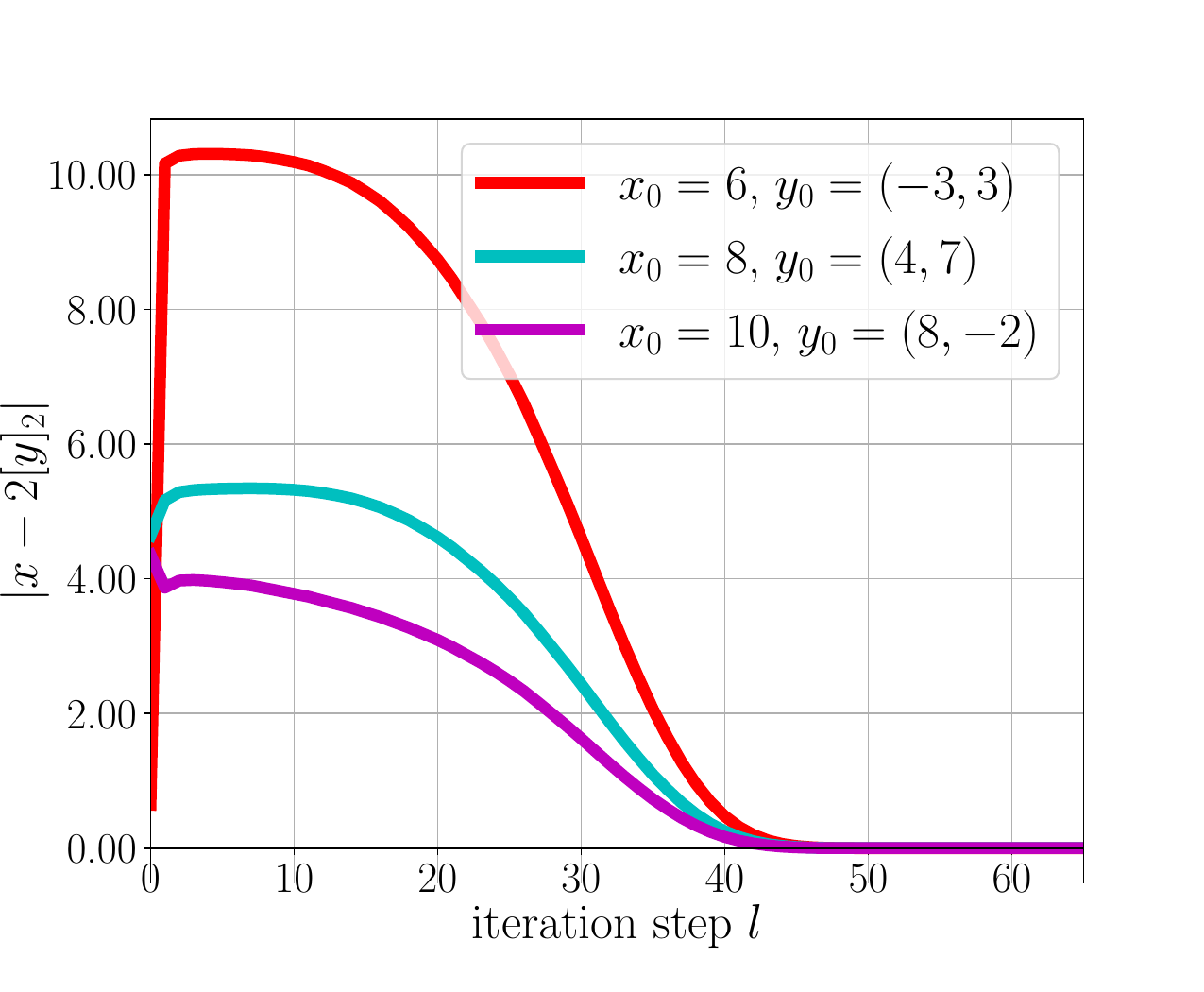}%
\label{fig2_3}}
\caption{Illustrating the numerical performance of PVFIM with different choices of initial points.}
\label{fig:2}
\end{figure*}

In Figure \ref{fig:1}, we set the initial point to be $x_0 = 3.03$, $\boldsymbol{y}_0 = (0, 9)$, and evaluate the performance of PVFIM with different choices of $(T_l, J_l, K_l)$ for each positive $l$. It shows that, for all the cases, PVFIM can converge to the optimal point $(x_1^*, \boldsymbol{y}_1^*)$ of problem SPBP$\epsilon$ in (\ref{eq25}), which naturally implies that PVFIM can converge to the optimal point $x_1^*$ of problem PBP$\epsilon$. Notice that $(T_l, J_l, K_l)=(2, l, 2l)$ and $(T_l, J_l, K_l)=((1/0.999)^l, 30, 150)$ do not follow the condition in Theorem \ref{tho:4}, and under the case that $(T_l, J_l, K_l)=((1/0.999)^l, 30, 150)$, PVFIM converges fastest among all the cases. Therefore, Figure \ref{fig:1} shows that in practice, $(T_l, J_l, K_l)$ does not have to follow the condition in Theorem \ref{tho:4}, and choosing the appropriate $(T_l, J_l, K_l)$ can accelerate the convergence of PVFIM.

In Figure \ref{fig:2}, we choose different initial points to evaluate the performance of PVFIM. It shows that PVFIM always can converge to a feasible point, whose $x$-component is an interior point of $\mathcal{X}$, of problem SPBP$\epsilon$ in (\ref{eq25}), and thus PVFIM always can converge to a stationary point of problem PBP$\epsilon$(see Proposition \ref{pro:7}).

\subsection{Generative Adversarial Networks}

In this subsection, we perform experiments on GAN to illustrate the applications of PVFIM for solving perturbed pessimistic bilevel problems. GAN generates samples from a data distribution by gaming, in which two models, i.e., a generator $G$ generating data, and a discriminator $D$ classifying the data as real or generated, are involved\cite{8}. The training objective for the discriminator is to correctly classify samples, while the training objective for the generator is to make the discriminator misclassify samples. The training objective of GAN can be expressed as the bilevel problem in (\ref{eq1})\cite{29} in which 
\begin{equation*}
F(\boldsymbol{x}, \boldsymbol{y}) = -\mathbb{E}_{\boldsymbol{u}\sim p_u}\log(D(\boldsymbol{y}; G(\boldsymbol{x}; \boldsymbol{u}))) - c_1 \|\boldsymbol{y}\|^2,
\end{equation*}
and 
\begin{align*}
f(\boldsymbol{x}, \boldsymbol{y}) &= -\mathbb{E}_{\boldsymbol{v}\sim p_{data}}\log D(\boldsymbol{y};\boldsymbol{v})  - \mathbb{E}_{\boldsymbol{u}\sim p_u}\log(1 - D(\boldsymbol{y}; G(\boldsymbol{x}; \boldsymbol{u}))) + c_2 \|\boldsymbol{x}\|^2
\end{align*}
where $\boldsymbol{x}$ denotes the parameters of the generator $G$, $\boldsymbol{y}$ denotes the parameters of the discriminator $D$, $p_{data}$ is the real data distribution, $p_u$ is the model generator distribution to be learned, $G(\boldsymbol{x}; \boldsymbol{u})$ is the data generated by the generator, $c_1$ is the coefficient of the regular term introduced to ensure the strong concavity of $F(\boldsymbol{x}, \cdot)$ w.r.t. $\boldsymbol{y}$, and $c_2$ is introduced to ensure the fully convexity of $f(\cdot, \cdot)$ w.r.t. $(\boldsymbol{x}, \boldsymbol{y})$. Notice that for the discriminator, only the parameters of the last linear layer are updated.

At present, there are mainly two methods to train GAN. The first method to train GAN is to perform gradient descent on $\boldsymbol{x}$ and ascent on $\boldsymbol{y}$ from the perspective of minimax optimization, which is done in the original GAN\cite{8}. While the second method to train GAN, done in unrolled GAN\cite{9}, is based on the unrolled optimizaiton to update $\boldsymbol{y}$ from the perspective of bilevel optimization. However, the convergence of the unrolled optimization methods for bilevel problems can be guaranteed only in the case where the solution to the LL problem is unique.

If considering GAN from the perspective of bilevel optimizaiton, we think that it would be suitable to reformulate the bilevel problem corresponding to GAN as the perturbed pessimistic bilevel problem in (\ref{eq4}) to deal with the case where there are multiple solutions in the LL problem, since the objectives of the generator and the discriminator are adversary. Therefore, PVFIM provides a method to train GAN.

In the following, we compare GAN, unrolled GAN, with our proposed method on different datasets. The experimental details are in the appendices.

Please note that in practie, the parameters $\boldsymbol{x}$ and $\boldsymbol{y}$ are located in bounded sets. Thus, the fesible regions of $\boldsymbol{x}$ and $\boldsymbol{y}$ can be considered to be large regions, and the projected operations in lines 4, 12 of Algorithm \ref{alg:2} can be ignored automatically. Furthermore, the projected operations in line 10 of Algorithm \ref{alg:2} can be ignored automatically if we choose $c_0$ in Algorithm \ref{alg:1} to be a sufficiently small number.

\subsubsection{Synthetic Data} We first evaluate the performance of our method PVFIM on a synthesized 2D dataset following \cite{9}. The real data distribution is a mixture of 5 Gaussians with standard deviation 0.02, and the probability of sampling points from each of the modes is 0.35, 0.35, 0.1, 0.1, and 0.1, respectively. The target samples are drawn from the real data distribution and the number of target samples is 512.
\begin{figure*}[!t]
\centering
\subfloat[]{\includegraphics[scale = 0.17]{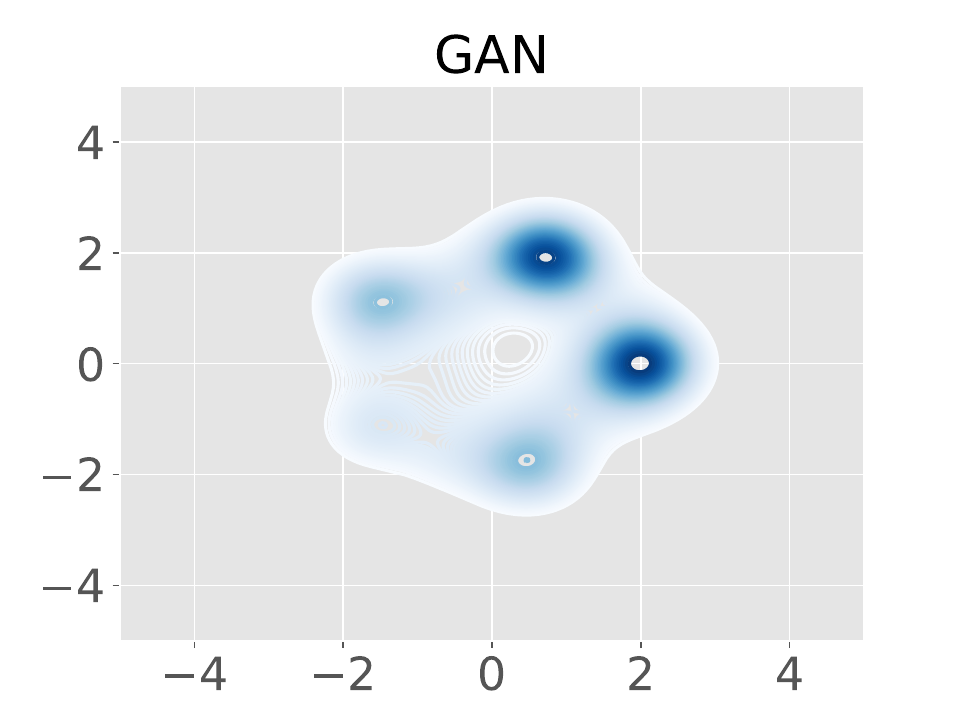}%
\label{fig3_1}}
\hfil
\subfloat[]{\includegraphics[scale = 0.17]{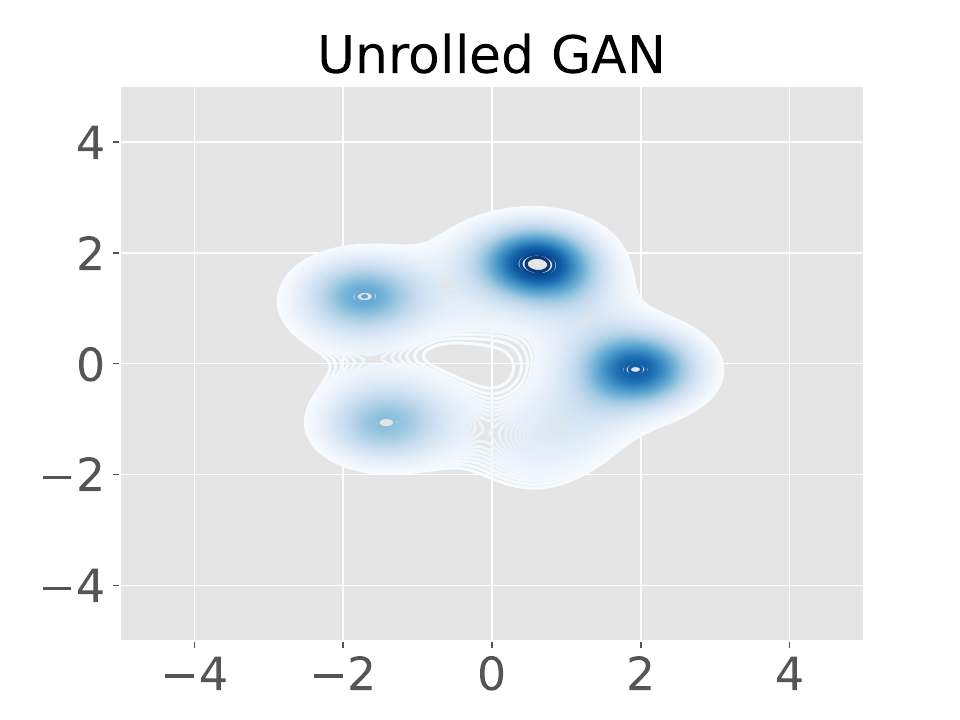}%
\label{fig3_2}}
\hfil
\subfloat[]{\includegraphics[scale = 0.17]{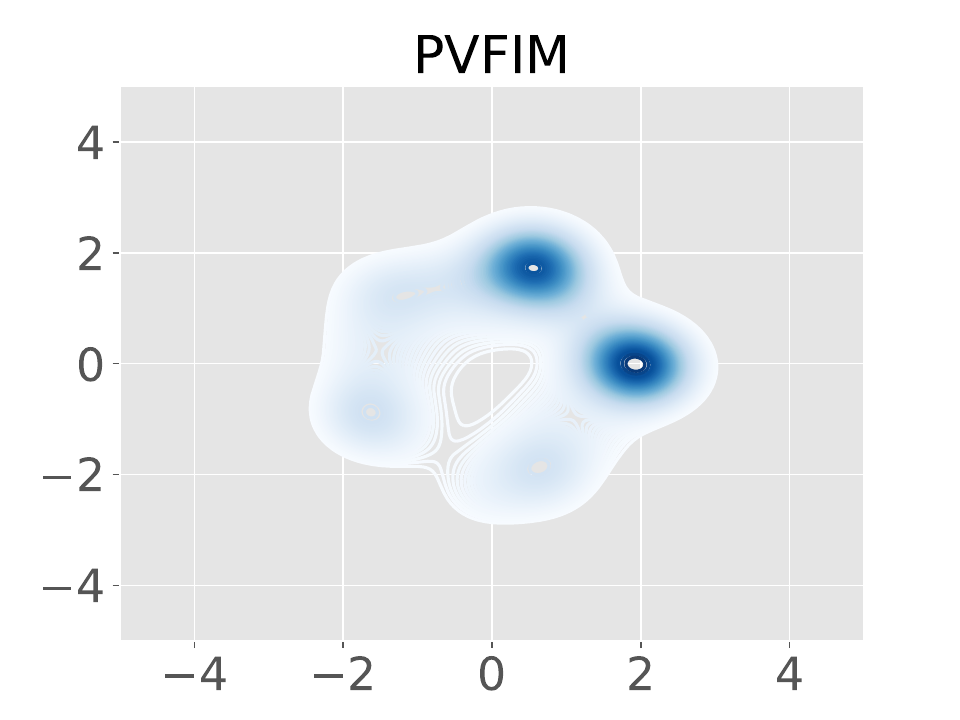}%
\label{fig3_3}}
\hfil
\subfloat[]{\includegraphics[scale = 0.17]{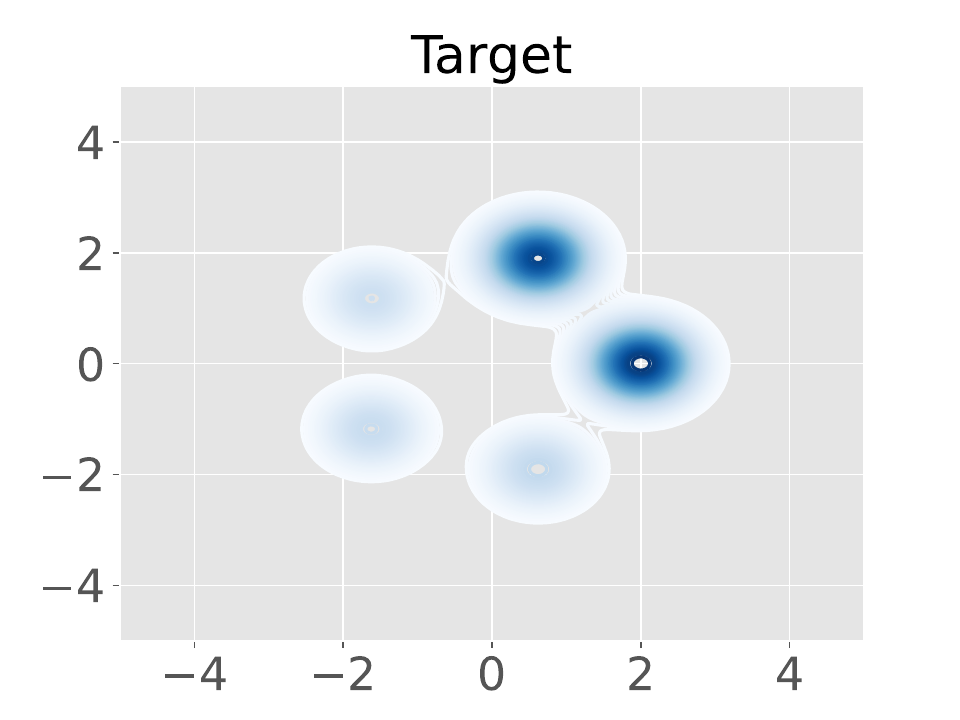}%
\label{fig_fourth_case}}
\caption{Comparison of the generated samples which have the smallest Wasserstein distance to the target samples during the iterations. (a) samples generated by GAN.(b) samples generated by Unrolled GAN. (c) samples generated by PVFIM. (d) target samples.}
\label{fig:3}
\end{figure*}

\begin{table}[htb]   
\centering
\begin{tabular}{cccc}\hline
 & GAN & Unrolled GAN & PVFIM \\  \hline
Wasserstein distance & 0.452 & 0.346 & \textbf{0.299} \\ \hline
\end{tabular}
\caption{The Wasserstin distance between the target samples and the generated samples in Fig. \ref{fig:3}.}
\label{table1}
\end{table}

From Figure \ref{fig:3} and Table \ref{table1}, we observe that PVFIM can generate samples which has the smallest Wasserstein distance to the target samples, which shows the validality of PVFIM on the application to GAN with synthetic data. Furthermore, since PVFIM can be used to solve the bilevel problem with multiple solutions in the lower level problem, it obtains better performance than unrolled GAN in the experiments. 

\subsubsection{Real-World Data} We then perform experiments on two real-world datasets: MNIST\cite{30} and CIFAR10, to evaluate the performance of PVFIM. The MNIST dataset consists of labeled images of $28 \times 28$ grayscale digits, and the CIFAR10 dataset is a natural scene dataset of $32 \times 32$. Inception Score(IS) is employed to evaluate the quality and diversity of the generated images, and the Frechet Inception Distance score(FID) is used for measuring the Frechet distance between the real and generated data distributions.

\begin{table}[htb]   
\centering
\begin{tabular}{ccccc}\hline
\multicolumn{1}{c}{\multirow{2}{*}{Method}} & \multicolumn{2}{c}{MNIST} & \multicolumn{2}{c}{CIFAR10} \\ 
 & IS & FID & IS& FID \\  \hline
GAN & 6.05 & 189 &2.89  &241  \\
Unrolled GAN &4.60  &250 &2.88  & 264\\ 
PVFIM & \textbf{6.32} & \textbf{181} &\textbf{2.90} & \textbf{223}  \\  \hline
\end{tabular}
\caption{Comparison of the best FID and IS during the iterations.}
\label{table2}
\end{table}

\begin{figure*}[!t]
\centering
\subfloat[]{\includegraphics[scale = 0.17]{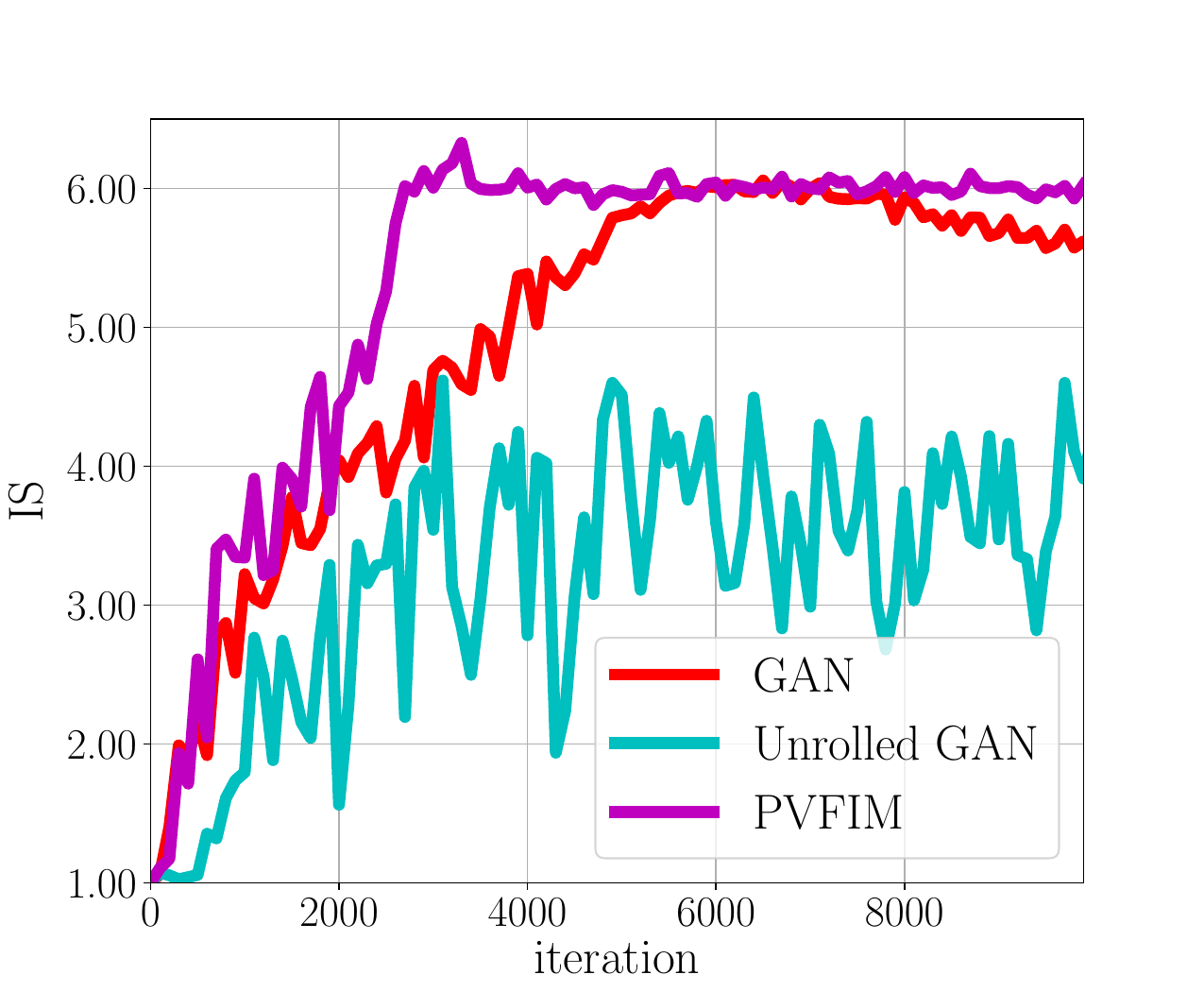}%
\label{fig4_1}}
\hfil
\subfloat[]{\includegraphics[scale = 0.17]{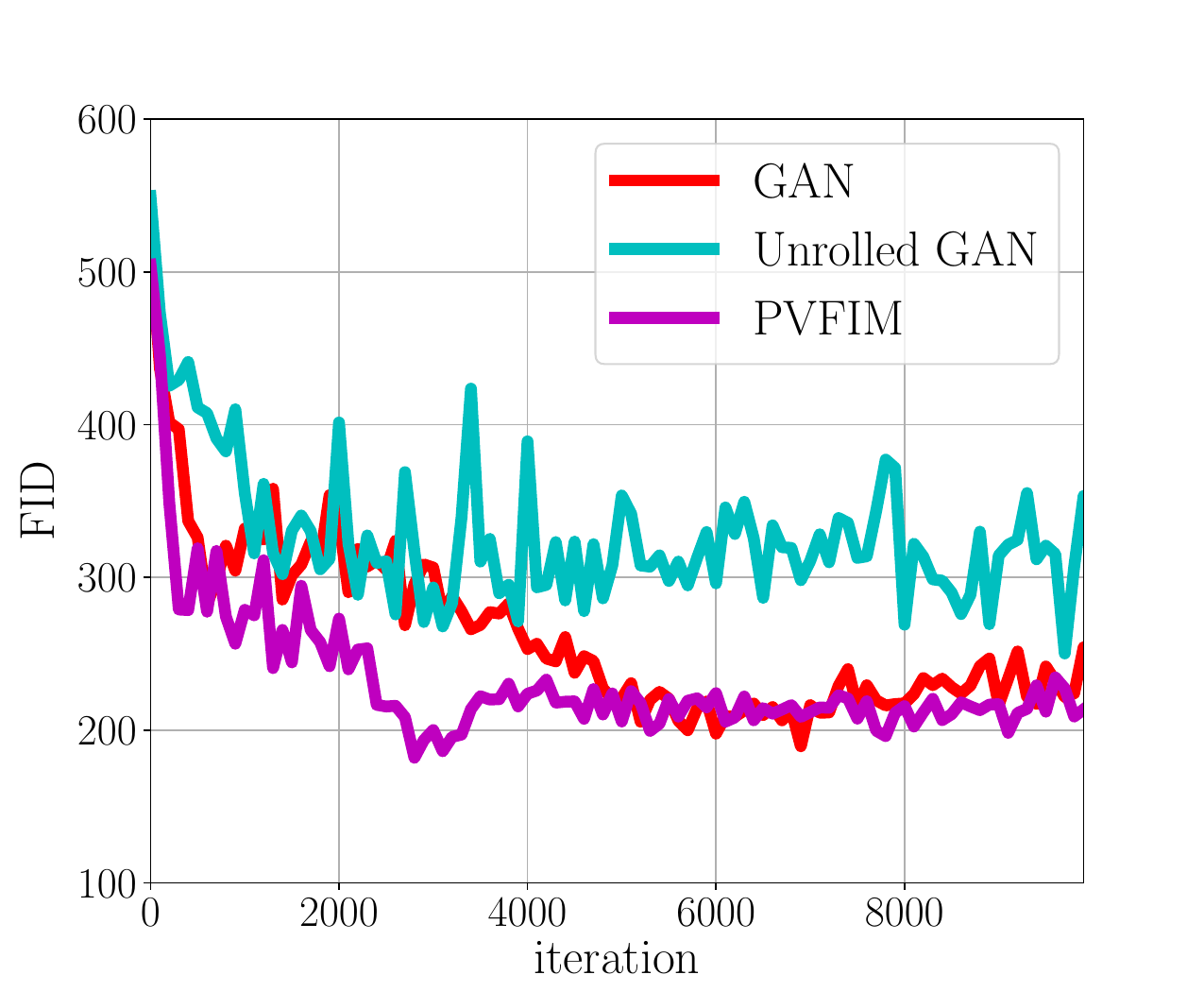}%
\label{fig4_2}}\\
\subfloat[]{\includegraphics[scale = 0.17]{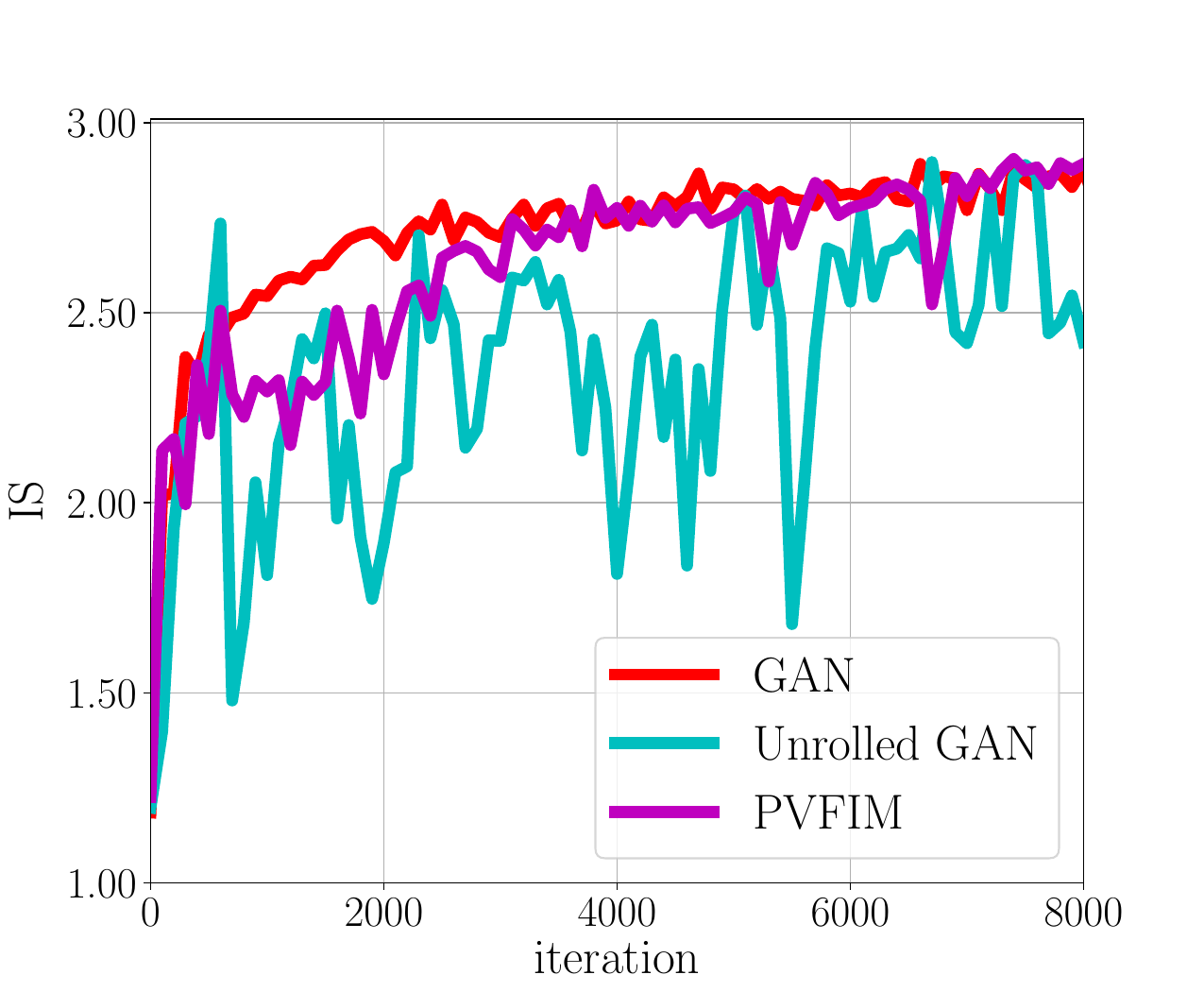}%
\label{fi4_3}}
\hfil
\subfloat[]{\includegraphics[scale = 0.17]{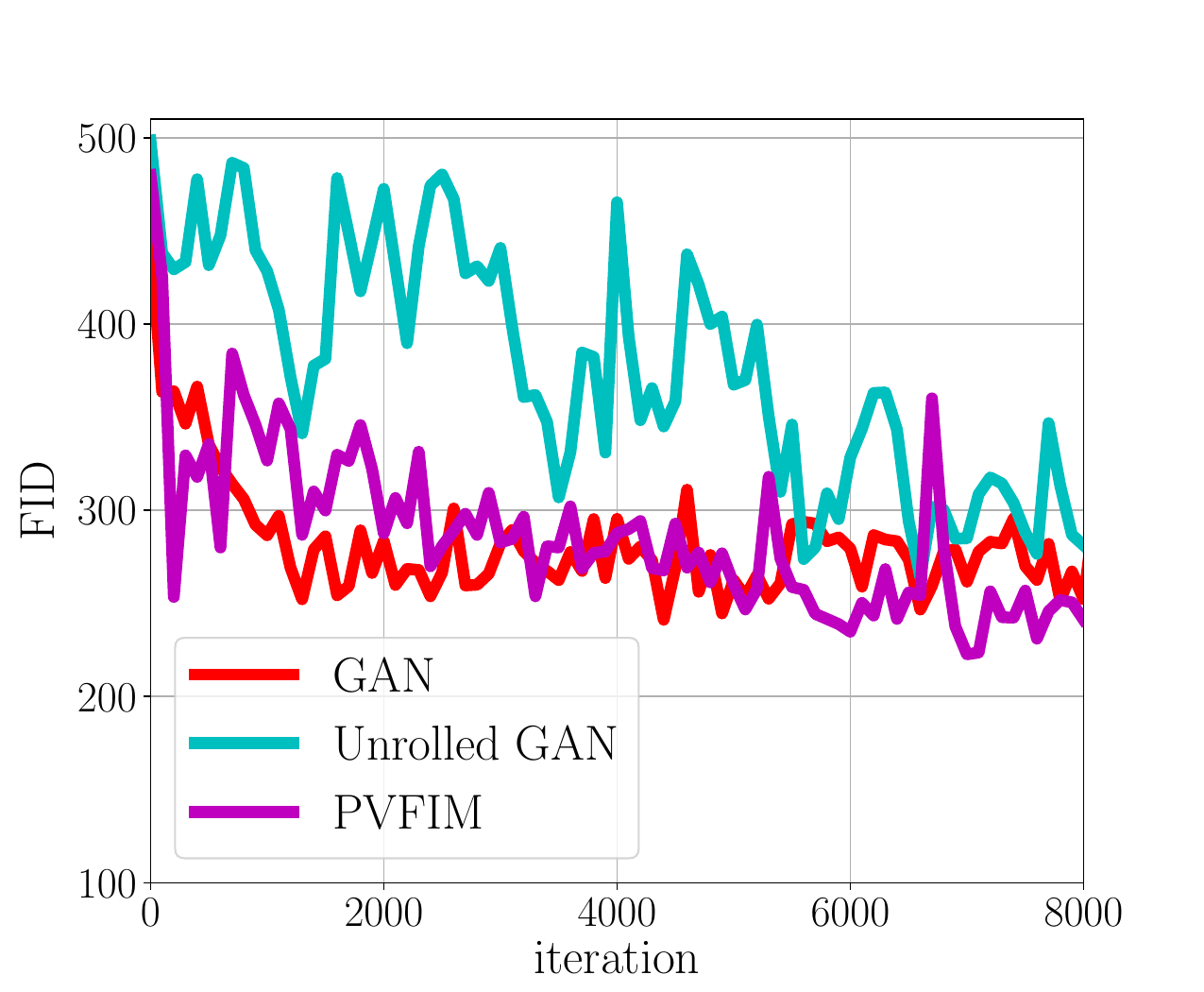}%
\label{fig4_4}}
\caption{Comparison of the IS(the higher, the better) and FID(the smaller, the better) over the iterations. (a), (b) present the results on the MNIST dataset, and (c), (d) present the results on the CIFAR10 dataset.}
\label{fig:4}
\end{figure*}

Figure \ref{fig:4} and Table \ref{table2} show that PVFIM can obtain the lowest FID and highest IS on both MNIST and CIFAR10 datasets, which shows the validality of PVFIM on the application to GAN with real-world data. Furthermore, since PVFIM can be used to solve the bilevel problem with multiple solutions in the lower level problem, it obtains better performance than unrolled GAN in the experiments.

\section{Conclusion} \label{section:6}
In the paper, we provide the first provably convergent algorithm for the perturbed pessimistic bilevel problem PBP$\epsilon$ with nonlinear lower level problem, and we provide a stationary condition for problem PBP$\epsilon$, which has not been given before. Moreover, the experiments are presented to validate the theoretical results and show the potential of our proposed algorithm in the applications to GAN.

\section*{Acknowledgments}
The work is partially supported by the National Key R\&D Program of China (2018AAA0100300 and 2020YFB1313503), National Natural Science Foundation of China(Nos. 61922019, 61733002, 61672125, and 61976041), and LiaoNing Revitalization Talents Program(XLYC1807088).

\appendix  
\section{Proof of Theorem \ref{tho:2} and Proposition \ref{pro:7}}  \label{appendix:1}

In this section, we first obtain a Fritz-John type necessary optimality condition for problem SPBP$\epsilon$ in (\ref{eq25}).

It needs to remark that although $F(\boldsymbol{x}, \boldsymbol{y})$, $f(\boldsymbol{x}, \boldsymbol{y})$, and $f^*(\boldsymbol{x})$ are differentiable under Assumptions \ref{assum:1} and \ref{assum:2}, where the differentiability of $f^*(\boldsymbol{x})$ can be obtained from Proposition \ref{pro:1}, $F^*(\boldsymbol{x})$ defined in (\ref{eq26}) is generally nonsmooth. As a result, problem SPBP$\epsilon$ in (\ref{eq25}) is a nonsmooth problem in general. Recall that if $F^*(\boldsymbol{x})$ is local Lipschitz continuous, a Fritz-John type necessary optimality condition for SPBP$\epsilon$ in (\ref{eq25}) can be obtained following from Theorem \ref{tho:1}. For more details, see the nonsmooth analysis material given in the main text.

In the following, we first investigate the local Lipschitz continuity of $F^*(\boldsymbol{x})$. Notice that by using the definition of $\mathcal{S}_{\epsilon}(\boldsymbol{x})$ in (\ref{eq10}), $F^*(\boldsymbol{x})$ in (\ref{eq26}) can be equivalently converted into 
\begin{equation} \label{eq29} \tag{29}
F^*(\boldsymbol{x}):= \max\limits_{\boldsymbol{y}}\{ F(\boldsymbol{x}, \boldsymbol{y}): \boldsymbol{y}\in \mathcal{Y}, f(\boldsymbol{x}, \boldsymbol{y}) \le f^*(\boldsymbol{x}) + \epsilon \}.
\end{equation}
Similar to the discussion in the nonsmooth analysis material of the main text, given $\boldsymbol{x}\in \mathcal{X}$, for $\boldsymbol{y}$ satisfying $\boldsymbol{y}\in \mathcal{Y}$, $f(\boldsymbol{x}, \boldsymbol{y}) \le f^*(\boldsymbol{x})+ \epsilon$, we define
\begin{align}  \label{eq30} \nonumber   
\mathcal{M}_{\boldsymbol{x}}^{\lambda}(\boldsymbol{y}) :=& \left \{ \sigma: 0 \in -\lambda \nabla_{\boldsymbol{y}}F(\boldsymbol{x}, \boldsymbol{y}) + \sigma \nabla_{\boldsymbol{y}} f(\boldsymbol{x}, \boldsymbol{y}) + \mathcal{N}_{\mathcal{Y}}(\boldsymbol{y}), \right. \\   \tag{30}
& \left. \qquad \sigma\ge 0, \qquad \sigma (f(\boldsymbol{x}, \boldsymbol{y})-f^*(\boldsymbol{x})-\epsilon)=0 \right \}
\end{align}
where $\lambda \in \{0, 1\}$, and the definition of $\mathcal{N}_{\mathcal{Y}}(\boldsymbol{y})$ is given in Proposition \ref{pro:2}. Then a sufficient condition for the local Lipschitz continuity of $F^*(\boldsymbol{x})$ is shown below.
\begin{lemma} \label{lemma:1}
Suppose Assumptions \ref{assum:1} and \ref{assum:2} hold. For problem SPBP$\epsilon$ in (\ref{eq25}), if $\boldsymbol{x}_0$ is an interior point of $\mathcal{X}$, then $F^*(\boldsymbol{x})$ in (\ref{eq26}) is Lipschitz continuous near $\boldsymbol{x}_0$, and the Clarke generalized gradient of $F^*(\boldsymbol{x})$ at $\boldsymbol{x}_0$ satisfies
\begin{equation*}
\partial F^*(\boldsymbol{x}_0) \subset \text{co}\mathcal{A}(\boldsymbol{x}_0)
\end{equation*}
where 
\begin{equation*}
\mathcal{A}(\boldsymbol{x}_0) := \mathop{\cup}\limits_{\boldsymbol{y}\in \mathcal{R}_{\epsilon}(\boldsymbol{x}_0)}\{ \nabla_{\boldsymbol{x}}F(\boldsymbol{x}_0, \boldsymbol{y}) - \sigma(\nabla_{\boldsymbol{x}}f(\boldsymbol{x}_0, \boldsymbol{y})- \nabla f^*(\boldsymbol{x}_0)): \sigma \in \mathcal{M}_{\boldsymbol{x}_0}^{1}(\boldsymbol{y})\}
\end{equation*}
with the definitions of $\mathcal{R}_{\epsilon}(\boldsymbol{x}_0)$ and $\mathcal{M}_{\boldsymbol{x_0}}^{1}(\boldsymbol{y})$ given in (\ref{eq24}) and (\ref{eq30}) with $\boldsymbol{x}=\boldsymbol{x}_0$ and $\lambda = 1$, respectively, and $\nabla f^*(\boldsymbol{x}_0)$ can be computed through Proposition \ref{pro:1}.
\end{lemma}

\begin{proof}
Define 
\begin{equation*}
\bar{F}(\boldsymbol{x}):= \min\limits_{\boldsymbol{y}}\{ -F(\boldsymbol{x}, \boldsymbol{y}): \boldsymbol{y}\in \mathcal{Y}, f(\boldsymbol{x}, \boldsymbol{y}) \le f^*(\boldsymbol{x}) + \epsilon \}.
\end{equation*}
It is obvious that $F^*(\boldsymbol{x}) = - \bar{F}(\boldsymbol{x})$ on $\mathcal{X}$ by the equivalent definition of $F^*(\boldsymbol{x})$ given in (\ref{eq29}).

We first prove the Lipschitz continuity of $\bar{F}(\boldsymbol{x})$ near $\boldsymbol{x}_0$, which further implies the Lipschitz continuity of $F^*(\boldsymbol{x})$ near $\boldsymbol{x}_0$. To start with, we define 
\begin{equation} \label{eq31} \tag{31}
\bar{\mathcal{R}}_{\epsilon}(\boldsymbol{x}_0) := { \underset {\boldsymbol{y}} { \operatorname {arg\,min} } \, \{-F(\boldsymbol{x}_0, \boldsymbol{y}): \boldsymbol{y} \in \mathcal{Y}, f(\boldsymbol{x}_0, \boldsymbol{y})\le f^*(\boldsymbol{x}_0) + \epsilon \}}.
\end{equation}
Notice that $\boldsymbol{x}_0$ is an interior point of $\mathcal{X}$. From Proposition \ref{pro:4}, we know that the Lipschitz continuity of $\bar{F}(\boldsymbol{x})$ near $\boldsymbol{x}_0$ can be guaranteed if $\mathcal{M}_{\boldsymbol{x}_0}^{0}(\boldsymbol{y})=\{0\}$ for all $\boldsymbol{y}\in \bar{\mathcal{R}}_{\epsilon}(\boldsymbol{x}_0)$, where the definition of $\mathcal{M}_{\boldsymbol{x}_0}^{0}(\boldsymbol{y})$ is given in (\ref{eq30}) with $\boldsymbol{x}=\boldsymbol{x}_0$ and $\lambda = 0$.

Then, in the following, we prove that $\mathcal{M}_{\boldsymbol{x}_0}^{0}(\boldsymbol{y}) = \{0\}$ for all $\boldsymbol{y} \in \bar{\mathcal{R}}_{\epsilon}(\boldsymbol{x}_0)$. Recall that for any $\boldsymbol{y}_0 \in \bar{\mathcal{R}}_{\epsilon}(\boldsymbol{x}_0)$, 

\begin{align*}  
\mathcal{M}_{\boldsymbol{x}_0}^{0}(\boldsymbol{y}_0) :=& \left \{ \sigma: 0 \in  \sigma \nabla_{\boldsymbol{y}} f(\boldsymbol{x}_0, \boldsymbol{y}_0) + \mathcal{N}_{\mathcal{Y}}(\boldsymbol{y}_0), \right. \\   
& \left. \qquad \sigma\ge 0, \qquad \sigma (f(\boldsymbol{x}_0, \boldsymbol{y}_0)-f^*(\boldsymbol{x}_0)-\epsilon)=0 \right \}.
\end{align*}
Assume that there exists $\sigma>0$ such that 
\begin{equation*}
0 \in \sigma \nabla_{\boldsymbol{y}}f(\boldsymbol{x}_0, \boldsymbol{y}_0) + \mathcal{N}_{\mathcal{Y}}(\boldsymbol{y}_0), \qquad \sigma(f(\boldsymbol{x}_0, \boldsymbol{y}_0) - f^*(\boldsymbol{x}_0) - \epsilon) = 0.
\end{equation*}
Then, we must have 
\begin{equation}  \label{eq32} \tag{32}
-\nabla_{\boldsymbol{y}}f(\boldsymbol{x}_0, \boldsymbol{y}_0) \in \mathcal{N}_{\mathcal{Y}}(\boldsymbol{y}_0),
\end{equation}
and 
\begin{equation}  \label{eq33} \tag{33}
f(\boldsymbol{x}_0, \boldsymbol{y}_0) - f^*(\boldsymbol{x}_0) - \epsilon = 0.
\end{equation}

From (\ref{eq32}) and the definition of $\mathcal{N}_{\mathcal{Y}}(\boldsymbol{y}_0)$ in Proposition \ref{pro:2}, we know that $\langle \nabla_{\boldsymbol{y}} f(\boldsymbol{x}_0, \boldsymbol{y}_0), \boldsymbol{y} - \boldsymbol{y}_0 \rangle \ge 0$, $\forall \boldsymbol{y}\in \mathcal{Y}$, which further implies that $f(\boldsymbol{x}_0, \boldsymbol{y}_0) = \min_{\boldsymbol{y}}\{ f(\boldsymbol{x}_0, \boldsymbol{y}): \boldsymbol{y}\in \mathcal{Y}\}$ since $f(\boldsymbol{x}_0, \cdot)$ is convex on $\mathcal{Y}$(see Assumption \ref{assum:2}). Therefore, we have $f(\boldsymbol{x}_0, \boldsymbol{y}_0) = f^*(\boldsymbol{x}_0)$ by the definition of $f^*(\boldsymbol{x})$ in (\ref{eq5}). However, it contradicts with the formula in (\ref{eq33}), since $\epsilon$ is a positive number. Therefore, $\mathcal{M}_{\boldsymbol{x}_0}^{0}(\boldsymbol{y}) = \{0\}$ for all $\boldsymbol{y}\in \bar{\mathcal{R}}_{\epsilon}(\boldsymbol{x}_0)$. As a result, both $\bar{F}(\boldsymbol{x})$ and $F^*(\boldsymbol{x})$ are Lipschitz continuous near $\boldsymbol{x}_0$.

Next, we consider the Clarke generalized gradient of $F^*(\boldsymbol{x})$ at $\boldsymbol{x}_0$, i.e., $\partial F^*(\boldsymbol{x}_0)$.

Recall that $F^*(\boldsymbol{x}_0) = -\bar{F}(\boldsymbol{x}_0)$. From Proposition \ref{pro:2}, we know that 
\begin{equation} \label{eq34} \tag{34}
\partial F^*(\boldsymbol{x}_0) = \text{co} \hat{\partial}(-\bar{F}(\boldsymbol{x}_0))= -\text{co} \hat{\partial}\bar{F}(\boldsymbol{x}_0).
\end{equation}
Since $\boldsymbol{x}_0$ is an interior point of $\mathcal{X}$, and $\mathcal{M}_{\boldsymbol{x}_0}^{0}(\boldsymbol{y}) = \{0\}$ for all $\boldsymbol{y}\in \bar{\mathcal{R}}_{\epsilon}(\boldsymbol{x}_0)$, from Proposition \ref{pro:4}, $\hat{\partial}\bar{F}(\boldsymbol{x}_0)$ satisfies 
\begin{equation*}
\hat{\partial}\bar{F}(\boldsymbol{x}_0)\subset \mathop{\cup}\limits_{\boldsymbol{y}\in \bar{\mathcal{R}}_{\epsilon}(\boldsymbol{x}_0)} \{ -\nabla_{\boldsymbol{x}}F(\boldsymbol{x}_0, \boldsymbol{y}) + \sigma(\nabla_{\boldsymbol{x}}f(\boldsymbol{x}_0, \boldsymbol{y}) - \nabla f^*(\boldsymbol{x}_0)): \sigma \in \mathcal{M}_{\boldsymbol{x}_0}^{1}(\boldsymbol{y})\}
\end{equation*}
where the definitions of $\mathcal{M}_{\boldsymbol{x}_0}^{1}(\boldsymbol{y})$ and $\bar{\mathcal{R}}_{\epsilon}(\boldsymbol{x}_0)$ are given in (\ref{eq30}) and (\ref{eq31}) with $\boldsymbol{x} = \boldsymbol{x}_0$ and $\lambda = 1$, respectively, and $\nabla f^*(\boldsymbol{x}_0)$ can be computed via Proposition \ref{pro:1}.

Furthermore, by the definition of $\mathcal{A}(\boldsymbol{x}_0)$ and the fact that $\bar{\mathcal{R}}_{\epsilon}(\boldsymbol{x}_0) = \mathcal{R}_{\epsilon}(\boldsymbol{x}_0)$, we know that $\hat{\partial}F(\boldsymbol{x}_0)\subset - \mathcal{A}(\boldsymbol{x}_0)$. Combined with (\ref{eq34}), that $\partial F^*(\boldsymbol{x}_0)\subset \text{co}\mathcal{A}(\boldsymbol{x}_0)$ can be obtained. Then, the proof is complete.

\end{proof}

\subsection{Proof of Theorem \ref{tho:2}}
In the following, the proof of Theorem \ref{tho:2} is shown.

\begin{proof}
Firstly, since $\bar{\boldsymbol{x}}$ is an interior point of $\mathcal{X}$, from Lemma \ref{lemma:1}, we know that $F^*(\boldsymbol{x})$ in (\ref{eq26}) is Lipschitz continuous near $\bar{\boldsymbol{x}}$, and the Clarke generalized gradient satisfies $\partial F^*(\bar{\boldsymbol{x}})\subset \text{co}\mathcal{A}(\bar{\boldsymbol{x}})$, i.e., there exist positive integers $I$, $J$, $r_{ij}\ge 0$ satisfying $\sum_{i=1}^{I}\sum_{j=1}^J r_{ij}=1$, $\boldsymbol{y}_i \in \mathcal{R}_{\epsilon}(\bar{\boldsymbol{x}})$, $\sigma_{ij}\in \mathcal{M}_{\bar{\boldsymbol{x}}}^{1}(\boldsymbol{y}_i)$ with $\mathcal{M}_{\bar{\boldsymbol{x}}}^1(\boldsymbol{y}_i)$ given in (\ref{eq30}) with $\boldsymbol{x} = \bar{\boldsymbol{x}}$, $\lambda = 1$, where $i\in \{1, \ldots, I\}$, $j\in \{1, \ldots, J\}$, such that 
\begin{equation}  \label{eq35} \tag{35}
\partial F^*(\bar{\boldsymbol{x}}) = \sum\limits_{i=1}^{I}\sum\limits_{j=1}^{J}r_{ij}(\nabla_{\boldsymbol{x}}F(\bar{\boldsymbol{x}}, \boldsymbol{y}_i) - \sigma_{ij}(\nabla_{\boldsymbol{x}}f(\bar{\boldsymbol{x}}, \boldsymbol{y}_i)- \nabla f^*(\bar{\boldsymbol{x}})))
\end{equation}
where we use the definition of $\mathcal{A}(\bar{\boldsymbol{x}})$ given in Lemma \ref{lemma:1} and the definition of convex hull.

Then, since $(\bar{\boldsymbol{x}}, \bar{\boldsymbol{y}})$ is a local optimal solution to problem SPBP$\epsilon$ in (\ref{eq25}), from Theorem \ref{tho:1}, we know that there exist $\lambda_1 \in \{0, 1\}$, $\lambda_2 \ge 0$, $\lambda_3 \ge 0$ not all zero such that 
\begin{align} \label{eq36}  \nonumber
&0 = \lambda_1 \nabla_{\boldsymbol{x}}F(\bar{\boldsymbol{x}}, \bar{\boldsymbol{y}}) + \lambda_2(\partial F^*(\bar{\boldsymbol{x}}) - \nabla_{\boldsymbol{x}}F(\bar{\boldsymbol{x}}, \bar{\boldsymbol{y}})) + \lambda_3(\nabla_{\boldsymbol{x}}f(\bar{\boldsymbol{x}}, \bar{\boldsymbol{y}}) - \nabla f^*(\bar{\boldsymbol{x}})), \\  \tag{36}
&0 \in \lambda_1 \nabla_{\boldsymbol{y}}F(\bar{\boldsymbol{x}}, \bar{\boldsymbol{y}}) - \lambda_2 \nabla_{\boldsymbol{y}}F(\bar{\boldsymbol{x}}, \bar{\boldsymbol{y}}) + \lambda_3\nabla_{\boldsymbol{y}}f(\bar{\boldsymbol{x}}, \bar{\boldsymbol{y}}) +\mathcal{N}_{\mathcal{Y}}(\bar{\boldsymbol{y}}),\\
& \lambda_3 (f(\bar{\boldsymbol{x}}, \bar{\boldsymbol{y}}) - f^*(\bar{\boldsymbol{x}}) - \epsilon) = 0   \nonumber
\end{align}
where $\partial F^*(\bar{\boldsymbol{x}})$ is given in (\ref{eq35}), $\nabla f^*(\bar{\boldsymbol{x}})$ can be computed through Proposition \ref{pro:1}, the definition of $\mathcal{N}_{\mathcal{Y}}(\bar{\boldsymbol{y}})$ is given in Proposition \ref{pro:2}, and we use the fact that $\mathcal{N}_{\mathcal{X}}(\bar{\boldsymbol{x}}) = 0$ since $\bar{\boldsymbol{x}}$ is an interior point of $\mathcal{X}$. 

Combining (\ref{eq35}), (\ref{eq36}) with the fact that $\sigma_{ij}\in \mathcal{M}_{\bar{\boldsymbol{x}}}^{1}(\boldsymbol{y}_i)$ for $i\in \{1, \ldots, I\}$, $j\in \{1, \ldots, J\}$, the proof is complete.
\end{proof}

\subsection{Proof of Proposition \ref{pro:7}}

In the following, the proof of Proposition \ref{pro:7} is shown.

\begin{proof}
Since $(\bar{\boldsymbol{x}}, \bar{\boldsymbol{y}})$ is a feasible point of problem SPBP$\epsilon$ in (\ref{eq25}), we know that $(\bar{\boldsymbol{x}}, \bar{\boldsymbol{y}})$ satisfies the following constraints:
\begin{equation}  \label{eq37} \tag{37}
\bar{\boldsymbol{x}}\in \mathcal{X}, \qquad \bar{\boldsymbol{y}}\in \mathcal{Y}, \qquad f(\bar{\boldsymbol{x}}, \bar{\boldsymbol{y}})\le f^*(\bar{\boldsymbol{x}}) + \epsilon, \qquad F^*(\bar{\boldsymbol{x}}) \le F(\bar{\boldsymbol{x}}, \bar{\boldsymbol{y}}).
\end{equation}
Moreover, by the definition of $F^*(\boldsymbol{x})$ in (\ref{eq26}), it is easy to know that for any $\boldsymbol{x}\in \mathcal{X}$, $\boldsymbol{y}\in \mathcal{Y}$ satisfying $f(\boldsymbol{x}, \boldsymbol{y}) \le f^*(\boldsymbol{x})+\epsilon$, we always have 
\begin{equation} \label{eq38} \tag{38}
F^*(\boldsymbol{x})-F(\boldsymbol{x}, \boldsymbol{y}) \ge 0 .
\end{equation}

Combining (\ref{eq37}) with (\ref{eq38}), we know that $F(\bar{\boldsymbol{x}}, \bar{\boldsymbol{y}}) \le F^*(\bar{\boldsymbol{x}}) \le F(\bar{\boldsymbol{x}}, \bar{\boldsymbol{y}})$, i.e., $F^*(\bar{\boldsymbol{x}}) - F(\bar{\boldsymbol{x}}, \bar{\boldsymbol{y}}) = 0$. Therefore, $(\bar{\boldsymbol{x}}, \bar{\boldsymbol{y}})$ is the global optimal solution of the following problem
\begin{align*}     
&\min\limits_{\boldsymbol{x}, \boldsymbol{y}} F^*(\boldsymbol{x}) - F(\boldsymbol{x}, \boldsymbol{y})\\
& ~\text{s.t.}~~\boldsymbol{x}\in \mathcal{X}, ~\boldsymbol{y}\in \mathcal{Y}, \\  
& ~~~~ ~~f(\boldsymbol{x}, \boldsymbol{y}) \le f^*(\boldsymbol{x})+ \epsilon.  \\    
\end{align*}

Notice that following the discussion in the proof of Theorem \ref{tho:2}, we know that $F^*(\boldsymbol{x})$ is Lipschitz continuous near $\bar{\boldsymbol{x}}$, and thus from Theorem \ref{tho:1} in the main text, there exist $\lambda_2 \in \{0, 1\}$, $\lambda_3$ not all zero such that 
\begin{align} \label{eq39}  \nonumber
&0 =  \lambda_2 (\partial F^*(\bar{\boldsymbol{x}}) - \nabla_{\boldsymbol{x}} F(\bar{\boldsymbol{x}}, \bar{\boldsymbol{y}})) + \lambda_3 (\nabla_{\boldsymbol{x}} f(\bar{\boldsymbol{x}}, \bar{\boldsymbol{y}}) - \nabla f^*(\bar{\boldsymbol{x}})), \\  \tag{39}
& 0 \in - \lambda_2 \nabla_{\boldsymbol{y}} F(\bar{\boldsymbol{x}}, \bar{\boldsymbol{y}}) + \lambda_3 \nabla_{\boldsymbol{y}}f(\bar{\boldsymbol{x}}, \bar{\boldsymbol{y}}) +\mathcal{N}_{\mathcal{Y}}(\bar{\boldsymbol{y}}), \\   \nonumber
& \lambda_3 (f(\bar{\boldsymbol{x}}, \bar{\boldsymbol{y}}) - f^*(\bar{\boldsymbol{x}}) - \epsilon) = 0
\end{align}
where $\partial F^*(\bar{\boldsymbol{x}})$ is given in (\ref{eq35}), $\nabla f^*(\bar{\boldsymbol{x}})$ can be computed via Proposition \ref{pro:1}, the definition of $\mathcal{N}_{\mathcal{Y}}(\bar{\boldsymbol{y}})$ is given in Proposition \ref{pro:2}, and we use the fact that $\mathcal{N}_{\mathcal{X}}(\bar{\boldsymbol{x}}) = 0$ since $\bar{\boldsymbol{x}}$ is an interior point of $\mathcal{X}$. 

Combining (\ref{eq35}) with (\ref{eq39}), we know that $(\bar{\boldsymbol{x}}, \bar{\boldsymbol{y}})$ satisfies the Fritz-John type necessary optimality condition in Theorem \ref{tho:2} with $\lambda_1=0$. Then the proof is complete.
\end{proof}

\section{Proof of Theorem \ref{tho:3} and Theorem \ref{tho:4}}\label{appendix:2}

\subsection{Proof of Supporting Lemmas}

In the following, we first present some results that will be used in the proof.

For the optimization problem
\begin{equation}  \label{eq40} \tag{40}
\min\limits_{\boldsymbol{x} \in \mathcal{A}} f(\boldsymbol{x})
\end{equation}
where $\mathcal{A}\subset \mathbb{R}^n$ is a nonempty compact convex set, we have the following results, which can be found in Theorem 10.21 and Theorem 10.29 of \cite{31}.

\begin{lemma}  \label{lemma:2}
For problem (\ref{eq40}), suppose $f(\boldsymbol{x})$ is $L_{f}$ smooth, convex over $\mathcal{A}$. Let $f_{\text{opt}}$ be the optimal value. Apply the projected gradient descent 
\begin{equation*}
\boldsymbol{x}_{k+1}= \text{proj}_{\mathcal{A}}\bigg(\boldsymbol{x}_k - \frac{1}{L_f}\nabla f(\boldsymbol{x}_k)\bigg)
\end{equation*}
to solve problem (\ref{eq40}) with $\boldsymbol{x}_0$ being the initial value. Then for any ${\boldsymbol{x}^*}\in { \operatorname {arg\,min} }_{\boldsymbol{x} \in \mathcal{A}}f(\boldsymbol{x})$ and $k > 0$, we have 
\begin{equation*}
f(\boldsymbol{x}_k) - f_{\text{opt}} \le \frac{1}{2k}L_f \| \boldsymbol{x}_0 - \boldsymbol{x}^* \|^2.
\end{equation*}
\end{lemma}

\begin{lemma} \label{lemma:3}
For problem (\ref{eq40}), suppose $f(\boldsymbol{x})$ is $L_f$ smooth, $\delta$ strongly convex over $\mathcal{A}$. Let $f_{\text{opt}}$ be the optimal value and $\boldsymbol{x}^*$ be the optimal solution. Apply the projected gradient descent 
\begin{equation*}
\boldsymbol{x}_{k+1} = \text{proj}_{\mathcal{A}}\bigg(\boldsymbol{x}_k - \frac{1}{L_f}\nabla f(\boldsymbol{x}_k)\bigg)
\end{equation*}
to solve problem (\ref{eq40}) with $\boldsymbol{x}_0$ being the initial value. Then, for any $k>0$, we have 
\begin{align*}
& \|\boldsymbol{x}_k - \boldsymbol{x}^*\|^2 \le \bigg(1 - \frac{\delta}{L_f}\bigg)^k \|\boldsymbol{x}_0 - \boldsymbol{x}^*\|^2,  \\
& f(\boldsymbol{x}_k) - f_{\text{opt}} \le \frac{L_f}{2}\bigg( 1 - \frac{\delta}{L_f}\bigg)^k \|\boldsymbol{x}_0 - \boldsymbol{x}^*\|^2.
\end{align*}
\end{lemma}

In the following, some results on functions $\boldsymbol{y}_j(\boldsymbol{x})$ and $f_j(\boldsymbol{x})$ are shown. 

\begin{lemma}  \label{lemma:4}
Suppose Assumptions \ref{assum:1}, \ref{assum:2}, \ref{assum:3}, and \ref{assum:6} hold. Let $J$ be an arbitrary positive integer. For each $\boldsymbol{x}\in \mathcal{X}$, apply $J$ steps of projected gradient descent in (\ref{eq14}) to solve problem (\ref{eq12}), and obtain the values $\boldsymbol{y}_1(\boldsymbol{x}), \ldots, \boldsymbol{y}_J(\boldsymbol{x})$. Define 
\begin{equation}    \label{eq41}  \tag{41}
f_{j}(\boldsymbol{x}):= f(\boldsymbol{x}, \boldsymbol{y}_j(\boldsymbol{x}))
\end{equation}
where $\boldsymbol{x}\in \mathcal{X}$, $j\in \{1, \ldots, J\}$. Then for any $j \in \{1, \ldots, J\}$, $\boldsymbol{y}_j(\boldsymbol{x})$ and $f_j(\boldsymbol{x})$ are differentiable on $\mathcal{X}$, and for any $\boldsymbol{x}\in \mathcal{X}$, $j\in \{1, \ldots, J\}$, we have 
\begin{equation*}
\|\nabla \boldsymbol{y}_j(\boldsymbol{x})\| \le M_0(J), \qquad \| \nabla f_j(\boldsymbol{x})\| \le M_1(J)
\end{equation*}
where $M_0(J)=J$, $M_1(J)= L_0(J+1)$. Furthermore, $f_J(\boldsymbol{x})$ is $M_2(J)$ smooth on $\boldsymbol{x}\in \mathcal{X}$, where 
\begin{equation*}
M_2(J) = L_1(1 + J)^2 + \frac{L_0 J}{L_1}(1 + J)(L_2 + L_3 J)
\end{equation*}
with $L_0$, $L_1$, $L_2$, $L_3$ given in Assumption \ref{assum:3}.
\end{lemma}

\begin{proof}
We first prove the differentiability of $\boldsymbol{y}_j(\boldsymbol{x})$ and $f_j(\boldsymbol{x})$. From Assumption \ref{assum:6}, we know that for any $\boldsymbol{x}\in \mathcal{X}$, $j\in \{1, \ldots, J\}$, 
\begin{equation} \label{eq42}  \tag{42}
\boldsymbol{y}_j(\boldsymbol{x}) = \boldsymbol{y}_{j-1}(\boldsymbol{x})- \frac{1}{L_1}\nabla_{\boldsymbol{y}}f(\boldsymbol{x}, \boldsymbol{y}_{j-1}(\boldsymbol{x})).
\end{equation}
Then, when $j=1$, we know that 
\begin{equation}  \label{eq43}  \tag{43}
\boldsymbol{y}_1(\boldsymbol{x}) = \boldsymbol{y}_0 - \frac{1}{L_1}\nabla_{\boldsymbol{y}}f(\boldsymbol{x}, \boldsymbol{y}_0)
\end{equation}
where we use the fact that $\boldsymbol{y}_0(\boldsymbol{x})=\boldsymbol{y}_0$. From Assumption \ref{assum:2}, the differentiability of $\boldsymbol{y}_1(\boldsymbol{x})$ on $\mathcal{X}$ can be obtained. By using induction, assume that $\boldsymbol{y}_{j-1}(\boldsymbol{x})$ is differentiable on $\mathcal{X}$. It is easy to know that $\boldsymbol{y}_j(\boldsymbol{x})$ in (\ref{eq42}) is differentiable on $\mathcal{X}$. Therefore, $\boldsymbol{y}_j(\boldsymbol{x})$ is differentiable on $\mathcal{X}$ for any $j\in \{1, \ldots, J\}$. Furthermore, by the definition of $f_j(\boldsymbol{x})$ in (\ref{eq41}), the differentiability of $f_j(\boldsymbol{x})$ on $\mathcal{X}$ also can be obtained by using Assumption \ref{assum:2} and the differentiability of $\boldsymbol{y}_j(\boldsymbol{x})$ on $\mathcal{X}$, with $j\in \{1, \ldots, J\}$.

Next, we prove the boundness of $\nabla \boldsymbol{y}_j(\boldsymbol{x})$ and $\nabla f_{j}(\boldsymbol{x})$.

In the following, we first prove that for any $\boldsymbol{x}\in \mathcal{X}$, $j\in \{1, \ldots, J\}$, we have $\|\nabla \boldsymbol{y}_j(\boldsymbol{x})\|\le j$ . From (\ref{eq43}), it is easy to know that $\|\nabla \boldsymbol{y}_1(\boldsymbol{x})\| \le 1$ for any $\boldsymbol{x}\in \mathcal{X}$ by using Assumption \ref{assum:3}. By using induction, we suppose that $\|\nabla \boldsymbol{y}_j(\boldsymbol{x})\| \le j$ for any $\boldsymbol{x}\in \mathcal{X}$, with $j \in \{1, \ldots, J-1\}$. Then from (\ref{eq42}), we have 
\begin{equation*}
\nabla \boldsymbol{y}_{j+1}(\boldsymbol{x}) = \bigg(\mathbf{I} - \frac{1}{L_1}\nabla_{\boldsymbol{y}\boldsymbol{y}}f(\boldsymbol{x}, \boldsymbol{y}_j(\boldsymbol{x}))\bigg)\nabla \boldsymbol{y}_j(\boldsymbol{x}) - \frac{1}{L_1}\nabla_{\boldsymbol{y}\boldsymbol{x}}f(\boldsymbol{x}, \boldsymbol{y}_j(\boldsymbol{x})),
\end{equation*}
and thus for any $\boldsymbol{x}\in \mathcal{X}$, 
\begin{align*}
\|\nabla \boldsymbol{y}_{j+1}(\boldsymbol{x})\| & \le \bigg\| \mathbf{I} - \frac{1}{L_1}\nabla_{\boldsymbol{y}\boldsymbol{y}}f(\boldsymbol{x}, \boldsymbol{y}_j(\boldsymbol{x}))\bigg\|\|\nabla \boldsymbol{y}_j(\boldsymbol{x})\| + \frac{1}{L_1}\|\nabla_{\boldsymbol{y}\boldsymbol{x}}f(\boldsymbol{x}, \boldsymbol{y}_j(\boldsymbol{x}))\| \\
& \overset{(i)}{\le} j+1
\end{align*}
where $(i)$ can be obtained by using Assumptions \ref{assum:2} and \ref{assum:3}. Therefore, it is proved, and as a result for any $\boldsymbol{x}\in \mathcal{X}$, $j\in \{1, \ldots, J\}$, we have $\|\nabla \boldsymbol{y}_j(\boldsymbol{x})\| \le M_0(J)$, where we use the fact that $M_0(J)=J$.

In the following, we prove the boundness of $\nabla f_j(\boldsymbol{x})$ on $\mathcal{X}$ for any $j\in \{1, \ldots, J\}$. For any $\boldsymbol{x}\in \mathcal{X}$, $j\in \{1, \ldots, J\}$, by the definition of $f_j(\boldsymbol{x})$ in (\ref{eq41}), we have 
\begin{equation*}
\nabla f_j(\boldsymbol{x}) = \nabla_{\boldsymbol{x}}f(\boldsymbol{x}, \boldsymbol{y}_j(\boldsymbol{x})) + (\nabla \boldsymbol{y}_j(\boldsymbol{x}))^\top \nabla_{\boldsymbol{y}}f(\boldsymbol{x}, \boldsymbol{y}_j(\boldsymbol{x})).
\end{equation*}
Then, 
\begin{align*}
\|\nabla f_j(\boldsymbol{x})\| &\le \| \nabla_{\boldsymbol{x}} f(\boldsymbol{x}, \boldsymbol{y}_j(\boldsymbol{x}))\| + \|\nabla \boldsymbol{y}_j(\boldsymbol{x})\|\|\nabla_{\boldsymbol{y}} f(\boldsymbol{x}, \boldsymbol{y}_j(\boldsymbol{x}))\|  \\
& \overset{(i)}{\le} M_1(J)
\end{align*}
where $(i)$ can be obtained by using the definition of $M_1(J)$, Assumption \ref{assum:3}, and the fact that $\|\nabla \boldsymbol{y}_j(\boldsymbol{x})\|  \le M_0(J)$ with $M_0(J)=J$.

Based on the above discussion, the boundness of $\nabla \boldsymbol{y}_j(\boldsymbol{x})$ and $\nabla f_j(\boldsymbol{x})$ are proved. In the following, we prove the Lipschitz smoothness of $f_J(\boldsymbol{x})$ on $\mathcal{X}$. By the definition of $f_J(\boldsymbol{x})$ in (\ref{eq41}), we have 

\begin{equation*}
\nabla f_J(\boldsymbol{x}) = \nabla_{\boldsymbol{x}}f(\boldsymbol{x}, \boldsymbol{y}_J(\boldsymbol{x})) + (\nabla \boldsymbol{y}_J(\boldsymbol{x}))^\top \nabla_{\boldsymbol{y}}f(\boldsymbol{x}, \boldsymbol{y}_J(\boldsymbol{x})).
\end{equation*}
Then for any $\boldsymbol{x}_1$, $\boldsymbol{x}_2 \in \mathcal{X}$, we have 
\begin{align}  \label{eq44}   \nonumber
&\|\nabla f_J(\boldsymbol{x}_1) - \nabla f_J(\boldsymbol{x}_2)\|  \\   \nonumber
&\le \|\nabla_{\boldsymbol{x}}f(\boldsymbol{x}_1, \boldsymbol{y}_J(\boldsymbol{x}_1)) - \nabla_{x}f(\boldsymbol{x}_2, \boldsymbol{y}_J(\boldsymbol{x}_2))\| \\  \nonumber
&~~~+ \|(\nabla \boldsymbol{y}_J(\boldsymbol{x}_1))^\top \nabla_{\boldsymbol{y}}f(\boldsymbol{x}_1, \boldsymbol{y}_J(\boldsymbol{x}_1)) - (\nabla \boldsymbol{y}_J(\boldsymbol{x}_2))^\top \nabla_{\boldsymbol{y}}f(\boldsymbol{x}_2, \boldsymbol{y}_J(\boldsymbol{x}_2))\|\\   \nonumber
& \le \|\nabla_{\boldsymbol{x}}f(\boldsymbol{x}_1, \boldsymbol{y}_J(\boldsymbol{x}_1)) - \nabla_{\boldsymbol{x}}f(\boldsymbol{x}_2, \boldsymbol{y}_J(\boldsymbol{x}_2))\|+ \|\nabla \boldsymbol{y}_J(\boldsymbol{x}_1)\| \|\nabla_{\boldsymbol{y}}f(\boldsymbol{x}_1, \boldsymbol{y}_J(\boldsymbol{x}_1)) - \nabla_{\boldsymbol{y}}f(\boldsymbol{x}_2, \boldsymbol{y}_J(\boldsymbol{x}_2))\| \\     \nonumber
&~~~ + \|\nabla \boldsymbol{y}_J(\boldsymbol{x}_1) - \nabla \boldsymbol{y}_J(\boldsymbol{x}_2)\| \|\nabla_{\boldsymbol{y}}f(\boldsymbol{x}_2, \boldsymbol{y}_J(\boldsymbol{x}_2))\| \\  \nonumber
& \overset{(i)}{\le} \|\nabla_{\boldsymbol{x}}f(\boldsymbol{x}_1, \boldsymbol{y}_J(\boldsymbol{x}_1)) - \nabla_{\boldsymbol{x}}f(\boldsymbol{x}_2, \boldsymbol{y}_J(\boldsymbol{x}_2))\| + J\|\nabla_{\boldsymbol{y}} f(\boldsymbol{x}_1, \boldsymbol{y}_J(\boldsymbol{x}_1)) - \nabla_{\boldsymbol{y}}f(\boldsymbol{x}_2, \boldsymbol{y}_J(\boldsymbol{x}_2))\|\\  \nonumber
& ~~~+ L_0 \|\nabla \boldsymbol{y}_J(\boldsymbol{x}_1) - \nabla \boldsymbol{y}_J(\boldsymbol{x}_2)\|    \tag{44}
\end{align} 
where $(i)$ is obtained by using Assumption \ref{assum:3} and the fact that $\| \nabla \boldsymbol{y}_J(\boldsymbol{x}) \|\le M_0(J) = J$ for any $\boldsymbol{x}\in \mathcal{X}$.

On one hand, for any $\boldsymbol{x}_1$, $\boldsymbol{x}_2 \in \mathcal{X}$, we have 

\begin{align}   \label{eq45}   \nonumber
&\|\nabla f(\boldsymbol{x}_1, \boldsymbol{y}_J(\boldsymbol{x}_1)) - \nabla f(\boldsymbol{x}_2, \boldsymbol{y}_J(\boldsymbol{x}_2))\|   \\   \nonumber
& \le \|\nabla f(\boldsymbol{x}_1, \boldsymbol{y}_J(\boldsymbol{x}_1)) - \nabla f(\boldsymbol{x}_1, \boldsymbol{y}_J(\boldsymbol{x}_2))\| + \|\nabla f(\boldsymbol{x}_1, \boldsymbol{y}_J(\boldsymbol{x}_2))- \nabla f(\boldsymbol{x}_2, \boldsymbol{y}_J(\boldsymbol{x}_2))\|  \\  \tag{45}
& \overset{(i)}{\le} L_1 \|\boldsymbol{y}_J(\boldsymbol{x}_1) - \boldsymbol{y}_J(\boldsymbol{x}_2)\| + L_1\|\boldsymbol{x}_1 - \boldsymbol{x}_2\|  \\   \nonumber
&  \overset{(ii)}{\le} (L_1 J + L_1)\|\boldsymbol{x}_1 - \boldsymbol{x}_2\|  \nonumber
\end{align}
where $(i)$ follows from Assumption \ref{assum:3}, $(ii)$ holds since $\|\nabla \boldsymbol{y}_J(\boldsymbol{x})\| \le M_0(J) = J$ for any $\boldsymbol{x}\in \mathcal{X}$.

On the other hand, it can be proved that for any $\boldsymbol{x}\in \mathcal{X}$, 
\begin{equation}    \label{eq46}  \tag{46}
\|\nabla^2 \boldsymbol{y}_J(\boldsymbol{x})\| \le \frac{1}{L_1}J(1 + J)(L_2 + L_3J)
\end{equation}
where $\nabla^2 \boldsymbol{y}_J(\boldsymbol{x})$ denotes the Hessian matrix of function $\boldsymbol{y}_J(\boldsymbol{x})$ w.r.t. $\boldsymbol{x}$. To be specific, from (\ref{eq43}), we know that $\nabla \boldsymbol{y}_1(\boldsymbol{x}) = -1/L_1 \nabla_{\boldsymbol{y}\boldsymbol{x}}f(\boldsymbol{x}, \boldsymbol{y}_0)$, and therefore, 
\begin{equation}  \label{eq47}  \tag{47}
\|\nabla^2 \boldsymbol{y}_1(\boldsymbol{x})\|  \le \frac{L_2}{L_1}
\end{equation}
which can be obtained by using Assumption \ref{assum:3}. Furthermore, from (\ref{eq42}), by simple calculation, it is easy to know that for any $j \in \{1, \ldots, J\}$, $\boldsymbol{x}\in \mathcal{X}$, 
\begin{equation*}
\nabla \boldsymbol{y}_j(\boldsymbol{x}) = \bigg(\mathbf{I} - \frac{1}{L_1}\nabla_{\boldsymbol{y}\boldsymbol{y}}f(\boldsymbol{x}, \boldsymbol{y}_{j-1}(\boldsymbol{x}))\bigg)\nabla \boldsymbol{y}_{j-1}(\boldsymbol{x}) - \frac{1}{L_1}\nabla_{\boldsymbol{y}\boldsymbol{x}}f(\boldsymbol{x}, \boldsymbol{y}_{j-1}(\boldsymbol{x})).
\end{equation*}
As a result, 
\begin{align*}
&\nabla^2 \boldsymbol{y}_j(\boldsymbol{x}) \\
=& \bigg(\mathbf{I} - \frac{1}{L_1}\nabla_{\boldsymbol{y}\boldsymbol{y}}f(\boldsymbol{x}, \boldsymbol{y}_{j-1}(\boldsymbol{x}))\bigg) \nabla^2 \boldsymbol{y}_{j-1}(\boldsymbol{x})\\
& -\frac{1}{L_1}\bigg(\nabla_{\boldsymbol{y}\boldsymbol{y}\boldsymbol{x}}f(\boldsymbol{x}, \boldsymbol{y}_{j-1}(\boldsymbol{x})) + \nabla_{\boldsymbol{y}\boldsymbol{y}\boldsymbol{y}}f(\boldsymbol{x}, \boldsymbol{y}_{j-1}(\boldsymbol{x}))\nabla \boldsymbol{y}_{j-1}(\boldsymbol{x})\bigg)\nabla \boldsymbol{y}_{j-1}(\boldsymbol{x}) \\
& - \frac{1}{L_1}\bigg(\nabla_{\boldsymbol{y}\boldsymbol{x}\boldsymbol{x}} f(\boldsymbol{x}, \boldsymbol{y}_{j-1}(\boldsymbol{x})) + \nabla_{\boldsymbol{y}\boldsymbol{x}\boldsymbol{y}} f(\boldsymbol{x}, \boldsymbol{y}_{j-1}(\boldsymbol{x})) \nabla \boldsymbol{y}_{j-1}(\boldsymbol{x})\bigg)
\end{align*}
and 
\begin{align*}
& \|\nabla^2 \boldsymbol{y}_j(\boldsymbol{x})\| \\
& \overset{(i)}{\le} \|\nabla^2 \boldsymbol{y}_{j-1}(\boldsymbol{x})\| + \frac{1}{L_1}J(L_3 + L_3J)+ \frac{1}{L_1}(L_2 + L_2 J)\\
& = \|\nabla^2 \boldsymbol{y}_{j-1}(\boldsymbol{x})\| + \frac{1}{L_1}(1 + J)(L_2 + L_3 J)
\end{align*}
where $(i)$ follows from Assumptions \ref{assum:2} and \ref{assum:3}.
Therefore, 
\begin{align*}
\|\nabla^2 \boldsymbol{y}_J(\boldsymbol{x})\| &\le \|\nabla^2 \boldsymbol{y}_{J-1}(\boldsymbol{x})\| + \frac{1}{L_1}(1 + J)(L_2 + L_3 J) \\
& \le \|\nabla^2 \boldsymbol{y}_1(\boldsymbol{x})\| +  \frac{1}{L_1}(J-1)(1+J)(L_2 + L_3 J)   \\
& \overset{(i)}{\le} \frac{J}{L_1}(1 + J)(L_2 + L_3 J)
\end{align*}
where $(i)$ follows from the inequality in (\ref{eq47}). Then it is proved.

Combining (\ref{eq44}), (\ref{eq45}), with (\ref{eq46}), we have 
\begin{align*}
\|\nabla f_J(\boldsymbol{x}_1) - \nabla f_J(\boldsymbol{x}_2)\| \le \bigg(L_1(1 + J)^2 + \frac{1}{L_1}L_0J(1 + J)(L_2 + L_3J)\bigg)\|\boldsymbol{x}_1 - \boldsymbol{x}_2\|.
\end{align*}
\end{proof}

In the following, the Lipschitz continuity of the gradient of function $G_{\epsilon, \tau, J}(\boldsymbol{x}, \boldsymbol{y})$ defined in (\ref{eq16}) is shown.

\begin{lemma}  \label{lemma:5}
Suppose Assumptions \ref{assum:1}, \ref{assum:2}, \ref{assum:3}, \ref{assum:6}, and \ref{assum:7} hold. For the function $G_{\epsilon, \tau, J}(\boldsymbol{x}, \boldsymbol{y})$ defined in (\ref{eq16}), let $\tau \in (0, 1)$, positive integer $J$ be any given numbers. Then for any $\boldsymbol{x}_1$, $\boldsymbol{x}_2 \in \mathcal{X}$, $\boldsymbol{y}$, $\boldsymbol{y}_1$, $\boldsymbol{y}_2 \in \mathcal{Y}_J$, where $\mathcal{Y}_J$ is defined in (\ref{eq19}), we have 
\begin{align*}
& \|\nabla_{\boldsymbol{x}} G_{\epsilon, \tau, J}(\boldsymbol{x}_1, \boldsymbol{y}_1) - \nabla_{\boldsymbol{x}} G_{\epsilon, \tau, J}(\boldsymbol{x}_2, \boldsymbol{y}_2)\| \le L_{11}(J)\big(\|\boldsymbol{x}_1 - \boldsymbol{x}_2\|^2 + \|\boldsymbol{y}_1 - \boldsymbol{y}_2\|^2\big)^{1/2}, \\
& \|\nabla_{\boldsymbol{y}} G_{\epsilon, \tau, J}(\boldsymbol{x}_1, \boldsymbol{y}) - \nabla_{\boldsymbol{y}}G_{\epsilon, \tau, J}(\boldsymbol{x}_2, \boldsymbol{y})\| \le L_{12}(J)\|\boldsymbol{x}_1 - \boldsymbol{x}_2\|
\end{align*}
where 
\begin{align*}
& L_{11}(J) = h_1 + \frac{1}{c}(M_2(J) + L_1)+ \frac{1}{c^2}(M_1(J) + L_0)^2,\\
& L_{12}(J) = h_1 + \frac{1}{c}L_1 + \frac{1}{c^2}L_0(M_1(J) + L_0),
\end{align*}
$h_1$, $L_1$, $L_0$ are given in Assumption \ref{assum:3}, $M_1(J)$, $M_2(J)$ are given in Lemma \ref{lemma:4}, and $c$ is given in Assumption \ref{assum:7}.
\end{lemma}

\begin{proof}
By the definition of $G_{\epsilon, \tau, J}(\boldsymbol{x}, \boldsymbol{y})$ in (\ref{eq16}), we have 
\begin{equation*}
\nabla_{\boldsymbol{x}}G_{\epsilon, \tau, J}(\boldsymbol{x}, \boldsymbol{y}) = \nabla_{\boldsymbol{x}}F(\boldsymbol{x}, \boldsymbol{y}) + \frac{\tau}{f_J(\boldsymbol{x}) + \epsilon - f(\boldsymbol{x}, \boldsymbol{y})}(\nabla f_J(\boldsymbol{x}) - \nabla_{\boldsymbol{x}}f(\boldsymbol{x}, \boldsymbol{y})).
\end{equation*}

Then for $\boldsymbol{x}_1$, $\boldsymbol{x}_2 \in \mathcal{X}$, $\boldsymbol{y}_1, \boldsymbol{y}_2 \in \mathcal{Y}_J$, we have 
\begin{align}\label{eq48} \nonumber
& \|\nabla_{\boldsymbol{x}} G_{\epsilon, \tau, J}(\boldsymbol{x}_1, \boldsymbol{y}_1) - \nabla_{\boldsymbol{x}}G_{\epsilon, \tau, J}(\boldsymbol{x}_2, \boldsymbol{y}_2)\|   \\   \nonumber
& \overset{(i)}{\le }\|\nabla_{\boldsymbol{x}} F(\boldsymbol{x}_1, \boldsymbol{y}_1) - \nabla_{\boldsymbol{x}} F(\boldsymbol{x}_2, \boldsymbol{y}_2)\| \\   \nonumber
&~~~+ \bigg\|\frac{\tau}{f_J(\boldsymbol{x}_1)+ \epsilon - f(\boldsymbol{x}_1, \boldsymbol{y}_1)} \nabla f_J(\boldsymbol{x}_1) - \frac{\tau}{f_J(\boldsymbol{x}_2) + \epsilon - f(\boldsymbol{x}_2, \boldsymbol{y}_2)} \nabla f_J(\boldsymbol{x}_2)\bigg\| \\   \tag{48}
& ~~~+ \bigg\| \frac{\tau}{f_J(\boldsymbol{x}_1) + \epsilon - f(\boldsymbol{x}_1, \boldsymbol{y}_1)}\nabla_{\boldsymbol{x}}f(\boldsymbol{x}_1, \boldsymbol{y}_1) - \frac{\tau}{f_J(\boldsymbol{x}_2) +\epsilon - f(\boldsymbol{x}_2, \boldsymbol{y}_2)}\nabla_{\boldsymbol{x}}f(\boldsymbol{x}_2, \boldsymbol{y}_2) \bigg\|   \nonumber
\end{align}
where $(i)$ uses the triangle inequality.

On one hand, 
\begin{align}   \label{eq49} \nonumber
& \bigg\| \frac{\tau}{f_J(\boldsymbol{x}_1) + \epsilon - f(\boldsymbol{x}_1, \boldsymbol{y}_1)}\nabla f_{J}(\boldsymbol{x}_1) - \frac{\tau}{f_J(\boldsymbol{x}_2) + \epsilon - f(\boldsymbol{x}_2, \boldsymbol{y}_2)}\nabla f_J(\boldsymbol{x}_2)  \bigg\|   \\  \nonumber
&\overset{(i)}{\le} \frac{\tau}{f_J(\boldsymbol{x}_1) + \epsilon - f(\boldsymbol{x}_1,\boldsymbol{y}_1)}\|\nabla f_J(\boldsymbol{x}_1) - \nabla f_J(\boldsymbol{x}_2)\|  \\ \nonumber
& ~~~+ \tau \|\nabla f_J(\boldsymbol{x}_2)\| \bigg|\frac{1}{f_J(\boldsymbol{x}_1) + \epsilon - f(\boldsymbol{x}_1, \boldsymbol{y}_1)}- \frac{1}{f_J(\boldsymbol{x}_2) + \epsilon - f(\boldsymbol{x}_2, \boldsymbol{y}_2)}\bigg|  \\   \tag{49}
& \overset{(ii)}{\le} \frac{\tau}{c}M_2(J)\|\boldsymbol{x}_1 - \boldsymbol{x}_2\| + \tau M_1(J)\bigg|\frac{f_J(\boldsymbol{x}_2) - f_J(\boldsymbol{x}_1) - f(\boldsymbol{x}_2, \boldsymbol{y}_2) + f(\boldsymbol{x}_1, \boldsymbol{y}_1)}{(f_J(\boldsymbol{x}_1) + \epsilon - f(\boldsymbol{x}_1, \boldsymbol{y}_1))(f_J(\boldsymbol{x}_2) + \epsilon - f(\boldsymbol{x}_2, \boldsymbol{y}_2))}\bigg|    \\   \nonumber
& \overset{(iii)}{\le} \bigg(\frac{\tau}{c}M_2(J) + \frac{\tau}{c^2}M_1(J)(M_1(J) + L_0)\bigg) \big(\|\boldsymbol{x}_1 - \boldsymbol{x}_2\|^2 + \|\boldsymbol{y}_1 - \boldsymbol{y}_2\|^2\big)^{\frac{1}{2}}
\end{align}
where $(i)$ uses the triangle inequality, $(ii)$ follows from Lemma \ref{lemma:4} and the fact that $\boldsymbol{y}_1 \in \mathcal{Y}_J$, and $(iii)$ follows from Assumption \ref{assum:3}, Lemma \ref{lemma:4}, and the fact that $\boldsymbol{y}_1$, $\boldsymbol{y}_2 \in \mathcal{Y}_J$.

On the other hand, 
\begin{align}    \label{eq50}  \nonumber
& \bigg\| \frac{\tau}{f_J(\boldsymbol{x}_1) + \epsilon - f(\boldsymbol{x}_1, \boldsymbol{y}_1)}\nabla_{\boldsymbol{x}}f(\boldsymbol{x}_1, \boldsymbol{y}_1) - \frac{\tau}{f_J(\boldsymbol{x}_2)+ \epsilon - f(\boldsymbol{x}_2, \boldsymbol{y}_2)} \nabla_{\boldsymbol{x}}f(\boldsymbol{x}_2, \boldsymbol{y}_2)  \bigg\|  \\  \nonumber
& \overset{(i)} {\le} \tau \|\nabla_{\boldsymbol{x}} f(\boldsymbol{x}_1, \boldsymbol{y}_1)\| \bigg|\frac{1}{f_J(\boldsymbol{x}_1)+ \epsilon - f(\boldsymbol{x}_1, \boldsymbol{y}_1)} - \frac{1}{f_J(\boldsymbol{x}_2)+ \epsilon - f(\boldsymbol{x}_2, \boldsymbol{y}_2)}\bigg|   \\   \tag{50}
&~~~ + \frac{\tau}{f_J(\boldsymbol{x}_2) + \epsilon - f(\boldsymbol{x}_2, \boldsymbol{y}_2)} \|\nabla_{\boldsymbol{x}} f(\boldsymbol{x}_1, \boldsymbol{y}_1) - \nabla_{\boldsymbol{x}} f(\boldsymbol{x}_2, \boldsymbol{y}_2)\|   \\  \nonumber
&  \overset{(ii)}{\le} \frac{\tau}{c^2}L_0 \bigg(|f_J(\boldsymbol{x}_1) - f_J(\boldsymbol{x}_2)| + |f(\boldsymbol{x}_1, \boldsymbol{y}_1) - f(\boldsymbol{x}_2, \boldsymbol{y}_2)|\bigg)+ \frac{\tau}{c}L_1 \big(\|\boldsymbol{x}_1 - \boldsymbol{x}_2\|^2 + \|\boldsymbol{y}_1 - \boldsymbol{y}_1\|^2\big)^{\frac{1}{2}} \\ \nonumber
& \overset{(iii)}{\le} \bigg(\frac{\tau}{c^2}L_0\big(M_1(J)+ L_0\big)+ \frac{\tau}{c}L_1\bigg)\big(\|\boldsymbol{x}_1 - \boldsymbol{x}_2\|^2 + \|\boldsymbol{y}_1 - \boldsymbol{y}_1\|^2\big)^{\frac{1}{2}}
\end{align}
where $(i)$ uses the triangle inequality, $(ii)$ follows from Assumption \ref{assum:3} and the fact that $\boldsymbol{y}_1$, $\boldsymbol{y}_2 \in \mathcal{Y}_J$, $(iii)$ follows from Assumption \ref{assum:3} and Lemma \ref{lemma:4}.

Combining (\ref{eq48}), (\ref{eq49}), with (\ref{eq50}), we have 
\begin{align*}
&\|\nabla_{\boldsymbol{x}} G_{\epsilon, \tau, J}(\boldsymbol{x}_1, \boldsymbol{y}_1) - \nabla_{\boldsymbol{x}}G_{\epsilon, \tau, J}(\boldsymbol{x}_2, \boldsymbol{y}_2)\|  \\
& \overset{(i)}{\le} \bigg(h_1 + \frac{\tau}{c}(M_2(J)+ L_1) + \frac{\tau}{c^2}(M_1(J)+ L_0)^2\bigg)\big(\|\boldsymbol{x}_1 - \boldsymbol{x}_2\|^2 + \|\boldsymbol{y}_1 - \boldsymbol{y}_2\|^2\big)^{\frac{1}{2}}\\
&\overset{(ii)}{\le} \bigg(h_1 + \frac{1}{c}(M_2(J)+ L_1) + \frac{1}{c^2}(M_1(J)+ L_0)^2\bigg)\big(\|\boldsymbol{x}_1 - \boldsymbol{x}_2\|^2 + \|\boldsymbol{y}_1 - \boldsymbol{y}_2\|^2\big)^{\frac{1}{2}}\\
\end{align*}
where $(i)$ follows from Assumption \ref{assum:3}, $(ii)$ follows from the fact that $0< \tau <1$.

Furthermore, by the definition of $G_{\epsilon, \tau, J}(\boldsymbol{x}, \boldsymbol{y})$ in (\ref{eq16}), we have  
\begin{equation}  \label{eq51}  \tag{51}
\nabla_{\boldsymbol{y}} G_{\epsilon, \tau, J}(\boldsymbol{x}, \boldsymbol{y}) = \nabla_{\boldsymbol{y}} F(\boldsymbol{x}, \boldsymbol{y}) - \frac{\tau}{f_J(\boldsymbol{x})+ \epsilon - f(\boldsymbol{x}, \boldsymbol{y})} \nabla_{\boldsymbol{y}}f(\boldsymbol{x}, \boldsymbol{y}).
\end{equation}
Then, for any $\boldsymbol{x}_1$, $\boldsymbol{x}_2 \in \mathcal{X}$, $\boldsymbol{y}\in \mathcal{Y}_J$, we have
\begin{align*}
&\|\nabla_{\boldsymbol{y}}G_{\epsilon, \tau, J}(\boldsymbol{x}_1, \boldsymbol{y}) - \nabla_{\boldsymbol{y}}G_{\epsilon, \tau, J}(\boldsymbol{x}_2, \boldsymbol{y})\|   \\
& \overset{(i)}{\le} \|\nabla_{\boldsymbol{y}}F(\boldsymbol{x}_1, \boldsymbol{y}) - \nabla_{\boldsymbol{y}}F(\boldsymbol{x}_2, \boldsymbol{y})\| + \frac{\tau}{f_J(\boldsymbol{x}_1) + \epsilon - f(\boldsymbol{x}_1, \boldsymbol{y})}  \| \nabla_{\boldsymbol{y}}f(\boldsymbol{x}_1, \boldsymbol{y})- \nabla_{\boldsymbol{y}}f(\boldsymbol{x}_2, \boldsymbol{y}) \|  \\
& ~~~+  \tau \|\nabla_{\boldsymbol{y}} f(\boldsymbol{x}_2, \boldsymbol{y})\| \bigg|\frac{1}{f_J(\boldsymbol{x}_1)+ \epsilon - f(\boldsymbol{x}_1, \boldsymbol{y})} - \frac{1}{f_J(\boldsymbol{x}_2)+ \epsilon - f(\boldsymbol{x}_2, \boldsymbol{y})}\bigg|  \\
& \overset{(ii)}{\le} h_1 \|\boldsymbol{x}_1 - \boldsymbol{x}_2\| + \frac{\tau}{c}L_1\|\boldsymbol{x}_1 - \boldsymbol{x}_2\| + \frac{\tau}{c^2}L_0 \bigg(M_1(J)\|\boldsymbol{x}_1 - \boldsymbol{x}_2\| + L_0 \|\boldsymbol{x}_1 - \boldsymbol{x}_2\|\bigg)  \\
& \le \bigg(h_1 + \frac{1}{c}L_1 + \frac{1}{c^2}L_0\big(M_1(J) + L_0\big)\bigg) \|\boldsymbol{x}_1 - \boldsymbol{x}_2\|
\end{align*}
where $(i)$ uses the triangle inequality, $(ii)$ follows from Assumption \ref{assum:3}, the fact that $\boldsymbol{y}\in \mathcal{Y}_J$, and Lemma \ref{lemma:4}.
\end{proof}

Let $c_0$ in (\ref{eq21}) be the number that equals to $c$ in Assumption \ref{assum:7}. Recall that for any $\tau\in (0, 1)$, positive integer $J$, $\boldsymbol{x}\in \mathcal{X}$, an approximate solution to problem (\ref{eq20}) can be obtained by running $K$ steps of projected gradient ascent starting from initial value $\boldsymbol{y}_0(\boldsymbol{x}) \in \mathcal{Y}_J(x)$ in the form of 
\begin{equation} \label{eq52}  \tag{52}
\boldsymbol{y}_{k+1}(\boldsymbol{x}) = \text{proj}_{\mathcal{Y}_J(\boldsymbol{x})}(\boldsymbol{y}_k(\boldsymbol{x}) + \beta \nabla_{\boldsymbol{y}} G_{\epsilon, \tau, J}(\boldsymbol{x}, \boldsymbol{y}_k(\boldsymbol{x})))
\end{equation}
with $k=0, \ldots, K-1$, $\beta$ being the stepsize, and $\mathcal{Y}_J(\boldsymbol{x})$ given in (\ref{eq21}) with $c_0=c$.

In the following, some properties on function $G_{\epsilon, \tau, J}(\boldsymbol{x}, \boldsymbol{y})$, and some convergence results on the above projected gradient ascent steps are shown.

\begin{lemma} \label{lemma:6}
Suppose Assumptions \ref{assum:1}, \ref{assum:2}, \ref{assum:3}, \ref{assum:6}, and \ref{assum:7} hold. Then, for any given $\tau \in (0, 1)$, positive integer $J$, $\boldsymbol{x}\in \mathcal{X}$, $G_{\epsilon, \tau, J}(\boldsymbol{x}, \boldsymbol{y})$ defined in (\ref{eq16}) is $L_G$ smooth, $\mu$-strongly concave w.r.t. $\boldsymbol{y}$ on the set $\mathcal{Y}_J(\boldsymbol{x})$, where $L_G = h_1 + 2/cL_1 + 4/c^2 L_0^2$, $\mu$ is given in Assumption \ref{assum:2}, $h_1$, $L_0$, $L_1$ are given in Assumption \ref{assum:3}, $c$ is given in Assumption \ref{assum:7}, and $\mathcal{Y}_J(\boldsymbol{x})$ is defined in (\ref{eq21}) with $c_0 = c$; let stepsize $\beta$ in (\ref{eq52}) be $\beta = 1/L_G$, then for any $K>0$, we have 
\begin{align*}
& G_{\epsilon, \tau, J}(\boldsymbol{x}, \boldsymbol{y}^*(\boldsymbol{x})) - G_{\epsilon, \tau, J}(\boldsymbol{x}, \boldsymbol{y}_K(\boldsymbol{x})) \le \frac{L_G}{2}\bigg(1 - \frac{\mu}{L_G}\bigg)^K \|\boldsymbol{y}_0(\boldsymbol{x})- \boldsymbol{y}^*(\boldsymbol{x})\|^2,\\
& \|\boldsymbol{y}_K(\boldsymbol{x}) - \boldsymbol{y}^*(\boldsymbol{x})\| \le \bigg(1 - \frac{\mu}{L_G}\bigg)^{\frac{K}{2}}\|\boldsymbol{y}_0(\boldsymbol{x}) - \boldsymbol{y}^*(\boldsymbol{x})\|
\end{align*}
where $\boldsymbol{y}_0(\boldsymbol{x})$ is the initial value, $\boldsymbol{y}_K(\boldsymbol{x})$ is the output of the $K$-th iteration in (\ref{eq52}), and $\boldsymbol{y}^*(\boldsymbol{x})$ is the optimal solution to problem (\ref{eq20}), i.e., 
\begin{equation}  \label{eq53} \tag{53}
\boldsymbol{y}^*(\boldsymbol{x}):=  { \underset {\boldsymbol{y}\in \mathcal{Y}_J(\boldsymbol{x})} { \operatorname {arg\,max} } \, G_{\epsilon, \tau, J}(\boldsymbol{x}, \boldsymbol{y})}.
\end{equation}
\end{lemma}

\begin{proof}
Given any $\tau \in (0, 1)$, positive integer $J$, $\boldsymbol{x}\in \mathcal{X}$. Firstly, we prove that $G_{\epsilon, \tau, J}(\boldsymbol{x}, \boldsymbol{y})$ is $L_G$ smooth w.r.t. $\boldsymbol{y}$ on the set $\mathcal{Y}_J(\boldsymbol{x})$. From (\ref{eq51}), it is easy to know that for any $\boldsymbol{x}\in \mathcal{X}$, $\boldsymbol{y}_1$, $\boldsymbol{y}_2 \in \mathcal{Y}_J(\boldsymbol{x})$, 
\begin{align*}
& \|\nabla_{\boldsymbol{y}} G_{\epsilon, \tau, J}(\boldsymbol{x}, \boldsymbol{y}_1) - \nabla_{\boldsymbol{y}}G_{\epsilon, \tau, J}(\boldsymbol{x}, \boldsymbol{y}_2)\|  \\
& \le \|\nabla_{\boldsymbol{y}}F(\boldsymbol{x}, \boldsymbol{y}_1) - \nabla_{\boldsymbol{y}}F(\boldsymbol{x}, \boldsymbol{y}_2)\|  \\
& ~~~+ \bigg\|\frac{\tau}{f_J(\boldsymbol{x}) + \epsilon - f(\boldsymbol{x}, \boldsymbol{y}_1)} \nabla_{\boldsymbol{y}}f(\boldsymbol{x}, \boldsymbol{y}_1) - \frac{\tau}{f_J(\boldsymbol{x}) + \epsilon - f(\boldsymbol{x}, \boldsymbol{y}_2)}\nabla_{\boldsymbol{y}}f(\boldsymbol{x}, \boldsymbol{y}_2)\bigg\| \\
& \overset{(i)}{\le}h_1 \|\boldsymbol{y}_1 - \boldsymbol{y}_2\| + \bigg\|\frac{\tau}{f_J(\boldsymbol{x}) + \epsilon - f(\boldsymbol{x}, \boldsymbol{y}_1)} \nabla_{\boldsymbol{y}}f(\boldsymbol{x}, \boldsymbol{y}_1) - \frac{\tau}{f_J(\boldsymbol{x}) + \epsilon - f(\boldsymbol{x}, \boldsymbol{y}_1)}\nabla_{\boldsymbol{y}}f(\boldsymbol{x}, \boldsymbol{y}_2)\bigg\|\\
& ~~~+ \bigg\| \frac{\tau}{f_J(\boldsymbol{x}) + \epsilon - f(\boldsymbol{x}, \boldsymbol{y}_1)}\nabla_{\boldsymbol{y}}f(\boldsymbol{x}, \boldsymbol{y}_2) - \frac{\tau}{f_{J}(\boldsymbol{x}) + \epsilon - f(\boldsymbol{x}, \boldsymbol{y}_2)}\nabla_{\boldsymbol{y}}f(\boldsymbol{x}, \boldsymbol{y}_2) \bigg\|  \\
& \overset{(ii)}{\le} h_1 \|\boldsymbol{y}_1 - \boldsymbol{y}_2\| + \frac{2\tau}{c}L_1\|\boldsymbol{y}_1 - \boldsymbol{y}_2\| + \tau L_0 \bigg|\frac{f(\boldsymbol{x}, \boldsymbol{y}_1) - f(\boldsymbol{x}, \boldsymbol{y}_2)}{(f_J(\boldsymbol{x}) + \epsilon - f(\boldsymbol{x}, \boldsymbol{y}_1))(f_J(\boldsymbol{x}) + \epsilon - f(\boldsymbol{x}, \boldsymbol{y}_2))}\bigg|  \\
& \overset{(iii)}{\le}\bigg(h_1 + \frac{2}{c}L_1 + \frac{4}{c^2}L_0^2\bigg)\|\boldsymbol{y}_1 - \boldsymbol{y}_2\|
\end{align*}
where $(i)$ follows from Assumption \ref{assum:3} and the triangle inequality, $(ii)$ follows from Assumption \ref{assum:3} and the fact that $\boldsymbol{y}_1 \in \mathcal{Y}_J(\boldsymbol{x})$, $(iii)$ follows from Assumption \ref{assum:3}, the fact that $\boldsymbol{y}_1$, $\boldsymbol{y}_2 \in \mathcal{Y}_J(\boldsymbol{x})$, and the fact that $0 < \tau < 1$. 

Therefore, $G_{\epsilon, \tau, J}(\boldsymbol{x}, \boldsymbol{y})$ is $L_G$ smooth w.r.t. $\boldsymbol{y}$ on the set $\mathcal{Y}_J(\boldsymbol{x})$. Furthermore, from Assumption \ref{assum:2}, it is easy to know that $G_{\epsilon, \tau, J}(\boldsymbol{x}, \boldsymbol{y})$ is $\mu$-strongly concave w.r.t. $\boldsymbol{y}$ on the set $\mathcal{Y}_J(\boldsymbol{x})$, and following the same discussion in Remark \ref{remark:1}, it is easy to know that $\mathcal{Y}_J(\boldsymbol{x})$ is a nonempty compact convex set. Then, from Lemma \ref{lemma:3}, for any $K>0$, we have 
\begin{align*}
& G_{\epsilon, \tau, J}(\boldsymbol{x}, \boldsymbol{y}^*(\boldsymbol{x})) - G_{\epsilon, \tau, J}(\boldsymbol{x}, \boldsymbol{y}_K(\boldsymbol{x})) \le \frac{L_G}{2}\bigg(1 - \frac{\mu}{L_G}\bigg)^K \big\|\boldsymbol{y}_0(\boldsymbol{x}) - \boldsymbol{y}^*(\boldsymbol{x})\big\|^2,  \\
&\|\boldsymbol{y}_K(\boldsymbol{x}) - \boldsymbol{y}^*(\boldsymbol{x})\| \le \bigg(1 - \frac{\mu}{L_G}\bigg)^{\frac{K}{2}}\|\boldsymbol{y}_0(\boldsymbol{x}) - \boldsymbol{y}^*(\boldsymbol{x})\|.
\end{align*}
Then the proof is complete.
\end{proof}

In the following, we show that function $\varphi_{\epsilon, \tau, J}(\boldsymbol{x})$ defined in (\ref{eq15}) is differentiable and Lipschitz smooth on $\mathcal{X}$.

\begin{lemma}  \label{lemma:7}
Suppose Assumptions \ref{assum:1}, \ref{assum:2}, \ref{assum:3}, \ref{assum:6}, and \ref{assum:7} hold. Then for any $\tau\in (0, 1)$, positive integer $J$, the function $\varphi_{\epsilon, \tau, J}(\boldsymbol{x})$ defined in (\ref{eq15}) is differentiable on $\mathcal{X}$, and for each $\boldsymbol{x}\in \mathcal{X}$, 
\begin{equation*}
\nabla \varphi_{\epsilon, \tau, J}(\boldsymbol{x}) = \nabla_{\boldsymbol{x}} G_{\epsilon, \tau, J}(\boldsymbol{x}, \boldsymbol{y}^*(\boldsymbol{x}))
\end{equation*}
where function $G_{\epsilon, \tau, J}(\boldsymbol{x}, \boldsymbol{y})$ is given in (\ref{eq16}), and $\boldsymbol{y}^*(\boldsymbol{x})$ is defined in (\ref{eq53}) with $\mathcal{Y}_J(\boldsymbol{x})$ given in (\ref{eq21}) with $c_0 = c$($c$ given in Assumption \ref{assum:7}); furthermore, function $\varphi_{\epsilon, \tau, J}(\boldsymbol{x})$ is $L_{\varphi}(J)$ smooth on $\mathcal{X}$, i.e., for any $\boldsymbol{x}_1$, $\boldsymbol{x}_2\in \mathcal{X}$, we have 
\begin{equation*}
\|\nabla \varphi_{\epsilon, \tau, J}(\boldsymbol{x}_1) - \nabla \varphi_{\epsilon, \tau, J}(\boldsymbol{x}_2)\| \le L_{\varphi}(J)\|\boldsymbol{x}_1 - \boldsymbol{x}_2\|
\end{equation*}
where $L_{\varphi}(J)= L_{11}(J)(1 + 1/\mu L_{12}(J))$, $L_{11}(J)$ and $L_{12}(J)$ are given in Lemma \ref{lemma:5}, and $\mu$ is given in Assumption \ref{assum:2}.
\end{lemma}

\begin{proof}
Given any $\tau\in (0,1)$, positive integer $J$. Let $\mathcal{Y}_J$ be the set defined in (19). we first prove the first conclusion. By the definition of $\varphi_{\epsilon, \tau, J}(\boldsymbol{x})$ in (15) and Assumption \ref{assum:7}, we know that for any $\boldsymbol{x}\in \mathcal{X}$, $\varphi_{\epsilon, \tau, J}(\boldsymbol{x})$ in (\ref{eq15}) can be written as 
\begin{equation}  \label{eq54}  \tag{54}
\varphi_{\epsilon, \tau, J}(\boldsymbol{x}) := \max\limits_{\boldsymbol{y}\in \mathcal{Y}_J}G_{\epsilon, \tau, J}(\boldsymbol{x}, \boldsymbol{y}) 
\end{equation}
where $G_{\epsilon, \tau, J}(\boldsymbol{x}, \boldsymbol{y})$ is given in (\ref{eq16}), and $\mathcal{Y}_J$ is nonempty, compact, and convex(see Remark \ref{remark:1}).

From Assumption \ref{assum:2}, it is easy to know that function $G_{\epsilon, \tau, J}(\boldsymbol{x}, \boldsymbol{y})$ is continuously differentiable on $\mathcal{X}\times \mathcal{Y}_J$, and for any $\boldsymbol{x}\in \mathcal{X}$, $G_{\epsilon, \tau, J}(\boldsymbol{x}, \cdot)$ is $\mu$-strongly concave w.r.t. $\boldsymbol{y}$ on $\mathcal{Y}_J$. Furthermore, for any $\boldsymbol{x}\in \mathcal{X}$, $\boldsymbol{y}^*(\boldsymbol{x})$ defined in (\ref{eq53}) is a maximal solution of function $G_{\epsilon, \tau, J}(\boldsymbol{x}, \boldsymbol{y})$ on the set $\mathcal{Y}_J$, i.e., 
\begin{equation}  \label{eq55}   \tag{55}
\boldsymbol{y}^*(\boldsymbol{x}) :=  { \underset {\boldsymbol{y}\in \mathcal{Y}_J} { \operatorname {arg\,max} } \, G_{\epsilon, \tau, J}(\boldsymbol{x}, \boldsymbol{y})}.
\end{equation}
In fact, given any $\boldsymbol{x}\in \mathcal{X}$. By using Assumption \ref{assum:7} and the fact that $\mathcal{Y}_J \subset \mathcal{Y}_J(\boldsymbol{x})\subset \mathcal{Y}$(see the definition of $\mathcal{Y}_J$ in (\ref{eq19}) and the definition of $\mathcal{Y}_J(\boldsymbol{x})$ in (\ref{eq21})), it is easy to know that 
\begin{equation*}
\max\limits_{\boldsymbol{y}\in \mathcal{Y}_J}G_{\epsilon, \tau, J}(\boldsymbol{x}, \boldsymbol{y}) = \max\limits_{\boldsymbol{y}\in \mathcal{Y}_J(\boldsymbol{x})}G_{\epsilon, \tau, J}(\boldsymbol{x}, \boldsymbol{y})
\end{equation*}
which is equal to the optimal value of problem (\ref{eq18}), and 
\begin{equation*}
{ \underset {\boldsymbol{y}\in \mathcal{Y}_J} { \operatorname {arg\,max} } \, G_{\epsilon, \tau, J}(\boldsymbol{x}, \boldsymbol{y})} \subset { \underset {\boldsymbol{y}\in \mathcal{Y}_J(\boldsymbol{x})} { \operatorname {arg\,max} } \, G_{\epsilon, \tau, J}(\boldsymbol{x}, \boldsymbol{y})}.
\end{equation*}
Since $\operatorname{arg\, max}_{\boldsymbol{y}\in \mathcal{Y}_J}G_{\epsilon, \tau, J}(\boldsymbol{x}, \boldsymbol{y})$ is nonempty(see Assumption \ref{assum:7}) and there exists an unique maximal solution for function $G_{\epsilon, \tau, J}(\boldsymbol{x}, \boldsymbol{y})$ on the set $\mathcal{Y}_J(\boldsymbol{x})$ due to the strong concavity of $G_{\epsilon, \tau, J}(\boldsymbol{x}, \boldsymbol{y})$ on the set $\mathcal{Y}_J(\boldsymbol{x})$(see Lemma \ref{lemma:6}), it is easy to know that the formula in (\ref{eq55}) holds. 

Then following from the Danskin's theorem in  \cite{13}, and using the equivalence form of function $\varphi_{\epsilon, \tau, J}(\boldsymbol{x})$ in (\ref{eq54}), we know that $\varphi_{\epsilon, \tau, J}(\boldsymbol{x})$ is differentiable on $\mathcal{X}$, and 
\begin{equation}  \label{eq56}\tag{56}
\nabla \varphi_{\epsilon, \tau, J}(\boldsymbol{x}) = \nabla_{\boldsymbol{x}}G_{\epsilon, \tau, J}(\boldsymbol{x}, \boldsymbol{y}^*(\boldsymbol{x}))
\end{equation}
with $\boldsymbol{y}^*(\boldsymbol{x})$ given in (\ref{eq53}).

In the following, the Lipschitz continuity of $\nabla \varphi_{\epsilon, \tau, J}(\boldsymbol{x})$ on $\mathcal{X}$ is proved. For any $\boldsymbol{x}_1$, $\boldsymbol{x}_2 \in \mathcal{X}$, from (\ref{eq56}), we know that 
\begin{equation}  \label{eq57} \tag{57}
\nabla \varphi_{\epsilon, \tau, J}(\boldsymbol{x}_1) = \nabla_{\boldsymbol{x}} G_{\epsilon, \tau, J}(\boldsymbol{x}_1, \boldsymbol{y}^*(\boldsymbol{x}_1)), \qquad \nabla \varphi_{\epsilon, \tau, J}(\boldsymbol{x}_2) = \nabla_{\boldsymbol{x}} G_{\epsilon, \tau, J}(\boldsymbol{x}_2, \boldsymbol{y}^*(\boldsymbol{x}_2))
\end{equation}
where $\boldsymbol{y}^*(\boldsymbol{x}_1)$ and $\boldsymbol{y}^*(\boldsymbol{x}_2)$ are given in (\ref{eq53}) with $\boldsymbol{x}= \boldsymbol{x}_1$ and $\boldsymbol{x}=\boldsymbol{x}_2$, respectively.

Since $\boldsymbol{y}^*(\boldsymbol{x}_1)$, $\boldsymbol{y}^*(\boldsymbol{x}_2) \in \mathcal{Y}_J$(see (\ref{eq55})), and $-G_{\epsilon, \tau, J}(\boldsymbol{x}, \cdot)$ is $\mu$-strongly convex w.r.t. $\boldsymbol{y}$ on $\mathcal{Y}_J$ for each $\boldsymbol{x}\in \mathcal{X}$, we have 
\begin{align*} 
G_{\epsilon, \tau, J}(\boldsymbol{x}_2, \boldsymbol{y}^*(\boldsymbol{x}_2)) &\le G_{\epsilon, \tau, J}(\boldsymbol{x}_2, \boldsymbol{y}^*(\boldsymbol{x}_1)) + \langle \nabla_{\boldsymbol{y}}G_{\epsilon, \tau, J}(\boldsymbol{x}_2, \boldsymbol{y}^*(\boldsymbol{x}_1)), \boldsymbol{y}^*(\boldsymbol{x}_2) - \boldsymbol{y}^*(\boldsymbol{x}_1)\rangle \\
&- \frac{\mu}{2}\|\boldsymbol{y}^*(\boldsymbol{x}_2) - \boldsymbol{y}^*(\boldsymbol{x}_1)\|^2,  
\end{align*}
and 
\begin{align*}
G_{\epsilon, \tau, J}(\boldsymbol{x}_2, \boldsymbol{y}^*(\boldsymbol{x}_1)) &\le G_{\epsilon, \tau, J}(\boldsymbol{x}_2, \boldsymbol{y}^*(\boldsymbol{x}_2)) + \langle \nabla_{\boldsymbol{y}}G_{\epsilon, \tau, J}(\boldsymbol{x}_2, \boldsymbol{y}^*(\boldsymbol{x}_2)), \boldsymbol{y}^*(\boldsymbol{x}_1) - \boldsymbol{y}^*(\boldsymbol{x}_2)\rangle \\
&- \frac{\mu}{2}\|\boldsymbol{y}^*(\boldsymbol{x}_2) - \boldsymbol{y}^*(\boldsymbol{x}_1)\|^2.
\end{align*}
From the above two inequalities, we have 
\begin{equation}  \label{eq58}  \tag{58}
\langle \nabla_{\boldsymbol{y}} G_{\epsilon, \tau, J}(\boldsymbol{x}_2, \boldsymbol{y}^*(\boldsymbol{x}_1)), \boldsymbol{y}^*(\boldsymbol{x}_2) - \boldsymbol{y}^*(\boldsymbol{x}_1)\rangle \ge \mu \|\boldsymbol{y}^*(\boldsymbol{x}_2) - \boldsymbol{y}^*(\boldsymbol{x}_1)\|^2
\end{equation}
where we use the fact that $\langle \nabla_{\boldsymbol{y}} G_{\epsilon, \tau, J}(\boldsymbol{x}_2, \boldsymbol{y}^*(\boldsymbol{x}_2)), \boldsymbol{y}^*(\boldsymbol{x}_1) - \boldsymbol{y}^*(\boldsymbol{x}_2) \rangle \le 0$ since $G_{\epsilon, \tau, J}(\boldsymbol{x}_2, \cdot)$ is strongly concave w.r.t. $\boldsymbol{y}$ on $\mathcal{Y}_J$ and $\boldsymbol{y}^*(\boldsymbol{x}_2) :=  \operatorname {arg\,max}_{\boldsymbol{y}\in \mathcal{Y}_J}G_{\epsilon, \tau, J}(\boldsymbol{x}_2, \boldsymbol{y})$. Furthermore, since $G_{\epsilon, \tau, J}(\boldsymbol{x}_1, \cdot)$ is strongly concave w.r.t. $\boldsymbol{y}$ on $\mathcal{Y}_J$, and $\boldsymbol{y}^*(\boldsymbol{x}_1) :=  \operatorname {arg\,max}_{\boldsymbol{y}\in \mathcal{Y}_J}G_{\epsilon, \tau, J}(\boldsymbol{x}_1, \boldsymbol{y})$, we have 
\begin{equation}   \label{eq59}  \tag{59}
\langle \nabla_{\boldsymbol{y}}G_{\epsilon, \tau, J}(\boldsymbol{x}_1, \boldsymbol{y}^*(\boldsymbol{x}_1)), \boldsymbol{y}^*(\boldsymbol{x}_2)- \boldsymbol{y}^*(\boldsymbol{x}_1) \rangle \le 0.
\end{equation}
Combining (\ref{eq58}) with (\ref{eq59}), we have 
\begin{align*}
\mu \|\boldsymbol{y}^*(\boldsymbol{x}_1) - \boldsymbol{y}^*(\boldsymbol{x}_2)\|^2 & \le \big\langle \nabla_{\boldsymbol{y}} G_{\epsilon, \tau, J}(\boldsymbol{x}_2, \boldsymbol{y}^*(\boldsymbol{x}_1)) - \nabla_{\boldsymbol{y}}G_{\epsilon, \tau, J}(\boldsymbol{x}_1, \boldsymbol{y}^*(\boldsymbol{x}_1)), \boldsymbol{y}^*(\boldsymbol{x}_2) - \boldsymbol{y}^*(\boldsymbol{x}_1) \big\rangle  \\
&  \overset{(i)}{\le}L_{12}(J)\|\boldsymbol{x}_1 - \boldsymbol{x}_2\| \|\boldsymbol{y}^*(\boldsymbol{x}_1)- \boldsymbol{y}^*(\boldsymbol{x}_2)\|
\end{align*}
where $(i)$ follows from Lemma \ref{lemma:5} and the Cauchy-Schwartz inequality. Then we have 
\begin{equation}   \label{eq60}  \tag{60}
\|\boldsymbol{y}^*(\boldsymbol{x}_1) - \boldsymbol{y}^*(\boldsymbol{x}_2)\| \le \frac{1}{\mu}L_{12}(J)\|\boldsymbol{x}_1 - \boldsymbol{x}_2\|.
\end{equation}

Based on the above discussion, for the $\nabla \varphi_{\epsilon, \tau, J}(\boldsymbol{x}_1)$ and $\nabla \varphi_{\epsilon, \tau, J}(\boldsymbol{x}_2)$ in (\ref{eq57}), we have 
\begin{align*}
& \|\nabla \varphi_{\epsilon, \tau, J}(\boldsymbol{x}_1) - \nabla \varphi_{\epsilon, \tau, J}(\boldsymbol{x}_2)\|  \\
&= \|\nabla_{\boldsymbol{x}} G_{\epsilon, \tau, J}(\boldsymbol{x}_1, \boldsymbol{y}^*(\boldsymbol{x}_1)) - \nabla_{\boldsymbol{x}}G_{\epsilon, \tau, J}(\boldsymbol{x}_2, \boldsymbol{y}^*(\boldsymbol{x}_2))\| \\
& \le \|\nabla_{\boldsymbol{x}} G_{\epsilon, \tau, J}(\boldsymbol{x}_1, \boldsymbol{y}^*(\boldsymbol{x}_1)) - \nabla_{\boldsymbol{x}}G_{\epsilon, \tau, J}(\boldsymbol{x}_2, \boldsymbol{y}^*(\boldsymbol{x}_1))\| \\
&~~~+ \|\nabla_{\boldsymbol{x}}G_{\epsilon, \tau, J}(\boldsymbol{x}_2, \boldsymbol{y}^*(\boldsymbol{x}_1)) - \nabla_{\boldsymbol{x}}G_{\epsilon, \tau, J}(\boldsymbol{x}_2, \boldsymbol{y}^*(\boldsymbol{x}_2))\| \\
& \overset{(i)}{\le} L_{11}(J)\|\boldsymbol{x}_1 - \boldsymbol{x}_2\| + L_{11}(J)\|\boldsymbol{y}^*(\boldsymbol{x}_1) - \boldsymbol{y}^*(\boldsymbol{x}_2)\|  \\
& \overset{(ii)} {\le} L_{11}(J)\bigg(1 + \frac{1}{\mu}L_{12}(J)\bigg)\|\boldsymbol{x}_1 - \boldsymbol{x}_2\|
\end{align*}
where $(i)$ follows from Lemma \ref{lemma:5}, $(ii)$ follows from the inequality in (\ref{eq60}). Then the proof is complete.
\end{proof}

In the following, we show that $\varphi_{\epsilon, \tau, J}(\boldsymbol{x})$ is bounded w.r.t. $\tau$, $J$, and $\boldsymbol{x}$.

\begin{lemma} \label{lemma:8}
Suppose Assumptions \ref{assum:1}, \ref{assum:2}, and \ref{assum:7} hold. Define
\begin{equation*}
M_0 := \min\limits_{\boldsymbol{x}\in \mathcal{X}, \boldsymbol{y}\in \mathcal{Y}} F(\boldsymbol{x}, \boldsymbol{y}), \qquad M_1 := \max\limits_{\boldsymbol{x}\in \mathcal{X},\boldsymbol{y}\in \mathcal{Y}} F(\boldsymbol{x}, \boldsymbol{y}), \qquad M_2 := \max\limits_{\boldsymbol{x}\in \mathcal{X}, \boldsymbol{y}\in \mathcal{Y}}|f(\boldsymbol{x}, \boldsymbol{y})|.
\end{equation*}
Then for any $\tau\in (0, 1)$, positive integer $J$, $\boldsymbol{x}\in \mathcal{X}$, we have 
\begin{equation*}
|\varphi_{\epsilon,\tau, J}(\boldsymbol{x})| \le M_3,
\end{equation*}
where $\varphi_{\epsilon, \tau, J}(\boldsymbol{x})$ is given in (15), $M_3 = \max\big\{|M_0 + \ln{c}|, |M_1 + \ln(2M_2 + \epsilon)|\big\}$, and $c$ is given in Assumption \ref{assum:7}.
\end{lemma}

\begin{proof}
Given any $\tau \in (0, 1)$, positive integer $J$, $\boldsymbol{x}\in \mathcal{X}$. Recall that 
\begin{equation*}
\varphi_{\epsilon, \tau, J}(\boldsymbol{x}) = \max\limits_{\boldsymbol{y}\in \mathcal{Y}_J} G_{\epsilon, \tau, J}(\boldsymbol{x}, \boldsymbol{y})
\end{equation*}
where $\mathcal{Y}_J$ is given in (\ref{eq19})(see the discussion in Lemma \ref{lemma:7}). By the definition of $G_{\epsilon, \tau, J}(\boldsymbol{x}, \boldsymbol{y})$ in (\ref{eq16}), we know that for any $\boldsymbol{y}\in \mathcal{Y}_J$, 
\begin{equation*}
F(\boldsymbol{x}, \boldsymbol{y}) + \tau \ln{c} \le G_{\epsilon, \tau, J}(\boldsymbol{x},\boldsymbol{y}) \le F(\boldsymbol{x}, \boldsymbol{y}) + \tau \ln (2M_2 + \epsilon).
\end{equation*}

Furthermore, since $0< \tau <1$, and $c<1$(see Assumption \ref{assum:7}), we know that 
\begin{equation*}
M_0 + \ln {c} \le G_{\epsilon, \tau, J}(\boldsymbol{x}, \boldsymbol{y}) \le M_1 + \ln(2M_2 +\epsilon)
\end{equation*}
and thus for any $\tau \in (0, 1)$, positive integer $J$, $\boldsymbol{x}\in \mathcal{X}$, we have $|\varphi_{\epsilon,\tau, J}(\boldsymbol{x})| \le M_3$. The proof is complete.
\end{proof}

\begin{lemma}  \label{lemma:9}
Suppose Assumptions \ref{assum:1}, \ref{assum:2}, \ref{assum:3}, \ref{assum:6}, and \ref{assum:7} hold. Assume that for any $\tau\in (0,1)$, positive integer $J$, $\boldsymbol{x}\in \mathcal{X}$, $\boldsymbol{y}^*(\boldsymbol{x})$ defined in (\ref{eq53}) is an interior point of $\mathcal{Y}$. Then, for any $\tau\in (0,1)$, positive integer $J$, $\boldsymbol{x}\in \mathcal{X}$, $\boldsymbol{y}^*(\boldsymbol{x})$ defined in  (\ref{eq53}) is an interior point of the set $\mathcal{Y}_J(\boldsymbol{x})$ defined in (\ref{eq21}) with $c_0 = c$($c$ given in Assumption \ref{assum:7}).
\end{lemma}

\begin{proof}
Given any $\tau \in (0, 1)$, positive integer $J$, $\boldsymbol{x}\in \mathcal{X}$. Based on the discussion in Lemma \ref{lemma:7}, we know that $\boldsymbol{y}^*(\boldsymbol{x})\in \mathcal{Y}_J$(see (\ref{eq55})), and thus we have
\begin{equation}  \label{eq61}  \tag{61}
\boldsymbol{y}^*(\boldsymbol{x})\in \mathcal{Y}, \qquad f_{J}(\boldsymbol{x})+ \epsilon - f(\boldsymbol{x}, \boldsymbol{y}^*(\boldsymbol{x})) \ge c
\end{equation}
which is obtained by the definition of $\mathcal{Y}_J$ in (\ref{eq19}). Since $f(\boldsymbol{x}, \cdot)$ is continuous w.r.t. $\boldsymbol{y}$ on $\mathcal{Y}$, we know that there exists a $\delta > 0$ such that for any $\boldsymbol{y}\in \mathbb{B}_{\delta}(\boldsymbol{y}^*(\boldsymbol{x})) \cap \mathcal{Y}$, we have 
\begin{equation}  \label{eq62} \tag{62}
f(\boldsymbol{x}, \boldsymbol{y}) \le f(\boldsymbol{x}, \boldsymbol{y}^*(\boldsymbol{x})) + \frac{c}{2}.
\end{equation}

Furthermore, since $\boldsymbol{y}^*(\boldsymbol{x})$ is an interior point of $\mathcal{Y}$(see the condition of Lemma \ref{lemma:9}), we know that there exists a $\delta_1 > 0$($\delta_1 \le \delta$) such that $\mathbb{B}_{\delta_1}(\boldsymbol{y}^*(\boldsymbol{x})) \subset \mathbb{B}_{\delta}(\boldsymbol{y}^*(\boldsymbol{x}))$ and $\mathbb{B}_{\delta_1}(\boldsymbol{y}^*(\boldsymbol{x})) \subset \mathcal{Y}$, and thus from (\ref{eq62}), we know that for any $\boldsymbol{y}\in \mathbb{B}_{\delta_1}(\boldsymbol{y}^*(\boldsymbol{x}))$, we have 
\begin{equation}  \label{eq63}  \tag{63}
\boldsymbol{y} \in \mathcal{Y}, \qquad f(\boldsymbol{x}, \boldsymbol{y})\le  f(\boldsymbol{x}, \boldsymbol{y}^*(\boldsymbol{x})) + \frac{c}{2}.
\end{equation}

By combining (\ref{eq61}) with (\ref{eq63}), it is easy to know that for any $\boldsymbol{y} \in \mathbb{B}_{\delta_1}(\boldsymbol{y}^*(\boldsymbol{x}))$,  we have $\boldsymbol{y}\in \mathcal{Y}$, $f_J(\boldsymbol{x}) + \epsilon - f(\boldsymbol{x}, \boldsymbol{y}) \ge c/2$, and thus $\mathbb{B}_{\delta_1}(\boldsymbol{y}^*(\boldsymbol{x})) \subset \mathcal{Y}_J(\boldsymbol{x})$. Then the proof is complete.
\end{proof}

\subsection{Proof of Theorem \ref{tho:3}}

Then, the proof of Theorem \ref{tho:3} is shown below.

\begin{proof}
Notice that for any $\sigma \in (0, 1)$, positive integer $J$, for the $\tau$ given in (\ref{eq27}), we have $\tau \in (0, 1)$, which can be obtained by using the definition of $c$ in Assumption \ref{assum:7}, and the definition of $\bar{l}(J)$ in Theorem \ref{tho:3}. Firstly, we prove that for any $\boldsymbol{x}_t\in \mathcal{X}$, $\boldsymbol{y}_K(\boldsymbol{x}_t)$, where $t \in \{0, \ldots, T-1\}$ and $\boldsymbol{y}_K(\boldsymbol{x}_t)$ is the output of the $K$-th iteration in line 10 of Algorithm \ref{alg:2} to solve problem (\ref{eq20}) with $\boldsymbol{x}=\boldsymbol{x}_t$, we have 
\begin{equation} \label{eq64}  \tag{64}
\|\nabla_{\boldsymbol{y}} G_{\epsilon, \tau, J}(\boldsymbol{x}_t, \boldsymbol{y}_K(\boldsymbol{x}_t))\| \le \sigma  
\end{equation}
where $G_{\epsilon, \tau, J}(\boldsymbol{x}, \boldsymbol{y})$ is defined in (16). In fact, for any $\boldsymbol{x_t}\in \mathcal{X}$, $\boldsymbol{y}_K(\boldsymbol{x}_t)$, we have 
\begin{align*}
\|\nabla_{\boldsymbol{y}} G_{\epsilon, \tau, J}(\boldsymbol{x}_t, \boldsymbol{y}_K(\boldsymbol{x}_t))\| &\le \|\nabla_{\boldsymbol{y}}G_{\epsilon, \tau, J}(\boldsymbol{x}_t, \boldsymbol{y}^*(\boldsymbol{x}_t))\| \\
& ~~~+ \| \nabla_{\boldsymbol{y}} G_{\epsilon, \tau, J}(\boldsymbol{x}_t, \boldsymbol{y}_K(\boldsymbol{x}_t)) - \nabla_{\boldsymbol{y}}G_{\epsilon, \tau, J}(\boldsymbol{x}_t, \boldsymbol{y}^*(\boldsymbol{x}_t))\|   \\
& \overset{(i)}{\le} L_G \|\boldsymbol{y}_K(\boldsymbol{x}_t) - \boldsymbol{y}^*(\boldsymbol{x}_t)\|  \\
& \overset{(ii)}{\le} L_G\bigg(1 - \frac{\mu}{L_G}\bigg)^{\frac{K}{2}}\|\boldsymbol{y}_0(\boldsymbol{x}_t) - \boldsymbol{y}^*(\boldsymbol{x}_t)\|   \\
&\overset{(iii)}{\le} 2ML_G\bigg(1 - \frac{\mu}{L_G}\bigg)^{\frac{K}{2}}   \overset{(iv)}{\le} \sigma
\end{align*}
where $(i)$ uses the fact that $\nabla_{\boldsymbol{y}}G_{\epsilon, \tau, J}(\boldsymbol{x}_t, \boldsymbol{y}^*(\boldsymbol{x}_t))=0$, which can be obtained by using the definition of $\boldsymbol{y}^*(\boldsymbol{x})$ in Theorem \ref{tho:3} and the conclusions in Lemma \ref{lemma:9}, and uses the conclusion in Lemma \ref{lemma:6}, $(ii)$ uses the conclusion in Lemma \ref{lemma:6}, $(iii)$ follows from Assumption \ref{assum:1}, $(iv)$ follows from the choice of $K$ in the condition of Theorem \ref{tho:3}, the fact that $\mu/L_G \le 1$ which can be obtained by using  the definition of $L_G$ in Lemma \ref{lemma:6}, and the definition of $\bar{l}(J)$ in Theorem \ref{tho:3}.

In the following, we prove that during the iterations in Algorithm \ref{alg:2}, there exists $t_1 \in \{0,\ldots, T-1\}$, such that 
\begin{equation*}
\langle \nabla_{\boldsymbol{x}}G_{\epsilon, \tau, J}(\boldsymbol{x}_{t_1}, \boldsymbol{y}_K(\boldsymbol{x}_{t_1})), \boldsymbol{x}- \boldsymbol{x}_{t_1} \rangle  \ge -\sigma, ~ \forall \boldsymbol{x}\in \mathcal{X}.
\end{equation*}

For any $\boldsymbol{x}_t$, $\boldsymbol{y}_K(\boldsymbol{x}_t)$, with $t\in \{0, \ldots, T-1\}$, by the definition of $G_{\epsilon, \tau, J}(\boldsymbol{x}, \boldsymbol{y})$ in (\ref{eq16}) and the definition of $\boldsymbol{a}_t$ in (\ref{eq22}), we have 
\begin{equation}  \label{eq65}  \tag{65}
\nabla_{\boldsymbol{x}} G_{\epsilon, \tau, J}(\boldsymbol{x}_t, \boldsymbol{y}_K(\boldsymbol{x}_t)) = \boldsymbol{a}_t + \boldsymbol{b}
\end{equation}
where 
\begin{equation}  \label{eq66}  \tag{66}
\boldsymbol{b}= \frac{\tau}{f_J(\boldsymbol{x}_t) + \epsilon - f(\boldsymbol{x}_t, \boldsymbol{y}_K(\boldsymbol{x}_t))} (\nabla f_J(\boldsymbol{x}_t) - \nabla_{\boldsymbol{x}}f(\boldsymbol{x}_t, \boldsymbol{y}_J(\boldsymbol{x}_t))).
\end{equation}
Then for any $\boldsymbol{x}\in \mathcal{X}$, we have 
\begin{align}  \label{eq67} \tag{67}
\langle \nabla_{\boldsymbol{x}} G_{\epsilon, \tau, J}(\boldsymbol{x}_t, \boldsymbol{y}_K(\boldsymbol{x}_t)), \boldsymbol{x}-\boldsymbol{x}_t \rangle &\overset{(i)}{=}\langle \boldsymbol{a}_t, \boldsymbol{x}-\boldsymbol{x}_t\rangle + \langle \boldsymbol{b}, \boldsymbol{x}- \boldsymbol{x}_t  \rangle    \\
& = \langle \boldsymbol{a}_t, \boldsymbol{x}-\boldsymbol{x}_{t+1} \rangle +  \langle \boldsymbol{a}_t, \boldsymbol{x}_{t+1} - \boldsymbol{x}_t\rangle + \langle \boldsymbol{b}, \boldsymbol{x} - \boldsymbol{x}_t \rangle    \nonumber
\end{align}
where $(i)$ uses the formula in (\ref{eq65}).

In the following, we first bound the three terms $\langle \boldsymbol{a}_t, \boldsymbol{x}- \boldsymbol{x}_{t+1}\rangle$, $\boldsymbol{a}_t$, and $\boldsymbol{b}$ in (\ref{eq67}).

Notice that $\boldsymbol{x}_{t+1} = \text{proj}_{\mathcal{X}}(\boldsymbol{x}_t - \eta \boldsymbol{a}_t)$(see line 12 in Algorithm \ref{alg:2}). By the definition of $\text{proj}_{\mathcal{X}}(\cdot)$, we have 
\begin{equation*}
\langle \boldsymbol{x}_t - \eta \boldsymbol{a}_t - \boldsymbol{x}_{t+1}, \boldsymbol{x}- \boldsymbol{x}_{t+1}\rangle  \le 0, \qquad  \forall \boldsymbol{x}\in \mathcal{X},
\end{equation*}
i.e., 
\begin{equation}   \label{eq68} \tag{68}
\langle \boldsymbol{a}_t, \boldsymbol{x}- \boldsymbol{x}_{t+1}\rangle \ge \frac{1}{\eta}\langle \boldsymbol{x}_t - \boldsymbol{x}_{t+1}, \boldsymbol{x}- \boldsymbol{x}_{t+1}\rangle
\end{equation}
for any $\boldsymbol{x}\in \mathcal{X}$. 

By the definition of $\boldsymbol{a}_t$ in (\ref{eq22}), we have 
\begin{align}  \label{eq69}    \nonumber
\|\boldsymbol{a}_t\| &\le \|\nabla_{\boldsymbol{x}} F(\boldsymbol{x}_t, \boldsymbol{y}_K(\boldsymbol{x}_t))\| +\frac{\tau}{f_J(\boldsymbol{x}_t) + \epsilon - f(\boldsymbol{x}_t, \boldsymbol{y}_K(\boldsymbol{x}_t))} \|\nabla_{\boldsymbol{x}} f(\boldsymbol{x}_t, \boldsymbol{y}_J(\boldsymbol{x}_t)) - \nabla_{\boldsymbol{x}}f(\boldsymbol{x}_t, \boldsymbol{y}_K(\boldsymbol{x}_t))\|  \\   \tag{69}
& \overset{(i)}{\le} \|\nabla_{\boldsymbol{x}} F(\boldsymbol{x}_t, \boldsymbol{y}_K(\boldsymbol{x}_t))\| + \frac{2\tau}{c}\big(\|\nabla_{\boldsymbol{x}} f(\boldsymbol{x}_t, \boldsymbol{y}_J(\boldsymbol{x}_t))\| + \|\nabla_{\boldsymbol{x}} f(\boldsymbol{x}_t, \boldsymbol{y}_K(\boldsymbol{x}_t))\|\big)  \\   \nonumber
& \overset{(ii)}{\le} h_0 + \frac{4\tau}{c}L_0 \overset{(iii)}{\le} h_0 + \frac{4}{c}L_0   
\end{align} 
where $(i)$ uses the fact that $\boldsymbol{y}_K(\boldsymbol{x}_t) \in \mathcal{Y}_J(\boldsymbol{x}_t)$, $(ii)$ follows from Assumption \ref{assum:3}, $(iii)$ follows from the fact that $ \tau \in (0, 1)$.

For $\boldsymbol{b}$ in (\ref{eq66}), we have 
\begin{align}  \label{eq70}  \nonumber
\|\boldsymbol{b}\| &\le \frac{2\tau}{c}\|\nabla f_J(\boldsymbol{x}_t) - \nabla_{\boldsymbol{x}}f(\boldsymbol{x}_t, \boldsymbol{y}_J(\boldsymbol{x}_t))\| \\   \nonumber
& \overset{(i)}{=}\frac{2\tau}{c}\|(\nabla \boldsymbol{y}_J(\boldsymbol{x}_t))^\top \nabla_{\boldsymbol{y}}f(\boldsymbol{x}_t, \boldsymbol{y}_J(\boldsymbol{x}_t))\|  \\ \tag{70}
& \overset{(ii)}{\le} \frac{2\tau}{c}L_0 M_0(J) \overset{(iii)}{\le}\frac{\sigma^2}{18H\bar{l}(J)} \overset{(iv)}{\le} \frac{\sigma}{6H}
\end{align}
where $(i)$ uses the definition of $f_J(\boldsymbol{x})$ in (\ref{eq13}), $(ii)$ follows from Assumption \ref{assum:3} and Lemma \ref{lemma:4}, $(iii)$ uses the definition of $M_0(J)$ in Lemma \ref{lemma:4} and the choice of $\tau$ in the condition of Theorem \ref{tho:3}, $(iv)$ uses the definition of $\bar{l}(J)$ in Theorem \ref{tho:3} and the fact that $\sigma\in (0, 1)$.

Combining (\ref{eq67}), (\ref{eq68}), (\ref{eq69}) with (\ref{eq70}), we know that for any $\boldsymbol{x}\in \mathcal{X}$, 
\begin{align}   \label{eq71}  \tag{71}
\langle  \nabla_{\boldsymbol{x}}G_{\epsilon, \tau, J}(\boldsymbol{x}_t, \boldsymbol{y}_K(\boldsymbol{x}_t)), \boldsymbol{x}-\boldsymbol{x}_t \rangle  \overset{(i)}{\ge} -\frac{2H}{\eta}\|\boldsymbol{x}_t - \boldsymbol{x}_{t+1}\| - \bigg(h_0 + \frac{4}{c}L_0\bigg)\|\boldsymbol{x}_t - \boldsymbol{x}_{t+1}\| - \frac{\sigma}{3}
\end{align}
where $(i)$ uses Assumption \ref{assum:1}.

Define
\begin{equation}  \label{eq72}  \tag{72}
\boldsymbol{g}_t = \max\limits_{\boldsymbol{x}\in \mathcal{X}}-\langle \nabla_{\boldsymbol{x}} G_{\epsilon, \tau, J}(\boldsymbol{x}_t, \boldsymbol{y}_K(\boldsymbol{x}_t)), \boldsymbol{x}-\boldsymbol{x}_t\rangle
\end{equation}
then from (\ref{eq71}), we have 
\begin{equation*}  
\boldsymbol{g}_t \le \bigg(\frac{2H}{\eta} + h_0 + \frac{4}{c}L_0\bigg)\|\boldsymbol{x}_t - \boldsymbol{x}_{t+1}\| + \frac{\sigma}{3}
\end{equation*}
and thus
\begin{equation}  \label{eq73} \tag{73}
\|\boldsymbol{x}_t - \boldsymbol{x}_{t+1}\|  \ge \frac{\boldsymbol{g}_t - {\sigma}/{3}}{{2H}/{\eta} + h_0 + 4/c L_0}.
\end{equation}

Since $\varphi_{\epsilon, \tau, J}(\boldsymbol{x})$ is $L_{\varphi}(J)$ smooth(see Lemma \ref{lemma:7}), we have 
\begin{align*}
\varphi_{\epsilon, \tau, J}(\boldsymbol{x}_{t+1}) & \le \varphi_{\epsilon, \tau, J}(\boldsymbol{x}_t) + \langle \nabla \varphi_{\epsilon, \tau, J}(\boldsymbol{x}_t), \boldsymbol{x}_{t+1} - \boldsymbol{x}_t \rangle  + \frac{L_{\varphi}(J)}{2}\|\boldsymbol{x}_t - \boldsymbol{x}_{t+1}\|^2   \\
& =\varphi_{\epsilon, \tau, J}(\boldsymbol{x}_t) + \langle \nabla \varphi_{\epsilon, \tau, J}(\boldsymbol{x}_t) - \boldsymbol{a}_t, \boldsymbol{x}_{t+1}-\boldsymbol{x}_t\rangle + \langle \boldsymbol{a}_t, \boldsymbol{x}_{t+1}-\boldsymbol{x}_t \rangle + \frac{L_{\varphi}(J)}{2}\|\boldsymbol{x}_{t} - \boldsymbol{x}_{t+1}\|^2  \\
& \overset{(i)}{\le} \varphi_{\epsilon, \tau, J}(\boldsymbol{x}_t) + \langle \nabla \varphi_{\epsilon, \tau, J}(\boldsymbol{x}_t) - \boldsymbol{a}_t, \boldsymbol{x}_{t+1} - \boldsymbol{x}_t\rangle - \frac{1}{\eta}\|\boldsymbol{x}_{t}- \boldsymbol{x}_{t+1}\|^2 + \frac{L_{\varphi}(J)}{2}\|\boldsymbol{x}_t - \boldsymbol{x}_{t+1}\|^2\\
& \overset{(ii)}{\le}\varphi_{\epsilon, \tau, J}(\boldsymbol{x}_t) + 2H\|\nabla \varphi_{\epsilon, \tau, J}(\boldsymbol{x}_t) - \boldsymbol{a}_t\| - \bigg(\frac{1}{\eta}- \frac{L_{\varphi}(J)}{2}\bigg)\|\boldsymbol{x}_t - \boldsymbol{x}_{t+1}\|^2
\end{align*}
where $(i)$ uses the inequality in (\ref{eq68}), $(ii)$ follows from Assumption \ref{assum:1}. Then, we have 
\begin{equation}  \label{eq74} \tag{74}
\bigg(\frac{1}{\eta} - \frac{L_{\varphi}(J)}{2}\bigg)\|\boldsymbol{x}_t - \boldsymbol{x}_{t+1}\|^2 \le \varphi_{\epsilon, \tau, J}(\boldsymbol{x}_t)- \varphi_{\epsilon, \tau, J}(\boldsymbol{x}_{t+1}) + 2 H\|\nabla \varphi_{\epsilon, \tau, J}(\boldsymbol{x}_t) - \boldsymbol{a}_t\|.
\end{equation}

In the following, we consider two cases:
\begin{enumerate}
\item[(i)] If there exists $t \in \{0, \ldots, T-1\}$ such that $\boldsymbol{g}_t \le \sigma/3$, by setting $t_1 = t$, we have 
\begin{equation*}
\langle \nabla_{\boldsymbol{x}}G_{\epsilon, \tau, J}(\boldsymbol{x}_{t_1}, \boldsymbol{y}_K(\boldsymbol{x}_{t_1})), \boldsymbol{x}- \boldsymbol{x}_{t_1}\rangle \ge -\sigma, \qquad \forall \boldsymbol{x}\in \mathcal{X},
\end{equation*}
which can be obtained by using the definition of $\boldsymbol{g}_t$ in (\ref{eq72}), and it is proved.

\item[(ii)] If $\boldsymbol{g}_t >  \sigma/3$ for any $t \in \{0, \ldots, T-1\}$, then for any $t\in \{0, \ldots, T-1\}$, by combining (\ref{eq73}) with (\ref{eq74}), we have
\begin{equation*}
\bigg(\frac{1}{\eta} - \frac{L_{\varphi}(J)}{2}\bigg) \frac{(\boldsymbol{g}_t - \sigma/3)^2}{(2H/\eta + h_0 + 4/c L_0)^2} \le \varphi_{\epsilon, \tau, J}(\boldsymbol{x}_{t})- \varphi_{\epsilon, \tau, J}(\boldsymbol{x}_{t+1}) + 2 H \|\nabla \varphi_{\epsilon, \tau, J}(\boldsymbol{x}_t) - \boldsymbol{a}_t\| 
\end{equation*}
and thus 
\begin{equation}  \label{eq75} \tag{75}
\bigg(\boldsymbol{g}_t - \frac{\sigma}{3}\bigg)^2 \overset{(i)}{\le }\bar{l}(J)(\varphi_{\epsilon, \tau, J}(\boldsymbol{x}_t)- \varphi_{\epsilon, \tau, J}(\boldsymbol{x}_{t+1})) + 2H\bar{l}(J)\|\nabla \varphi_{\epsilon, \tau, J}(\boldsymbol{x}_t) - \boldsymbol{a}_t\|
\end{equation}
where $(i)$ uses the choice of $\eta$ and the definition of $\bar{l}(J)$ in Theorem \ref{tho:3}. 

Moreover, by the definition of $\nabla \varphi_{\epsilon, \tau, J}(\boldsymbol{x}_t)$ in Lemma \ref{lemma:7} and the formula in (\ref{eq65}), we have 
\begin{align} \label{eq76} \nonumber
\|\nabla \varphi_{\epsilon, \tau, J}(\boldsymbol{x}_t) - \boldsymbol{a}_t\| &\le \|\nabla_{\boldsymbol{x}} G_{\epsilon, \tau, J}(\boldsymbol{x}_t, \boldsymbol{y}^*(\boldsymbol{x}_t)) - \nabla_{\boldsymbol{x}}G_{\epsilon, \tau, J}(\boldsymbol{x}_t, \boldsymbol{y}_K(\boldsymbol{x}_t))\| \\  \nonumber
&+ \bigg\|\frac{\tau}{f_J(\boldsymbol{x}_t) +\epsilon - f(\boldsymbol{x}_t, \boldsymbol{y}_K(\boldsymbol{x}_t))}(\nabla f_J(\boldsymbol{x}_t) - \nabla_{\boldsymbol{x}} f(\boldsymbol{x}_t, \boldsymbol{y}_J(\boldsymbol{x}_t)))\bigg\|  \\   \nonumber
&\overset{(i)}{\le}L_{11}(J)\|\boldsymbol{y}_K(\boldsymbol{x}_t) - \boldsymbol{y}^*(\boldsymbol{x}_t)\| + \frac{2\tau}{c}\bigg\|(\nabla \boldsymbol{y}_J(\boldsymbol{x}_t))^\top \nabla_{\boldsymbol{y}}f(\boldsymbol{x}_t, \boldsymbol{y}_J(\boldsymbol{x}_t))\bigg\|  \\  \tag{76}
& \overset{(ii)}{\le} L_{11}(J)\bigg(1 - \frac{\mu}{L_G}\bigg)^{\frac{K}{2}}\|\boldsymbol{y}_0(\boldsymbol{x}_t) - \boldsymbol{y}^*(\boldsymbol{x}_t)\| + \frac{2\tau}{c}M_0(J)L_0  \\
& \overset{(iii)}{\le} 2ML_{11}(J)\bigg(1 - \frac{\mu}{L_G}\bigg)^{\frac{K}{2}}+ \frac{2\tau}{c}M_0(J)L_0  \overset{(iv)}{\le} \frac{\sigma^2}{9H\bar{l}(J)}  \nonumber
\end{align}
where $(i)$ follows from Lemma \ref{lemma:5}, the fact that $\boldsymbol{y}_K(\boldsymbol{x}_t)\in \mathcal{Y}_J(\boldsymbol{x}_t)$, and the definition of $f_J(\boldsymbol{x})$ in (13), $(ii)$ follows from Lemma \ref{lemma:6}, Assumption \ref{assum:3}, and Lemma \ref{lemma:4}, $(iii)$ follows from Assumption \ref{assum:1}, $(iv)$ follows from the choice of $K$ and $\tau$ and the definiction of $\bar{l}(J)$ in Theorem \ref{tho:3}.

Combining (\ref{eq75}) with (\ref{eq76}), we have 
\begin{equation*}
\bigg(\boldsymbol{g}_t - \frac{\sigma}{3}\bigg)^2 \le \bar{l}(J)(\varphi_{\epsilon, \tau, J}(\boldsymbol{x}_t) - \varphi_{\epsilon, \tau, J}(\boldsymbol{x}_{t+1})) + \frac{2\sigma^2}{9}
\end{equation*}
and thus
\begin{equation*}
\frac{1}{T}\sum\limits_{t=0}^{T-1}\bigg(\boldsymbol{g}_t - \frac{\sigma}{3}\bigg)^2 \le \frac{\bar{l}(J)}{T}\bigg(\varphi_{\epsilon, \tau, J}(\boldsymbol{x}_0) - \min\limits_{\boldsymbol{x}\in \mathcal{X}}\varphi_{\epsilon, \tau, J}(\boldsymbol{x})\bigg) + \frac{2\sigma^2}{9} \overset{(i)}{\le}\frac{4\sigma^2}{9}
\end{equation*}
where $(i)$ follows from the choice of $T$ in Theorem \ref{tho:3} and the fact that $|\varphi_{\epsilon, \tau, J}(\boldsymbol{x})| \le M_3$ for any $\boldsymbol{x}\in \mathcal{X}$(see Lemma \ref{lemma:8}).

Therefore, there exists $t_1 \in \{0,\ldots, T-1\}$ such that 
\begin{equation*}
\bigg(\boldsymbol{g}_{t_1} - \frac{\sigma}{3}\bigg)^2 \le \frac{4\sigma^2}{9}
\end{equation*}
and as a result
\begin{equation*}
\langle \nabla_{\boldsymbol{x}}G_{\epsilon, \tau, J}(\boldsymbol{x}_{t_1}, \boldsymbol{y}_K(\boldsymbol{x}_{t_1})), \boldsymbol{x}- \boldsymbol{x}_{t_1}\rangle \ge -\sigma, \qquad \forall \boldsymbol{x}\in \mathcal{X}
\end{equation*}
which can be obtained by using the definition of $\boldsymbol{g}_{t_1}$ in (\ref{eq72}).
\end{enumerate}

Based on the above discussion, we know that there always exists $t_1 \in \{0, \ldots, T-1\}$ such that 
\begin{equation*}
\langle  \nabla_{\boldsymbol{x}} G_{\epsilon, \tau, J}(\boldsymbol{x}_{t_1}, \boldsymbol{y}_K(\boldsymbol{x}_{t_1})), \boldsymbol{x}- \boldsymbol{x}_{t_1}\rangle \ge -\sigma, \qquad \forall \boldsymbol{x}\in \mathcal{X}.
\end{equation*}
Furthermore for the above $t_1$, we also have $\|\nabla_{\boldsymbol{y}} G_{\epsilon, \tau, J}(\boldsymbol{x}_{t_1}, \boldsymbol{y}_K(\boldsymbol{x}_{t_1}))\| \le \sigma$, which follows from (\ref{eq64}). Then, $(\boldsymbol{x}_{t_1}, \boldsymbol{y}_K(\boldsymbol{x}_{t_1}))$ is a $\sigma$-FNE point of problem (\ref{eq15})(see the definition of $\sigma$-FNE point in Definition \ref{def:3}). The proof is complete.
\end{proof}

\subsection{Proof of Theorem \ref{tho:4}}
In the following, we prove Theorem \ref{tho:4}.

\begin{proof}
Since $(\boldsymbol{x}_l, \boldsymbol{y}_l)$ is the $\sigma_l$-FNE point of problem (\ref{eq15}) with $\tau = \tau_l$, $J=J_l$, from Definition \ref{def:3}, we  know that for any $\boldsymbol{x}\in \mathcal{X}$,
\begin{equation*}
\langle \nabla_{\boldsymbol{x}}G_{\epsilon, \tau_l, J_l}(\boldsymbol{x}_l, \boldsymbol{y}_l), \boldsymbol{x}-\boldsymbol{x}_l\rangle  \ge -\sigma_l, \qquad 
\|\nabla_{\boldsymbol{y}}G_{\epsilon, \tau_l, J_l}(\boldsymbol{x}_l, \boldsymbol{y}_l)\|  \le \sigma_l.
\end{equation*}
That is to say, for any $\boldsymbol{x}\in \mathcal{X}$, we have 
\begin{align}  \label{eq77}  \tag{77}
&\bigg\langle \nabla_{\boldsymbol{x}} F(\boldsymbol{x}_l, \boldsymbol{y}_l) + \frac{\tau_l}{f_{J_l}(\boldsymbol{x}_l) + \epsilon - f(\boldsymbol{x}_l, \boldsymbol{y}_l)}(\nabla f_{J_l}(\boldsymbol{x}_l) - \nabla_{\boldsymbol{x}}f(\boldsymbol{x}_l, \boldsymbol{y}_l)), \boldsymbol{x}- \boldsymbol{x}_l \bigg\rangle \ge -\sigma_{l},\\  \nonumber
& \bigg\|\nabla_{\boldsymbol{y}} F(\boldsymbol{x}_l, \boldsymbol{y}_l) - \frac{\tau_l}{f_{J_l}(\boldsymbol{x}_l) + \epsilon - f(\boldsymbol{x}_l, \boldsymbol{y}_l)}\nabla_{\boldsymbol{y}}f(\boldsymbol{x}_l, \boldsymbol{y}_l)\bigg\| \le \sigma_l
\end{align}
where we use the definition of $G_{\epsilon, \tau, J}(\boldsymbol{x}, \boldsymbol{y})$ in (\ref{eq16}).

In the following, we first show that the following two terms 
\begin{equation*}
\frac{\tau_l}{f_{J_l}(\boldsymbol{x}_l) + \epsilon - f(\boldsymbol{x}_l, \boldsymbol{y}_l)} (\nabla f_{J_l}(\boldsymbol{x}_l) - \nabla_{\boldsymbol{x}}f(\boldsymbol{x}_l, \boldsymbol{y}_l)), \qquad \frac{\tau_l}{f_{J_l}(\boldsymbol{x}_l) + \epsilon - f(\boldsymbol{x}_l, \boldsymbol{y}_l)} \nabla_{\boldsymbol{y}}f(\boldsymbol{x}_l, \boldsymbol{y}_l)
\end{equation*}
converge to zero as $l$ tends to infinity. 

In fact, On one hand, 
\begin{align*}
& \bigg\|\frac{\tau_l}{f_{J_l}(\boldsymbol{x}_l) + \epsilon - f(\boldsymbol{x}_l, \boldsymbol{y}_l)} (\nabla f_{J_l}(\boldsymbol{x}_l) - \nabla_{\boldsymbol{x}}f(\boldsymbol{x}_l, \boldsymbol{y}_l))\bigg\|   \\
& \overset{(i)}{\le} \frac{2\tau_l}{c}\bigg(\|\nabla_{\boldsymbol{x}}f(\boldsymbol{x}_l, \boldsymbol{y}_{J_l}(\boldsymbol{x}_l)) - \nabla_{\boldsymbol{x}} f(\boldsymbol{x}_l, \boldsymbol{y}_l)\| + \|(\nabla \boldsymbol{y}_{J_l}(\boldsymbol{x}_l))^\top \nabla_{\boldsymbol{y}} f(\boldsymbol{x}_l, \boldsymbol{y}_{J_l}(\boldsymbol{x}_l))\|\bigg) \\
& \overset{(ii)}{\le} \frac{2\tau_l}{c}\bigg(L_1 \|\boldsymbol{y}_{J_l}(\boldsymbol{x}_l) - \boldsymbol{y}_l\| + L_0 M_0(J_l)\bigg) \overset{(iii)}{\le}\frac{4L_1M}{c}\tau_l +\frac{2L_0}{c}J_l \tau_l \overset{(iv)}{\le} \frac{2}{c}\sigma_{l}^2(2L_1M + L_0)
\end{align*}
where $(i)$ uses the fact that $\boldsymbol{y}_l \in \mathcal{Y}_{J_l}(\boldsymbol{x}_l)$, the definition of $\mathcal{Y}_{J_l}(\boldsymbol{x}_l)$ and the definition of $f_{J_l}(\boldsymbol{x}_l)$ in (\ref{eq13}), $(ii)$ follows from Assumption \ref{assum:3} and Lemma \ref{lemma:4}, $(iii)$ follows from Assumption \ref{assum:1} and the definition of $M_0(J)$ in Lemma \ref{lemma:4}, $(iv)$ follows from the choice of $\tau_l$ and $c$ in Assumption \ref{assum:7}. Then, we have 
\begin{equation} \label{eq78} \tag{78}
\frac{\tau_l}{f_{J_l}(\boldsymbol{x}_l) + \epsilon - f(\boldsymbol{x}_l, \boldsymbol{y}_l)}(\nabla f_{J_l}(\boldsymbol{x}_l) - \nabla_{\boldsymbol{x}}f(\boldsymbol{x}_l, \boldsymbol{y}_l)) \rightarrow 0  
\end{equation} 
as $l\rightarrow \infty$ since $\sigma_l \rightarrow 0$. 

On the other hand, 
\begin{equation*}
\bigg\|\frac{\tau_l}{f_{J_l}(\boldsymbol{x}_l) + \epsilon - f(\boldsymbol{x}_l, \boldsymbol{y}_l)} \nabla_{\boldsymbol{y}} f(\boldsymbol{x}_l, \boldsymbol{y}_l)\bigg\| \overset{(i)}{\le} \frac{2\tau_l}{c} L_0 \overset{(ii)}{\le} \frac{2L_0}{c}\sigma_l^2
\end{equation*}
where $(i)$ follows from the fact that $\boldsymbol{y}_l \in \mathcal{Y}_{J_l}(\boldsymbol{x}_l)$, and Assumption \ref{assum:3}, $(ii)$ follows from the choice of $\tau_l$ in Theorem \ref{tho:4} and the fact that $c\le L_0H$(see Assumption \ref{assum:7}). Therefore, we have 
\begin{equation}  \label{eq79}  \tag{79}
\frac{\tau_l}{f_{J_l}(\boldsymbol{x}_l) + \epsilon - f(\boldsymbol{x}_l, \boldsymbol{y}_l)} \nabla_{\boldsymbol{y}}f(\boldsymbol{x}_l, \boldsymbol{y}_l) \rightarrow 0
\end{equation}
as $l\rightarrow \infty$ since $\sigma_l \rightarrow 0$.

Let $(\bar{\boldsymbol{x}}, \bar{\boldsymbol{y}})$ be the accumulation point of sequence $\{(\boldsymbol{x}_l, \boldsymbol{y}_l)\}$ with $\bar{\boldsymbol{x}}$ being the interior point of $\mathcal{X}$. Without loss of generality, we let $\boldsymbol{x}_l \rightarrow \bar{\boldsymbol{x}}$, $\boldsymbol{y}_l \rightarrow \bar{\boldsymbol{y}}$ as $l$ tends to infinity. In the following, we prove that $\bar{\boldsymbol{x}}$ is a stationary point of problem PBP$\epsilon$.

Firstly, combining (\ref{eq77}), (\ref{eq78}), with (\ref{eq79}), for any $\boldsymbol{x}\in \mathcal{X}$, we have 
\begin{equation}  \label{eq80} \tag{80}
\langle \nabla_{\boldsymbol{x}} F(\bar{\boldsymbol{x}}, \bar{\boldsymbol{y}}), \boldsymbol{x}- \bar{\boldsymbol{x}}\rangle \ge 0 
\end{equation} 
and 
\begin{equation}   \label{eq81}  \tag{81}
\nabla_{\boldsymbol{y}} F(\bar{\boldsymbol{x}}, \bar{\boldsymbol{y}}) = 0
\end{equation}
which can be obtained by letting $l$ in (\ref{eq77}) tends to infinity. Since $\bar{\boldsymbol{x}}$ is an interior point of $\mathcal{X}$, from (\ref{eq80}), we have 
\begin{equation}  \label{eq82}  \tag{82}
\nabla_{\boldsymbol{x}}F(\bar{\boldsymbol{x}}, \bar{\boldsymbol{y}}) = 0.
\end{equation}

Then, we prove that $\bar{\boldsymbol{y}} \in \mathcal{R}_{\epsilon}(\bar{\boldsymbol{x}})$. From (\ref{eq81}), we know that $\bar{\boldsymbol{y}} =  \operatorname {arg\,max}_{\boldsymbol{y}\in \mathcal{Y}} \, F(\bar{\boldsymbol{x}}, \boldsymbol{y})$ since $F(\bar{\boldsymbol{x}}, \cdot)$ is strongly concave w.r.t. $\boldsymbol{y}$ on $\mathcal{Y}$(see Assumption \ref{assum:2}). Thus, to prove that $\bar{y}\in \mathcal{R}_{\epsilon}(\bar{\boldsymbol{x}})$, we only need to prove that $\bar{\boldsymbol{y}}\in \mathcal{S}_{\epsilon}(\bar{\boldsymbol{x}})$(see the definition of $\mathcal{R}_{\epsilon}(\bar{\boldsymbol{x}})$ in (\ref{eq24})).

Notice that for each $(\boldsymbol{x}_l, \boldsymbol{y}_l)$, we have 
\begin{equation}  \label{eq83}  \tag{83}
f(\boldsymbol{x}_l, \boldsymbol{y}_l) \le f_{J_l}(\boldsymbol{x}_l) +\epsilon
\end{equation}
since $(\boldsymbol{x}_l, \boldsymbol{y}_l)$ must be a feasible point of problem (\ref{eq15}) with $\tau = \tau_l$, $J=J_l$. 

Moreover, it can be proved that 
\begin{equation}  \label{eq84} \tag{84}
f_{J_l}(\boldsymbol{x}_l) \rightarrow f^*(\bar{\boldsymbol{x}})
\end{equation}
as $l \rightarrow \infty$. In fact, for each $l$, let $\boldsymbol{y}^* \in \operatorname {arg\,min}_{\boldsymbol{y}\in \mathcal{Y}} \, f(\boldsymbol{x}_l, \boldsymbol{y})$, we have
\begin{align*}
|f_{J_l}(\boldsymbol{x}_l) - f^*(\bar{\boldsymbol{x}})| &\le |f_{J_l}(\boldsymbol{x}_l) - f^*(\boldsymbol{x}_l)| + |f^*(\boldsymbol{x}_l) - f^*(\bar{\boldsymbol{x}})|\\
&\overset{(i)}{\le}\frac{1}{2J_l}L_1\|\boldsymbol{y}_{J_l}(\boldsymbol{x}_l) - \boldsymbol{y}^*\|^2 + |f^*(\boldsymbol{x}_l) - f^*(\bar{\boldsymbol{x}})|  \\
& \overset{(ii)}{\le}2\frac{1}{J_l}L_1 M^2 +|f^*(\boldsymbol{x}_l) - f^*(\bar{\boldsymbol{x}})|
\end{align*}
where $(i)$ follows from Lemma \ref{lemma:2}, (ii) follows from Assumption \ref{assum:1}. On one hand, for any $\epsilon_1 > 0$, there exists positive integer $l_1$ such that for any $l \ge l_1$, we have $|f^*(\boldsymbol{x}_l) - f^*(\bar{\boldsymbol{x}})| \le \epsilon_1$, which can be obtained by noticing that $f^*(\boldsymbol{x})$ is continuous on $\mathcal{X}$(see Proposition \ref{pro:1}), and using the fact that $\boldsymbol{x}_l\rightarrow \bar{\boldsymbol{x}}$. On the other hand, for any $\epsilon_1 > 0$, there exists positive integer $l_2$ such that for any $l \ge l_2$, we have $2/J_lL_1M^2 \le \epsilon_1$, which can be obtained by noticing that $J_l \rightarrow \infty$. In conclusion, for any $\epsilon_1 > 0$, there existis $\bar{l} = \max\{l_1, l_2\}$ such that for any $l \ge \bar{l}$, we have $|f_{J_l}(\boldsymbol{x}_l) - f^*(\bar{\boldsymbol{x}})|\le 2\epsilon_1$, and thus $f_{J_l}(\boldsymbol{x}_l) \rightarrow f^*(\bar{\boldsymbol{x}})$ as $l$ tends to infinity. 

Combining (\ref{eq83}) with (\ref{eq84}), and let $l$ tends to infinity, we have $f(\bar{\boldsymbol{x}}, \bar{\boldsymbol{y}}) \le f^*(\bar{\boldsymbol{x}}) + \epsilon$. By the definition of $\mathcal{S}_{\epsilon}(\bar{\boldsymbol{x}})$, we know tht $\bar{\boldsymbol{y}} \in \mathcal{S}_{\epsilon}(\bar{\boldsymbol{x}})$ and thus $\bar{\boldsymbol{y}}\in \mathcal{R}_{\epsilon}(\bar{\boldsymbol{x}})$. 

Finally, from (\ref{eq81}), (\ref{eq82}), it is easy to know that $\bar{\boldsymbol{x}}$ is a stationary point of problem PBP$\epsilon$ in (\ref{eq4}) with $\lambda_1 = 1$, $\lambda_2 = 2$, $\lambda_3 = 0$, $I = 1$, $J=1$, $r_{11}=1$, $\sigma_{11}=0$, and $\boldsymbol{y}_1 = \bar{\boldsymbol{y}}$. Then the proof is complete.
\end{proof}

\section{Experimental Details of Synthetic Perturbed Pessimistic Bilevel Problems}\label{appendix:2}

In the experiment, $\epsilon$ is set to be 0.5, $c$ in Assumption \ref{assum:7} is set to be 0.25. For Algorithms 1 and 2, for each positive integer $l$, $\tau_l$, $J_l$, $T_l$, and $K_l$ are set to be $0.999^l$, $l$, $(1/0.999)^l$ and $2l$, respectively, and stepsizes $\alpha$, $\beta$, $\eta$ are set to be 0.1 and 1e-4, and $1/(l^3 + 0.1)$, respectively.

\section{Experimental Details of Generative Adversarial Networks}

\subsection{Synthetic Data}

For the experiments on the synthetic data, the noise samples are vectors of 256 independent and identically distributed Gaussian variables with mean zero and standard deviation of 1, and the number of noise samples is 512. Furthermore, the samples are fixed during the iterations. The generator network consists of a fully connected network with 2 hidden layers of size 128 with tanh activations followed by a linear projection to 2 dimensions. The discriminator network consists of a fully connected network with 2 hidden layers of size 128 with tanh activations followed by a output layer of size 1. 

For GAN, unrolled GAN, and PVFIM, we use Adam optimizers with learning rates 1e-3 and 1e-4 to optimize the parameter $\boldsymbol{x}$ of the generator and the parameter $\boldsymbol{y}$ of the discriminator, respectively. For both GAN and unrolled GAN, the parameter $\boldsymbol{y}$ of the discriminator is updated 6 times for each given parameter $\boldsymbol{x}$ of the generator. For PVFIM, we set $\epsilon$=1e-5. Furthermore, for each positive integer $l$, we set $\tau_l$, $J_l$, $T_l$, and $K_l$ in Algorithms 1 and 2 to be $0.999^l$, $3$, $1.002^l$, and $3$, respectively.

\subsection{Real-World Data}
For the MNIST dataset, the noise samples are vectors of 100 independent and identically distributed Gaussian variables with mean zero and standard deviation of 1. We sample 4000 images from MNIST dataset for training, and the number of noise samples is 4000. Furthermore, the samples are fixed during the iterations. The generator network consists of a fully connected network with hidden layer sizes to be 128, 256, 512, 1024. The discriminator network consists of a fully connected network with hidden layer sizes to be 512, 256.

For the CIFAR10 dataset, we sample 4000 images from CIFAR10 for training, and the number of noise samples is 4000. Furthermore, the samples are fixed during the iterations. The generator network is a 4 layer deconvolutional neural network, and the discriminator network is a 4 layer convolutional neural network. The number of units for the generator is $[256, 128, 64, 3]$, and the number of units for the discriminator is $[64, 128, 256, 1]$. 

For GAN, unrolled GAN, and PVFIM, we use Adam optimizers with learning rates 2e-4 and 2e-4 to optimize the parameter $\boldsymbol{x}$ of the generator and the parameter $\boldsymbol{y}$ of the discriminator, respectively. For both GAN and unrolled GAN, the parameter $\boldsymbol{y}$ of the discriminator is iterated 2 times for each given parameter $\boldsymbol{x}$ of the generator. For PVFIM, we set $\epsilon$=1e-5. Furthermore, for each positive integer $l$, we set $\tau_l$, $J_l$, $T_l$, and $K_l$ in Algorithms 1 and 2 to be $0.999^l$, $1$, $1.43^l$, and $1$, respectively.

\bibliographystyle{elsarticle-harv} 
\bibliography{mybible}

\end{document}